\documentclass{memo-l}
\usepackage{amssymb, amstext, amscd, amsmath, amsthm}
\usepackage[all]{xy}
\usepackage{centernot}
\usepackage{mathtools}
\usepackage{stmaryrd}
\usepackage{mathabx}
\usepackage{tikz}

\usetikzlibrary{positioning}

\newtheorem{theorem}{Theorem}[chapter]
\newtheorem{proposition}[theorem]{Proposition}
\newtheorem{lemma}[theorem]{Lemma}
\newtheorem{corollary}[theorem]{Corollary}

\newtheorem{problem}{Problem}

\theoremstyle{definition}
\newtheorem{definition}[theorem]{Definition}
\newtheorem{example}[theorem]{Example}

\theoremstyle{remark}
\newtheorem{remark}[theorem]{Remark}

\numberwithin{section}{chapter}
\numberwithin{equation}{chapter}

\newcommand{\bbA}{{\mathbb{A}}}
\newcommand{\bbB}{{\mathbb{B}}}
\newcommand{\bbC}{{\mathbb{C}}}
\newcommand{\bbD}{{\mathbb{D}}}
\newcommand{\bbI}{{\mathbb{I}}}
\newcommand{\bbF}{{\mathbb{F}}}

\newcommand{\bbN}{{\mathbb{N}}}

\newcommand{\bbR}{{\mathbb{R}}}
\newcommand{\bbT}{{\mathbb{T}}}
\newcommand{\bbZ}{{\mathbb{Z}}}

\newcommand{\A}{{\mathcal{A}}}
\newcommand{\B}{{\mathcal{B}}}
\newcommand{\C}{{\mathcal{C}}}
\newcommand{\D}{{\mathcal{D}}}

\newcommand{\F}{{\mathcal{F}}}
\newcommand{\G}{{\mathcal{G}}}
\renewcommand{\H}{{\mathcal{H}}}
\newcommand{\I}{{\mathcal{I}}}
\newcommand{\J}{{\mathcal{J}}}
\newcommand{\K}{{\mathcal{K}}}
\renewcommand{\L}{{\mathcal{L}}}
\newcommand{\M}{{\mathcal{M}}}
\newcommand{\N}{{\mathcal{N}}}
\renewcommand{\O}{{\mathcal{O}}}

\newcommand{\R}{{\mathcal{R}}}
\renewcommand{\S}{{\mathcal{S}}}
\newcommand{\T}{{\mathcal{T}}}
\newcommand{\U}{{\mathcal{U}}}

\newcommand{\X}{{\mathcal{X}}}

\newcommand{\fA}{{\mathfrak{A}}}

\newcommand{\fJ}{{\mathfrak{J}}}

\newcommand{\fM}{{\mathfrak{M}}}

\renewcommand{\phi}{\varphi}
\newcommand{\upchi}{{\raise.35ex\hbox{\ensuremath{\chi}}}}

\newcommand{\Ad}{\operatorname{Ad}}

\newcommand{\alg}{\operatorname{alg}}
\newcommand{\Aut}{\operatorname{Aut}}
\newcommand{\card}{\operatorname{card}}

\newcommand{\id}{{\operatorname{id}}}

\newcommand{\rt}{\operatorname{rt}}

\newcommand{\Rad}{\operatorname{Rad}}

\newcommand\Span{\mathop{\rm span}}

\newcommand\supp{\mathop{\rm supp}}

\newcommand{\lt}{\operatorname{lt}}

\newcommand{\ca}{\mathrm{C}^*}
\newcommand{\cenv}{\mathrm{C}^*_{\text{env}}}
\newcommand{\cmax}{\mathrm{C}^*_{\text{max}}}

\newcommand{\phist}{u}

\newcommand{\sca}[1]{\left\langle#1\right\rangle}

\newcommand\cpr{\rtimes_{\alpha}^{r}\, {\mathcal{G}}}
\newcommand\cpf{\rtimes_{\alpha}\, {\mathcal{G}}}
\newcommand\cpu{ \check{\rtimes}_{\alpha}{\mathcal{G}}}
\newcommand\cpd{ \hat{\rtimes}_{\alpha}\, {\mathcal{G}}}

\begin{document}

\frontmatter
\title[Crossed Products]{Crossed Products of Operator Algebras}

\author[E.G. Katsoulis]{Elias~G.~Katsoulis}
\address {Department of Mathematics
\\East Carolina University\\ Greenville, NC 27858\\USA}
\email{katsoulise@ecu.edu}

\author[C. Ramsey]{Christopher~Ramsey}
\address {Department of Mathematics and Statistics
\\ MacEwan University \\ Edmonton, Alberta\\Canada}
\email{ramseyc5@macewan.ca}

\date{December 24, 2015}



\maketitle

\tableofcontents

\begin{abstract}
We study crossed products of arbitrary operator algebras by locally compact groups of completely isometric automorphisms. We develop an abstract theory that allows for generalizations of many of the fundamental results from the selfadjoint theory to our context. We complement our generic results with the detailed study of many important special cases. In particular we study crossed products of tensor algebras, triangular AF algebras and various associated C$^*$-algebras. We make contributions to the study of C$^*$-envelopes, semisimplicity, the semi-Dirichlet property, Takai duality and the Hao-Ng isomorphism problem.  We also answer questions from the pertinent literature.
\end{abstract}

\mainmatter

\chapter{Introduction} \label{intro}
In this monograph we develop a theory of crossed products that allows for a locally compact group to act on an arbitrary operator algebra, not just a $\ca$-algebra. We establish foundational results, uncover permanence properties and demonstrate important connections between our crossed product theory and various lines of current research in both the non-selfadjoint and the $\ca$-algebra theory.

The reader familiar with the non-selfadjoint literature knows well that crossed product type constructions have occupied the theory since its very beginnings. However most constructions in that theory involve the action of a semigroup which rarely happens to be a group, on an operator algebra which is usually a $\ca$-algebra. There is a good reason for this and it goes back to the early work of Arveson \cite{Arv1, ArvJ}. Arveson recognized that in order to better encode the dynamics of a homeomorphism $\sigma$ acting on a locally compact space $\X$, one should abandon group actions and instead focus on the action of $\bbZ^+$ on $C_0(\X)$ implemented by the positive iterates of $\sigma$. This initiated the study of what Peters coined as the semicrossed product $C_0(\X) \times_{\sigma} \bbZ^+$ \cite{Pet}. The study of semicrossed products by $\bbZ^+$, $\bbF^+_n$ (the free semigroup on $n$ generators) and other important semigroups has produced a steady stream of important results and continues to this day at an increasing pace and depth \cite{Arv1, ArvJ, DFK,  DavKatCr, DavKatAn, DavKatMem,  KakKatJFA1, MS, Pet, Ram2}. 

In this monograph we follow a less-travelled path: we start with an arbitrary operator algebra, preferably non-selfadjoint, and we allow a whole group to act on it. It is remarkable that there have been no systematic attempts to build a comprehensive theory around such algebras even though this class includes all crossed product $\ca$-algebras.  Admittedly, our interest in group actions on non-selfadjoint operator algebras arose reluctantly as well. Indeed, apart from certain important cases (see, e.g., \cite{Dav, MS2, R2}), the structure of automorphisms for non-selfadjoint operator algebras is not well understood. Our initial approach stemmed from an attempt to settle two open problems regarding semi-Dirichlet algebras (which we do settle using the crossed product). We soon realized that even for very ``elementary" automorphisms (gauge actions), the crossed product demonstrates a behavior that allows for significant results. 

The monograph is organized in eight chapters, including this introduction which appears as Chapter~\ref{intro}. Chapter~\ref{prel} establishes the terminology used in the monograph and contains many of the fundamental results from operator algebra theory that we require in the sequel. Most of the results contained here come from five main sources \cite{BlLM, BO, Katsura, Paulsen, Will}, with additional sources mentioned within the chapter. Chapter~\ref{prel} also contains some original material, i.e., Propositions~\ref{prop integral} and \ref{extension}, to be used in later chapters.

In Chapter~\ref{basic} we define the various crossed products appearing in the monograph. Given a $\ca$-dynamical system $(\A, \G, \alpha)$ there are two natural choices for a crossed product, the (full) crossed product $\A \cpf$ and the reduced crossed product $\A \cpr$. In the general case of an operator algebra $\A$ there are many more choices, which we call relative crossed products, depending on the various choices of a $\ca$-cover for $\A$. After a careful consideration, we single out the appropriate choice for the (full) crossed product (Definition~\ref{fulldefn}) as the relative crossed product coming from the universal $\ca$-cover $\cmax(\A)$ of $\A$. Because all relative reduced crossed products coincide (Corollary \ref{allrcoincide}), the quest for a reduced crossed product trivializes. With the appropriate definitions at hand, we can now transfer results from the selfadjoint theory to our context. For instance, in Theorem~\ref{full Raeb} we generalize to the non-selfadjoint setting a result of Raeburn~\cite{Raeb} regarding the universality of the crossed product of $\ca$-algebras. In Theorem~\ref{r=f} we show that if the locally compact group $\G$ is amenable, then all relative crossed products coincide; the proof of this result requires the theory of maximal dilations \cite{DrMc}. In Theorem~\ref{Naimark} we give a ``covariant" generalization of Naimark's Theorem on positive definite group representations. This allows us to obtain the von Neumann-type inequality of Corollary~\ref{vonNeumann}. 

Iterated crossed products play a prominent role in the selfadjoint theory. Our first task in Chapter \ref{Takai} is to explain how to make sense of an iterated crossed product within the framework of our theory. After accomplishing this, we move on to Takai duality. Indeed one of the central results of the selfadjoint theory involving iterated crossed products is the Takai Duality Theorem \cite{Takai}, which extends the Pontryagin Duality to the context of operator algebras and $\ca$-dynamical systems. In Theorem~\ref{Takai duality} we succeed in extending the Takai Duality to the context of arbitrary dynamical systems not just selfadjoint. Apart from its own interest, this extension has significant applications for the study of semisimplicity for operator algebras, as witnessed in Chapter~\ref{semis}. (See Theorem~\ref{Tsemis} and Example~\ref{Texam}.)

One of the immediate consequences of our early theory and a key ingredient in the proof of our Takai duality, is the identity
\begin{equation*} 
\cmax \big(\A \cpf\big)\simeq \cmax (\A)\cpf .
\end{equation*}
(See Theorem~\ref{cmaxthm}.)
One of the motivating questions of the monograph is the validity of the other identity
\begin{equation} \label{introident}
\cenv\big(\A \cpf\big)\simeq \cenv(\A)\cpf ,
\end{equation}
regarding the $\ca$-envelope of the crossed product. In Chapter~\ref{basic} we verify this identity in the case where $\G$ is a locally compact abelian group (Theorem~\ref{abelianenv}). In Chapter~\ref{Dirichlet Sect} we continue this investigation and in Theorem~\ref{Dirichletenv} we verify (\ref{introident}) in the case where $\A$ is Dirichlet but $\G$ arbitrary. In Chapter~\ref{Dirichlet Sect} we also present the first application of our theory. In \cite{DKDoc}, Davidson and Katsoulis made a comprehensive study of dilation theory, commutant lifting and semicrossed products, with the class of semi-Dirichlet algebras playing a central role in the theory. At the time of the writing of \cite{DKDoc}, our understanding of the abundance of semi-Dirichlet algebras was limited and the following two questions arose regarding them. Are there any semi-Dirichlet algebras which are not isometrically isomorphic to tensor algebras of $\ca$-correspondences? Are there any semi-Dirichlet algebras which are neither tensor algebras of $\ca$-correspondences nor Dirichlet algebras? In Theorem~\ref{crossedtensor} and Corollary~\ref{neither} we answer both questions in the affirmative. A key ingredient in producing these results is Theorem~\ref{generatingSemiDir} which asserts that the reduced crossed product of a semi-Dirichlet operator algebra is also semi-Dirichlet. If one wishes to study semi-Dirichlet algebras, then the crossed product is indeed an indispensable tool.

In Chapter~\ref{semis}, we uncover another permanence property in the theory of crossed product algebras. In Theorem~\ref{firstsemisimple} we show that if $\A$ is a semisimple operator algebra and $\G$ a discrete abelian group, then $\A\cpf $ is semisimple. This raises the question whether the converse is also true. It turns out that in certain cases this is indeed true but in other cases it is not.  To demonstrate this we investigate a class of operator algebras which was quite popular in the mid 90s: triangular AF algebras \cite{DavKatAdv, Don, DonH, DonHJFA, Hu, LS, Pow}. Building on the beautiful ideas of Donsig \cite{Don}, we prove Theorem~\ref{mainsemisimple} which states  that if $\A$ is a strongly maximal TAF algebra and $\G$ a discrete abelian group, then the dynamical system $(\A, \G, \alpha)$ is linking if and only if $\A \rtimes_\alpha \G$ is semisimple. In Example~\ref{linking example}, we present an example of a non-semisimple TAF algebra $\A$ that admits a linking automorphism $\alpha$. Therefore $\A \rtimes_{\alpha} \bbZ$ is semisimple even though $\A$ is not, thus refuting the converse of Theorem~\ref{firstsemisimple}. On the other hand, Theorem~\ref{TUHFsemisimple} shows that for TUHF algebras the semisimplicity of $\A$ and $\A \cpf$ are equivalent properties. We expect more in this direction, with the investigation of other dynamical systems $(\A, \G , \alpha)$ and the semisimplicity of the associated crossed products. We truly envision the study of semisimplicity (or other permanence properties) for crossed products as a theory that will parallel in interest and abundance of results that of simplicity for selfadjoint crossed products. As evidence we offer a remarkable, we believe, result which shows that for crossed products by compact abelian groups, the situation of Theorem~\ref{firstsemisimple} reverses. In Theorem~\ref{Tsemis} we show that if $\A\cpf$ is a semisimple operator algebra and $\G$ a compact abelian group, then $\A$ is semisimple. Furthermore, in Example~\ref{Texam} we show the converse is not in general true. Both these results are accomplished through the use of our non-selfadjoint Takai duality.

Chapter~\ref{Hao} makes a connection with a topic in $\ca$-algebra theory, which is currently under investigation or impacts the work of various authors, including Abadie, Bedos, Deaconu, Hao,
Kaliszewski, Katsura, Kim, Kumjian, Ng, Quigg, Schafhauser and others \cite{Ab, BKQR, DKQ, HN, KQR, Katsuracomm, Kim, Sch}. These authors are either using or currently investigating the validity of the Hao-Ng isomorphism Theorem beyond the class of amenable locally compact groups. This is a problem seemingly irrelevant to the non-selfadjoint theory as it involves the functoriality of two crossed product constructions in $\ca$-algebra theory. It is a consequence of our Theorem~\ref{HNtensor1} that the investigation of the previously mentioned authors  is intimately connected with the verification of the identity (\ref{introident}) for a very special class of non-selfadjoint dynamical systems $(\A, \G , \alpha)$, where $\A$ is the tensor algebra of a $\ca$-correspondence and $\alpha: \G \rightarrow \Aut \A$, the action of a locally compact group by generalized gauge automorphisms. Actually, Theorem~\ref{HNtensor1} leads to a recasting of the Hao-Ng Isomorphism Problem, which we verify in the case of (not necessarily injective) Hilbert bimodules (Theorem~\ref{HNBim}).

It is worth mentioning that the main focus of Chapter~\ref{Hao} is not the Hao-Ng Isomorphism problem itself but instead verifying another permanence property for the crossed product: the crossed product of a tensor algebra $\A$ by a locally compact group $\G$ of gauge automorphisms remains in the class of tensor algebras. (We have seen in Theorem~\ref{crossedtensor} that this is not the case when the group $\G$ acts by arbitrary automorphisms.) In order to obtain the affirmative answer (Theorem~\ref{HNtensor2}) we use a result of independent interest, which we label the Extension Theorem. The Extension Theorem (Theorem~\ref{extend}) gives a very broad criterion for verifying whether an operator algebra ``naturally" containing a $\ca$-correspondence $X$ is isomorphic to the tensor algebra of $X$. This a very general result with additional applications to appear elsewhere.

The monograph closes with Chapter~\ref{problems}, where we list some open problems for further investigation. With each open problem listed, we give a brief commentary intended to help the reader guide himself through the pertinent material or literature. Two of these problems concern the classification of crossed products. This a topic which is left untouched in this monograph and we plan to address it in a subsequent work.  

Beyond the specific problems of Chapter~\ref{problems}, this work also suggests two general directions for future research: one ``abstract" and another one more ``concrete". The current $\ca$-algebra literature is occupied with the study of more general concepts of a crossed product, e.g., crossed products by coactions, twisted actions and much more. This is such a broad area that we will not even be attempting to survey it here; see however \cite{Es} and the references therein. If one wishes to develop non-commatative duality and the appropriate versions of our non-selfadjoint Takai duality of Chapter~\ref{Takai}, then the abstract study of these more general crossed products at the non-selfadjoint level is of the highest priority. 

On the concrete side, the non-selfadjoint operator algebra theory is currently being infused by a wealth of very deep and far-reaching studies. In a broad sense, the current rapid development of ``free analysis" and ``free function theory" involves naturally certain non-selfadjoint operator algebras, as witnessed in the works of Muhly and Solel \cite{MSf1, MSf2}, Popescu \cite{Pop3, Pop4, Pop5, Pop6} and others.  Closer to this monograph, Shalit and Solel \cite{SS} pioneered recently the study of a new class of operator algebras, the tensor algebras of subproduct systems. These algebras  exhibit a very diverse and unexpected behavior as demonstrated in the recent papers \cite{DMar2, KS}. Even for very special cases the study of these algebras relates to topics which are very popular and quite demanding, such as the Drury-Arveson spaces and algebraic varieties associated with ideals \cite{DRS, DRS2, Har}, stochastic matrices \cite{DMar, DMar2}, subshifts \cite{KS, SS} and more. It seems to us that the tensor algebras of subproduct systems and their peripheral algebras should be the natural place to extend the theory of Chapter~\ref{Hao}. In particular it would be very interesting to see how the Hao-Ng isomorphism problem manifests itself in that context. This of course hinges on understanding what the Cuntz-Pimsner algebra of a subproduct system should be; again this is a topic of important current research \cite{KS, DMar2, Vis, Vis2}. 

Finally a word about the groups appearing in this monograph. Our main goal is to develop a comprehensive theory of crossed products that is applicable to all locally compact groups. Hence the majority of our work concerns that generality. Nevertheless many of our results are new and interesting even in the case where $\G=\bbZ$. For instance, this is the case with all (counter)examples appearing in Chapter~\ref{Dirichlet Sect} or the semisimplicity results of Chapter~\ref{semis}. A special mention needs to made for Chapter~\ref{Hao}. There we took the unusual step of ``duplicating" proofs in order to give a more elementary and self-contained treatment of the case where $\G$ is discrete. We believe that this adds to the monograph as it makes very accessible a work that bridges the selfadjoint with the non-selfadjoint theory.


\chapter{Preliminaries} \label{prel}
\section{Generalities} \label{Gen}
The term \textit{operator algebra} is understood to mean a norm closed subalgebra of the algebra of all bounded operators acting on a Hilbert space. All algebras in this monograph are assumed to be approximately unital, i.e., they possess a contractive approximate identity. All representations (and homomorphisms into multiplier algebras, whenever applicable) will be required to be non-degenerate.

On occasion we will need to exploit the richer structure of unital operator algebras.  If $\A$ is an operator algebra without a unit, let $\A^1 \equiv \A + \bbC I$. If $\phi : \A \rightarrow \B$ is a completely isometric homomorphism between non-unital operator algebras, then Meyer \cite[Corollary 3.3]{Mey} shows that $\phi$ extends to a complete isometry $\phi^1: \A^1 \rightarrow \B^1$. This shows that the unitization of $\A$ is unique up to complete isometry.

In the category of unital algebras with morphisms the completely contractive maps, the concept of a dilation of a morphism is defined as follows. Let $\A$ be a unital operator algebra and $\pi:  \A \rightarrow B(\H)$ be a completely contractive map. A \textit{dilation} $\rho:\A \rightarrow B(\K)$ for $\pi$ is a completely contractive map so that $P_{\H} \rho(.)\mid_{\H}= \pi$. A completely contractive map is called \textit{maximal} if it admits no non-trivial dilations. (Since we are within the unital category, all maps so far are either assumed or required to be unital.) Dritschel and McCullough \cite[Theorem 1.2]{DrMc} have shown that any completely contractive representation $\pi$ of an operator algebra $\A$ admits a maximal dilation $\rho$, which also happens to be multiplicative.

Given an operator algebra $\A$, a $\ca$-cover $(\C, j)$ for $\A$ consists of a $\ca$-algebra $\C$ and a completely isometric homomorphism $ j : \, \A \rightarrow \C$ with $\C = \ca(j(\A))$. 

\begin{definition} \label{firstequiv}
Two $\ca$-covers $( \C_i, j_i)$, $i=1,2$, of an operator algebra $\A$ are said to be \textit{equivalent}, denoted as $(\C_1, j_1)\simeq (\C_2 , j_2)$, provided that there exists $*$-isomorphism $j \colon \C_1 \rightarrow \C_2$ that makes
following diagram
 \begin{equation} \label{little triangle}
\xymatrix{\C_1 \ar[rd]^{j} \\
\A \ar[u]^{j_1} \ar[r]_{j_2}  & \C_2}
\end{equation}
commutative.
\end{definition} 

Given an operator algebra $\A$, the $\ca$-envelope $\cenv(\A) \equiv (\cenv(\A), j)$ is any $\ca$-cover of $\A$ satisfying the following property: for any other $\ca$-cover $(\C , i)$ of $\A$  there exists a $*$-epimorphism $\phi : \, \C \rightarrow \cenv(\A)$ so that $\phi(i(a)) = j(a)$, for all $a \in \A$. As it turns out the collection of all $\ca$-covers that qualify as the $\ca$-envelope for $\A$ forms an equivalence class under the equivalence of Definition~\ref{firstequiv}. The concept of the $\ca$-envelope plays a paramount role in abstract operator algebra theory \cite{Arvsub, Arvenv, DavKen}.

If $(\C , j)$ is a $\ca$-cover of a unital operator algebra $\A$, then there exists a largest ideal $\J \subseteq \C$, the \textit{Shilov ideal} of $\A$ in $(\C, j)$, so that the quotient map $\C \rightarrow \C \slash \J$ when restricted on $j(\A)$ is completely isometric. It turns out that $\cenv(\A) \simeq ( \C \slash \J, q\circ j)$, where $q:\C \rightarrow \C\slash \J$ is the natural quotient map. A related result asserts that if $\pi: \A \rightarrow B(\H)$ is a completely isometric representation of a unital operator algebra $\A$ and $\rho$ a maximal dilation of $\pi$, then $\big( \ca \big(\rho(\A)\big), \rho\big) \simeq \cenv(\A)$. See \cite{Arvenv, DrMc} for more details.
If $\A$ is a non-unital operator algebra then we can describe the $\ca$-envelope of $\A$ by invoking its unitization as follows: if $\cenv(\A^1)\simeq \big(\cenv(\A^1), j^1\big)$, then $\cenv(\A) \simeq (\C, j)$, where $\C \equiv \ca(j^1(\A))$ and $j\equiv {j^1}_{\mid \A}$. See the proof of~\cite[Proposition 4.3.5]{BlLM} for the precise argument. The existence of a $\ca$-envelope for a non-unital algebra $\A$ implies now the existence of a Shilov ideal for $\A$ in any $\ca$-cover $(\C, j)$. See the proof of \cite[Proposition 1.9]{Kak2}.

If $\A$ is an operator algebra then there exists a $\ca$-cover $\cmax(\A) \equiv ( \cmax(\A) , j)$ with the following universal property: if $\pi: \A \rightarrow \C$ is any completely contractive homomorphism into a $\ca$-algebra $\C$, then there exists a (necessarily unique) $*$-homomorphism $\phi: \cmax(\A) \rightarrow \C$ such that $\phi\circ j = \pi$. The cover $\cmax(\A)$ is called the maximal or universal $\ca$-algebra of $\A$. The equivalence class of this $\ca$-cover also plays a crucial role in abstract operator algebra theory~\cite{Bl1, Bl2}. See also \cite{BlLM} and the references therein for more applications of $\cmax(\A)$.

 We list a few more results regarding (approximately unital) operator algebras. The interested reader should consult the comprehensive monograph of Blecher and Le Merdy \cite{BlLM} for more details. By \cite[Lemma 2.1.7]{BlLM}, the $\ca$-cover of an approximately unital operator algebra $\A$ is actually unital only when $\A$ itself is unital. Furthermore a contractive approximate unit for $\A$ is also an approximate unit for any $\ca$-cover $\C = \ca(\A)$ of $\A$  \cite[Lemma 2.1.7]{BlLM}. 
If $\A$ is an operator algebra, then
\[
M(\A) \equiv \{ x \in \A^{**} \mid xa , ax \in \A, \mbox{ for all } a \in \A\}
\]
is the multiplier algebra of $\A$. For any completely isometric non-degenerate representation $\pi: \A \rightarrow B(\H)$, the algebra $$\{ T \in B(\H)\mid T\pi(a), \pi(a)T \in \A, \mbox{for all } a \in \A\}$$ is completely isometrically isomorphic to $M(\A)$ via an isomorphism that fixes $\A$ elementwise \cite[Proposition 2.6.8]{BlLM}. Furthermore, $M(\A) \subseteq M(\C))$ for any $\ca$-cover $\C$ of $\A$ \cite[page 87]{BlLM}. Therefore, $\A \subseteq M(\A)$ is a (two-sided) ideal, which is essential both as a left and a right ideal of $M(\C)$.

Let $\A$, $\B$ be operator algebras. A completely contractive homomorphism $\phi: \A\rightarrow M(\B)$ is said to be a multiplier-nondegenerate morphism, if both $ [ \phi(\A)\B]$ and $[ \B\phi(\A) ]$ are dense in $\B$. There are many equivalent formulations of this property based on Cohen's factorization theorem; see \cite[Section~\ref{prel}.6.11]{BlLM}. A  multiplier-nondegenerate morphism $\phi: \A \rightarrow M(\B)$ always admits a unique, unital and completely contractive extension $\overline{\phi} : M(\A) \rightarrow \M(\B)$  \cite[Proposition 2.6.12]{BlLM}; such a map is easily seen to be strictly continuous.  

Finally we need to explain how we make sense of integrals where the integrand is a function taking values in the multiplier algebra of an operator algebra. (Propositions~\ref{Raeburnthm}, \ref{classifyrepns} and Theorem~\ref{full Raeb}.) If the integrand is norm continuous, then see \cite[Lemma 1.91]{Will}. Otherwise we use the following.

\begin{proposition} \label{prop integral}
Let $\G$ be a locally compact group with left-invariant Haar measure $\mu$. Let $\A$ be an operator algebra and let $\G \ni s \mapsto f(s) \in M(\A)$ be a strictly continuous function with compact support. Then there exists a unique element $\int f(s) d\mu(s) \in M(\A)$ satisfying
\begin{equation} \label{integralsense}
\begin{split}
\Big(\int f(s) d\mu(s) \Big) a&= \int f(s) a d \mu(s)\\
a\Big(\int f(s) d\mu(s) \Big) &= \int af(s) d \mu(s), 
\end{split}
\end{equation}
for all $a \in \A$.

Furthermore, if $\B$ is an approximately unital operator algebra and $\phi : \A\rightarrow M(\B)$ is a completely contractive, multiplier-nondegenerate morphism, then
\begin{equation} \label{barL}
\overline{\phi}\Big(\int f(s)d\mu(s)\Big)= \int \overline{\phi}\big(f(s)\big) d \mu(s).
\end{equation}
\end{proposition}
\begin{proof}
If $\A$ is a $\ca$-algebra, then the existence and uniqueness of such an element follows from Lemma 1.101 in \cite{Will}. We will rely on this result in order to to explain the validity of (\ref{integralsense}) and (\ref{barL}) in general.

Let $\C$ be a $\ca$-cover for $\A$; as we noticed earlier we have $M(\A) \subseteq M(\C)$. Let $\{e_i\}_{i \in \bbI}$ be a contractive approximate identity for $\A$ (and therefore for $\C$ as well). For any $c \in \C$, the functions $G \ni s \mapsto f(s)c \in \C$ and $s  \mapsto cf(s) \in \C$ can be uniformly approximated by the norm continuous functions $s \mapsto f(s)e_ic $, $i \in \bbI$, and $s  \mapsto ce_i f(s) $, $i \in \bbI$, respectively and so they are norm continuous. Hence $s\mapsto f(s)$ is strictly continuous in $\M(\C)$. Lemma 1.101 in \cite{Will} implies the existence of an element $\int f(s)d \mu(s) \in M(\C)$ so that 
\begin{equation*}
\begin{split}
\Big(\int f(s) d\mu(s) \Big) c&= \int f(s) c d \mu(s)\\
c\Big(\int f(s) d\mu(s) \Big) &= \int cf(s) d \mu(s), 
\end{split}
\end{equation*}
for all $c \in \C$. However, the above equations show that for any $a \in \A$ both $\big(\int f(s) d\mu(s) \big) a$ and $a\big(\int f(s) d\mu(s) \big)$ are in $\A$ and so $\int f(s) d\mu(s) \in M(\A)$.

In order to establish (\ref{barL}), assume that $\phi: \A \rightarrow M(\B)$ is a multiplier- nondegenerate morphism, i.e., $ [ \phi(\A)\B]$ and $[ \B\phi(\A) ]$ are dense in $\B$. Since $\B$ is also approximately unital, both integrals in (\ref{barL}) are well-defined. Therefore, for arbitrary $a \in \A$, $b \in \B$, we have
\begin{equation*}
\begin{split}
\overline{\phi}\Big(\int f(s)d \mu (s)&\Big) \phi (a)b =  \phi\Big( \int f(s)  d \mu (s) a\Big)b =  \phi\Big( \int f(s) a d \mu (s) \Big)b\\
&= \Big(\int \phi\big( f(s)a\big) d \mu (s)\Big) b = \Big(\int \overline{\phi}\big( f(s) \big) \phi(a) d \mu (s)\Big) b\\
& =\Big(\int \overline{\phi}\big( f(s) \big)d \mu (s)\Big)  \phi(a) b .
\end{split}
\end{equation*}
A similar argument establishes 
\[
 b\phi (a)\overline{\phi}\Big(\int f(s)d \mu (s)\Big) =b \phi(a) \Big(\int \overline{\phi}\big( f(s) \big)d \mu (s)\Big).
\]
Since $\B\subseteq M(\B)$ is an essential ideal, the conclusion follows.
\end{proof}

\begin{remark} \label{multiplier nondeg}
If $\phi: \A \rightarrow B(\H)$ is a contractive, non-degenerate representation, then it can also be viewed as a morphism $\phi: \A \rightarrow M\big( \K(\H)\big)$, where $\K(\H)$ denotes the compact operators. Since $\A$ is approximately unital, then it follows that $\phi: \A \rightarrow M\big( \K(\H)\big)$ is also a multiplier-nondegenerate morphism and so (\ref{barL}) is applicable for such a $\phi$.

To see the multiplier-nondegeneracy of $\phi$, let $\{e_i\}_{i \in \bbI}$ be a contractive approximate identity for $\A$. The non-degeneracy of $\phi$ implies that $\{\phi (e_i)\}_{i \in \bbI}$ converges strongly to the identity $I \in B(\H)$. Hence for an $k \in \K(\H)$, we have $\lim_i \phi(e_i)k = k$ in norm and so \cite[Lemma 2.1.6]{BlLM} implies $\lim_i k^*\phi(e_i) = k^*$. Therefore $[ \K(\H)\phi(\A)] \subseteq \K (\H)$ is dense. The density of  $[ \phi (\A) \K(\H)]$ in $\K(\H)$ is elementary to verify.
\end{remark}

\section{$\ca$-correspondences and tensor algebras}
A $\ca$- correspondence $(X,\C,\phi_X)$ consists of a $\ca$-algebra $\C$, a Hilbert $\C$-module $(X, \sca{\phantom{,},\phantom{,}})$ and a
(non-degenerate) $*$-homomorphism $\phi_X\colon \C \rightarrow \L(X)$ into the adjointable operators on $X$.

Two $\ca$-correspondences $(X, \C, \phi_X)$ and $(Y, \D, \phi_Y)$ are said to be \textit{unitarily equivalent} if there exist a $*$-isomorphism $\sigma: \C \rightarrow \D$ and a linear surjection $W: X \rightarrow Y$ so that:
\begin{itemize}
\item[(i)] $W(x c)= (W x)\sigma(c)$ and $W(\phi_X(c)x) = \phi_Y(\sigma(c))Wx$, 
\item[(ii)] $\sca{Wx, W x'} = \sigma\big( \sca{x, x'}\big)$,
\end{itemize}
for all $c \in \C$ and $x,x'\in X$. In that case we say that the pair $(W, \sigma)$ implements the unitary equivalence.

An isometric (Toeplitz) representation $(\rho,t)$ of a $\ca$-correspondence into
a $\ca$-algebra $\D$, is a pair consisting of  a $*$-homomorphism $\rho\colon \C \rightarrow \D$ and a linear map $t\colon X \rightarrow D$, such that
\begin{enumerate}
 \item $\rho(c)t(x)=t(\phi_X(c)(x))$,
 \item $t(x)^*t(x')=\rho(\sca{x,x'})$,
\end{enumerate}
for all $c \in \C$ and $x,x'\in X$. A representation $(\rho , t)$ is said to be \textit{injective} iff $\rho$ is injective; in that case $t$ is an isometry.

The $\ca$-algebra generated by a representation $(\rho,t)$ equals the closed linear span of $t^n(\bar{x})t^m(\bar{y})^*$, where for simplicity $\bar{x}\equiv (x_1,\dots,x_n)\in X^n$ and $t^n(\bar{x})\equiv t(x_1)\dots t(x_n)$.
For any
representation $(\rho,t)$ there exists a $*$-homomorphism
$\psi_t:\K(X)\rightarrow B$, such that $\psi_t(\theta^X_{x,y})=
t(x)t(y)^*$.

It is easy to see that for a $\ca$-correspondence $(X,\C,\phi_X)$ there exists a universal Toeplitz representation, denoted as $(\rho_{\infty} , t_{\infty})$, so that any other representation of $(X,\C,\phi_X)$ is equivalent to a direct sum of sub-representations of $(\rho_{\infty} , t_{\infty})$. We define the Cuntz-Pimsner-Toeplitz $\ca$-algebra $\T_X$ as the $\ca$-algebra generated by all elements of the form $\rho_{\infty}(c), t_{\infty}(x)$, $c \in \C$, $x \in X$. By the universality of $(\rho_{\infty} , t_{\infty})$ the algebra $\T_X$ satisfies the following property: for any Toeplitz representation $(\rho,t)$ of $X$, there exists a representation $\rho \rtimes t$ of $\T_X$, called \textit{the integrated form of}  $(\rho,t)$, so that $\rho(c) = \big((\rho \rtimes t)\circ \rho_{\infty}\big)(c)$, for all $c \in \C$, and $t(x) = \big((\rho \rtimes t)\circ t_{\infty}\big)(x)$, for all $ x \in X$. (See \cite[Definition 3.1]{KatsuraJFA} and the subsequent comments.)

We say that a Toeplitz representation $(\rho , t)$ admits a gauge action if there exists a family $\{ \gamma_z\}_{z \in \bbT}$ of $*$-endomorphisms of $\ca\big( (\rho \rtimes t)(\T_X)\big)$ so that 
\[
\gamma_z(\rho(c)) = \rho(c),  \mbox{ for all } c \in \C, \quad \gamma_z(t(x))=z t(x), \mbox{ for all } x \in X.
\]
The following result of Katsura \cite[Theorem 6.2]{KatsuraJFA} gives an easy to use criterion for verifying that a Toeplitz representation $(\rho, t)$ \textit{integrates} to a faithful representation $\rho\rtimes t$ of $\T_X$.

\begin{theorem} [Gauge Invariant Uniqueness Theorem] \label{GIUThm}

Let $(X, \C, \phi_X)$ be a $\ca$-correspondence and let $(\rho, t)$ a Toeplitz representation of $(X , \C, \phi_X)$ that admits a gauge action and satisfies
\begin{equation} \label{Icriterion}
I_{(\rho, t)}' \equiv \{ c \in \C \mid \rho(c) \in  \psi_t (\K(X)) \} = \{0\}.
\end{equation}
Then $\rho\rtimes t$ is a faithful representation of $\T_X$.
\end{theorem}

It is worth giving an example of a concrete Toeplitz representation of $(X, \C, \phi_X)$ that will help make some of the results that follow more transparent. For that we need to define first the concept of the stabilized tensor product between $\ca$-correspondences.

Let $(X, \C, \phi_X)$ and $(Y, \C, \phi_Y)$ be $\ca$-correspondences. The \emph{interior} or \emph{stabilized tensor product}, denoted by $X \otimes Y$, is the quotient of the vector space tensor product $X \otimes_{\alg} Y$ by the subspace generated by the elements of the form
\begin{align*}
x c \otimes y - x \otimes \phi_Y(c) y, \quad x\in X, y\in Y, c
\in \C.
\end{align*}
It becomes a pre-Hilbert $\C$-module when equipped with
\begin{align*}
 (x \otimes y)c&\equiv x \otimes (y c ),\\
\sca{x_1\otimes y_1, x_2\otimes y_2}&\equiv \sca{y_1,
\phi_Y(\sca{x_1,x_2})y_2},
\end{align*}
where $x,x_1,x_2 \in X$, $y, y_1, y_2 \in Y$ and $c \in \C$. For $S\in \L(X)$ we define $S\otimes \id_Y \in \L(X\otimes Y)$ as the mapping $$x \otimes y \mapsto S(x)\otimes y, \quad x \in X, y \in Y.$$
Hence $X \otimes Y$ becomes a $\ca$-correspondence by defining $\phi_{X\otimes Y}(c) \equiv \phi_X(c) \otimes \id_Y$, $ c \in \C$.

The \emph{Fock space} $\F_{X}$ over the correspondence $(X, \C, \phi_X)$ is the interior direct sum of the $X^{\otimes n}\equiv X^{\otimes n-1} \otimes X$, $n \in \bbN$, with the structure of a direct sum of $\ca$-correspondences over $\C$,
\[
\F_{X}= \C \oplus X \oplus X^{\otimes 2} \oplus \dots .
\]
Given $x \in X$, the (left) creation operator $t'_{\infty}(x) \in \L(\F_{X})$ is defined as
\[
t'_{\infty}(x)( c, \zeta_{1}, \zeta_{2}, \dots ) = (0, x c, x
\otimes \zeta_1, x \otimes \zeta_2, \dots),
\]
where $c \in \C$ and $\zeta_n \in X^{\otimes n}$, for all $n \in \bbN$.
 For any $c \in \C$, we define
 $$\rho'_{\infty}(c) = L_c \oplus \phi_X(c) \oplus (\oplus_{n=1}^{\infty} \phi_X(c) \otimes \id_{n}).$$
 It is easy to verify that $( \rho'_{\infty}, t'_{\infty})$ is a Toeplitz representation of $(X, \C)$ which is called the \emph{Fock representation} of $(X, \C)$. It follows from Theorem~\ref{GIUThm} that the representation $\rho'_{\infty}\rtimes t'_{\infty}  \colon \T_X \rightarrow \L(\F_{X})$ is actually faithful.

Given a $\ca$-correspondence $(X, \C, \phi_X)$, there is a natural non-selfadjoint subalgebra of $\T_X$ that plays an important role in this monograph.

\begin{definition}
The \emph{tensor algebra} $\T_{X}^+$ of a $\ca$-correspondence
$(X,\C,\phi_X)$ is the norm-closed subalgebra of $\T_X$ generated by
all elements of the form $\rho_{\infty}(c), t_{\infty}(x)$, $c \in \C$, $x \in X$.
\end{definition}

It is worth mentioning here that $\T^+_{X}$ also sits naturally inside the Cuntz-Pimsner algebra $\O_{X}$ associated with the $\ca$-correspondence $X$. This follows from work in \cite{FMR, KatsoulisKribsJFA, MS} which we now describe.

If $(X, \C, \phi_X)$ is a $\ca$-correspondence, then let $$J_X\equiv \ker\phi_X^{\perp}\cap \phi_X^{-1}(\K(X)).$$ A representation $(\rho, t)$ of $(X, \C,\phi_X)$ is said to be \textit{covariant} iff $\psi_t ( \phi_X(c)) = \rho (c)$, for all $c \in J_X$. The universal $\ca$-algebra for ``all" covariant representations of $(X, \C ,\phi_X)$ is the Cuntz-Pimsner algebra $\O_X$. The algebra $\O_X$ contains (a faithful copy of) $\C$ and (a unitarily equivalent) copy of $X$. Katsoulis and Kribs \cite[Lemma 3.5]{KatsoulisKribsJFA} have shown that the non-selfadjoint algebra of $\O_X$ generated by these copies of $\C$ and $X$ is completely isometrically isomorphic to $\T_X^+$. Furthermore, $\cenv(\T_X^+)\simeq \O_X$. See \cite{KatsoulisKribsJFA, MS} for more details.

The tensor algebras for $\ca$-correspondences were pioneered by Muhly and Solel in \cite{MS}. They form a broad class of non-selfadjoint operator algebras which includes as special cases Peters' semicrossed products \cite{Pet}, Popescu's non-commutative disc algebras \cite{Pop}, the tensor algebras of graphs (introduced in \cite{MS} and further studied in \cite{KaKr}) and the tensor algebras for multivariable dynamics \cite{DavKatMem}, to mention but a few.

Due to its universality, the Cuntz-Pimsner-Toeplitz $\ca$-algebra $\T_X$ admits a gauge action $\{ \psi_z\}_{z \in \bbT}$ that leaves $\rho_{\infty}(\C)$ elementwise invariant and ``twists" each  $t_{\infty}(x)$, $x \in X$, by a unimodular scalar $z \in \bbT$, that is $\psi_z(t_{\infty}(x))= z t_{\infty}(x)$, $x \in X$. Using this action, and reiterating a familiar trick with the Fejer kernel, one can verify that each element $a \in \T_X^+$ admits a Fourier series expansion
\begin{equation} \label{Cesaro}
a = \rho_{\infty}(c) +\sum_{n=1}^{\infty} \,  t_{\infty}(x_n), \quad c \in \C, \, x_n \in X^{\otimes n}, n=1,2, \dots,
\end{equation}
where the summability is in the Cesaro sense.

One of the immediate consequences of (\ref{Cesaro}) is that the diagonal of $\T_X^+$ equals $\C$, i.e., $\T_X^+ \cap (\T_X^+)^*= \rho_{\infty}(\C)$. Another consequence now follows.

 If $(X, \C,\phi_X)$ is a $\ca$-correspondence and $\rho$ a bounded multiplicative linear functional on $\C$, then $\fM_{\rho}$ will denote the collection of all bounded multiplicative linear functionals on $\T_X^+$, whose restriction on $\C$ agrees with $\rho$.

\begin{proposition} \label{extension}
Let $(X, \C,\phi_X)$ be a $\ca$-correspondence and $\rho$ is a  bounded multiplicative linear functional on $\C$. If $\fM_{\rho}$ is as above, then $\fM_{\rho}$ is either a singleton or it is at least the size of the continuum.
\end{proposition}

\begin{proof}
Due to the gauge action $\{ \psi_z\}_{z \in \bbT}$ discussed above, $\T_X$ admits an expectation
\begin{equation} \label{condexpO}
\Phi: \T_X \longrightarrow \T_{X}^{\operatorname{fix}}: a \longmapsto \frac{1}{2 \pi}\int \psi_t(a) dt
\end{equation}
onto the fixed point algebra of $\{ \psi_z\}_{z \in \bbT}$. When restricted on $\T_X^+$, the expectation $\Phi$ is multiplicative and projects onto $\rho_{\infty}(\C)$.

If $\rho$ is a  bounded multiplicative linear functional on $\C$, then $\rho\circ \Phi \in \fM_{\rho}$. Hence $\fM_{\rho} \neq \emptyset$. If $\rho_1, \rho_2 \in \fM_{\rho}$ are distinct, then at least one of them, say $\rho_1$, does not annihilate $X$. But then, $\rho_1 \circ\ \psi_z$, $z \in \bbT$, are all distinct elements of $\fM_{\rho}$ and the conclusion follows.
\end{proof}

\section{Crossed products of $\ca$-algebras} \label{abelian}

The crossed product of an operator algebra will be formally defined in the next chapter. Nevertheless we collect here various known results regarding crossed products of $\ca$-algebras to be used throughout the monograph. Our main references are \cite{BO, Will}; we follow closely \cite{Will} in terms of notation.

Let $\G$ be a discrete amenable group, let $\C$ be a $\ca$-algebra and let $\alpha: \G \rightarrow \Aut \C$ be a representation. Since $\G$ is amenable, both the full crossed product $\C \rtimes_{\alpha} \G$ and the reduced $\C \rtimes_{\alpha}^r \G$ coincide. On $\C \rtimes_{\alpha} \G$ there is a well-defined faithful expectation $\Phi_e$ projecting on $\C \subseteq \C \rtimes_{\alpha} \G$, which satisfies
  \[
  \Phi_e\big(\sum _{g\in \G}  c_gU_g\big)= c_e
  \]
  for any finite sum of the form $ \sum _{g\in \G}  c_gU_g$, where $U_g$ are the universal unitaries in the multiplier algebra $M(\C \rtimes_{\alpha} \G)$ implementing the action of $\alpha_g$, $g \in \G$. 
  
  If $S \in \C \rtimes_{\alpha} \G$, then the Fourier coefficients $\{ \Phi_g (S)\}_{g \in \G}$ of $S$ are defined by the formula $\Phi_g (S)\equiv \Phi_e (SU_g^*)$, $g \in \G$. It is easy to see that if $\{ S_n\}_n$ is a sequence of polynomials in $ \C \rtimes_{\alpha} \G$ converging to $S$, then $\lim_n \Phi_g (S_n) = \Phi_g(S)$, $\forall g \in \G$.
  
  Since the group $\G$ is amenable, it contains a \textit{Folner net}, i.e., a net $\{ F_i\}_{i \in \bbI}$ of finite subsets of $\G$ so that
  \[
 \lim_{i \in \bbI} \frac{|gF_i \cap F_i |}{|F_i|} = 1, \quad \forall g \in \G.
  \]
This allows us to deduce a Cesaro type approximation for any $  S \in \C \rtimes_{\alpha} \G$ using polynomials with coefficients ranging over $\{ \Phi_g (S)\}_{g \in \G}$.

\begin{proposition} \label{Ceasaroapprox}
Let $(\C, \G , \alpha)$ be as above and let $ S \in \C \rtimes_{\alpha} \G$. Then given $\epsilon > 0$ there exists a finite set $F_{\epsilon} \subseteq \G$ so that
\[
\Big \| S - \sum_{g \in \G} \frac{|gF_{\epsilon} \cap F_{\epsilon}|}{|F_{\epsilon}|} \Phi_g (S) U_g\Big\| \leq \epsilon.
\]
In particular, if $\Phi_g (S) =0$, $\forall g \in \G$, then $S=0$.
\end{proposition} 

\begin{proof}
By  \cite[Lemma 4.2.3]{BO}, for any finite set $F \subseteq \G$, the map
\begin{equation} \label{Ceasarosubst}
c_g U_g \longmapsto \frac{|gF \cap F|}{|F|}  c_gU_g, \quad c_g \in \C, g \in \G
\end{equation}
extends to a completely contractive map $\Psi_F$ on $\C \rtimes_{\alpha} \G$. 
If $S \in \C \rtimes_{\alpha} \G$, then the net $\{ \Psi_{F_i}(S)\}_{i \in \bbI}$ converges to $S$, where $\{ F_i \}_{i \in \bbI}$ is a Folner net for $\G$. Choose $F_{\epsilon}$ so that $\| S - \Psi_{F_{\epsilon}}(S)\| \leq \epsilon$. The conclusion follows now by applying $\Psi_{F_{\epsilon}}$ to any sequence $\{ S_n\}_n$ of polynomials in $\C \rtimes_{\alpha} \G$ converging to $S$.
\end{proof}

In the case where $\G$ is a discrete abelian group we can say something more. In that case the Pontryagin dual $\hat{\G}$ of $\G$, equipped with the compact-open topology is compact and therefore it admits a (normalized) Haar measure $d\gamma$. One can then verify that for an $S \in \C \rtimes_{\alpha} \G$ we have 
\begin{equation} \label{semisimpleeq}
\Phi_g (S)U_g = \int_{\hat{\G}} \hat{\alpha}_{\gamma} (S) \gamma(g) d\gamma, \quad  g \in \G,
\end{equation}
where $\hat{\G} \ni \gamma \mapsto \hat{\alpha}_{\gamma} \in \Aut{\C \rtimes_\alpha^r \G}$ is the dual action, i.e., $\hat{\alpha}_{\gamma} (c U_g)= \overline{\gamma(g)}c U_g$, $c \in \C$, $g \in \G$. 

Hence, if $\J \subseteq \C \rtimes_{\alpha} \G$ is a closed linear space which is left invariant by $\{\hat{\alpha}_{\gamma} \}_{\gamma \in \hat{\G}}$, then $\Phi_g (S)U_g \in \J$, for any $g \in \G$ and $S \in \J$.

\chapter{Definitions and Fundamental Results} \label{basic}
 In what follows, a \textit{dynamical system} $(\A, \G, \alpha)$ consists of an approximately unital operator algebra $\A$ and a locally compact (Hausdorff) group $\G$  acting continuously on $\A$ by completely isometric automorphisms, i.e., there exists a group representation $\alpha: \G \rightarrow \Aut\A$ which is continuous in the point-norm topology. (Here $\Aut \A$  denotes the collection of all completely isometric automorphisms of $\A$.) The group $\G$ is equipped with a left-invariant Haar measure $\mu$; the modular function of $\mu$ will be denoted as $\Delta$. Usually $\alpha(s)$, $s \in \G$, will be denoted as $\alpha_s$ and on occasion as $s$. 
  
  Now let $(\C , \G, \alpha)$ be a $\ca$-dynamical system and let $C_c(\G, \C)$ denote the continuous compactly supported functions from $\G$ into $\C$. Then $C_c(\G, \C)$ is a $*$-algebra in the usual way \cite[page 48]{Will}. In the sequel, if $c \in \C$ and $f\in C_c(\G)$ then $f \otimes c \in C_c(\G , \A)$ will denote the function $f \otimes c (s)= f(s)c$, $s \in \G$. Any covariant representation $(\pi, \phist, \H)$ of $(\C , \G, \alpha)$ induces a representation $\pi \rtimes \phist$ on $C_c(\G, \C)$, which is called the integrated form of $(\pi, \phist, \H)$ \cite[Proposition 2.23]{Will}. The full crossed product $\ca$-algebra $\C \cpf$ is the completion of $C_c(\G, \C)$ with respect to an appropriate supremum norm arising from all integrated covariant representations of $(\C , \G, \alpha)$. The reduced crossed product $\C\cpr$ is defined using the left regular representation for $\G$. See \cite{Will} for more details.

In the case of an arbitrary dynamical system $(\A , \G, \alpha)$, we appeal to the selfadjoint theory described above in order to define crossed product algebras. Here we have several options for defining a full or reduced crossed product, depending on the various choices of a $\ca$-cover for $\A$.

\begin{definition} \label{admit}
Let $(\A, \G, \alpha)$ be a dynamical system and let $(\C, j)$ be a $\ca$-cover of $\A$. Then $(\C, j)$ is said to be $\alpha$-admissible if there exists a group representation $\dot{\alpha}: \G \rightarrow \Aut(\C)$ which extends the representation
  \begin{equation} \label{detailona}
  \G \ni s \mapsto j\circ \alpha_s \circ j^{-1} \in \Aut (j(\A)).
  \end{equation} 
\end{definition}

Note that the representation $\dot{\alpha}$ of Definition~\ref{admit} is automatically continuous over a dense subalgebra of $\C$ and so an easy $\epsilon/3$ argument actually shows that $\dot{\alpha}: \G \rightarrow \Aut(\C)$ is a continuous group representation. In the sequel, since $\dot{\alpha}$ is uniquely determined by its action on $j(\A)$, both (\ref{detailona}) and its extension $\dot{\alpha}$ will be denoted by the same symbol $\alpha$.

\begin{definition}[Relative Crossed Product] \label{relativedefn}
Let $(\A, \G, \alpha)$ be a dynamical system and let $(\C, j)$ be an $\alpha$-admissible $\ca$-cover for $\A$. Then, $\A \rtimes_{\C, j, \alpha} \G$ and $\A \rtimes_{\C, j, \alpha}^r \G$ will denote the subalgebras of the crossed product $\ca$-algebras $\C \rtimes_{\alpha} \G$ and $\C \rtimes_{\alpha}^r \G$ respectively, which are generated by $C_c \big(\G , j(\A)\big) \subseteq C_c\big(\G , \C\big)$.
\end{definition}

One has to be a bit careful with Definition~\ref{relativedefn} when dealing with an \textit{abstract} operator algebra. It is common practice in operator algebra theory to denote a $\ca$-cover by the use of set theoretic inclusion. Nevertheless a $\ca$-cover for $\A$ is not just an inclusion of the form $A \subseteq \C$ but instead a pair $(\C, j)$, where $\C$ is a $\ca$-algebra, $j : \A \rightarrow \C$ is a complete isometry and $\C = \ca(j(\A))$. Furthermore, in the case of an $\alpha$-admissible $\ca$-cover,  it seems that the structure of the relative crossed product for $\A$ should depend on the nature of the embedding $j$ and one should keep that in mind when working with that crossed product. To put it differently, assume that $(\A, \G, \alpha)$ is a dynamical system and $(\C_i, j_i)$, $i=1,2$, are $\ca$-covers for $\A$. Further assume that the representations $\G \ni s \mapsto j_i \circ \alpha_s \circ j_i^{-1} \in \Aut (j_i(\A))$ extend to $*$-representations $\alpha_i : \G \rightarrow \Aut(\C_i)$, $i=1,2$. It is not at all obvious that whenever $\C_1 \simeq \C_2$ (or even $\C_1 = \C_2$), the $\ca$- dynamical systems $(\C_i, \G, \alpha_i)$ are conjugate nor that the corresponding crossed product algebras are isomorphic. Therefore the (admittedly) heavy notation $\A \rtimes_{\C, j, \alpha} \G$ and $\A \rtimes_{\C,j, \alpha}^r \G$ seems to be unavoidable.

We have already encountered the concept of equivalence between $\ca$-covers in Definition~\ref{firstequiv}. Indeed our view of $\cmax ( \A)$ and $\cenv( \A)$ in Chapter~\ref{prel} is essentially that of an equivalence class of $\ca$-covers and not just of a single element.

\begin{lemma} 
 Let $(\A, \G, \alpha)$ be a dynamical system and let $(\C_1, j_1)$ be an $\alpha$-admissible $\ca$-cover for $\A$. If $(\C_2, j_2)$ is another $\ca$-cover of $\A$ which is equivalent to $(\C_1 , j_1)$, then $(\C_2 , j_2)$ is also $\alpha$-admissible and so both representations $\G \ni s \mapsto j_i \circ \alpha_s \circ j_i^{-1} \in \Aut (j_i(\A))$ extend to $*$-representations $\alpha_i : \G \rightarrow \Aut(\C_i)$, $i=1,2$. . Furthermore
\[
\A \rtimes_{\C_1 , j_1, \alpha_1} \G \simeq  \A \rtimes_{\C_2, j_2, \alpha_2}\G\mbox{  and  }
\A \rtimes_{\C_1, j_1, \alpha_1}^r \G \simeq \A \rtimes_{\C_2, j_2, \alpha_2}^r\G
\]
via complete isometries that map generators to generators. 
\end{lemma}

\begin{proof}
Let $j:\C_1 \rightarrow \C_2$ be a $*$-isomorphism so that the diagram (\ref{little triangle}) commutes. To see that $(\C_2, j_2)$ is $\alpha$-admissible simply notice that $j \circ \alpha_{1, s} \circ j^{-1}$ extends $j_2\circ \alpha_s \circ j_2^{-1}$ to a $*$-automorphism of $\C_2$, for any $s \in \G$. Hence $(\C_2 , j_2)$ is also $\alpha$-admissible.

Now note that $j$ is \textit{$\G$-equivariant}, i.e., it implements a conjugacy between the $\ca$-dynamical systems $(\C_1, \G, \alpha_1)$ and $(\C_2, \G, \alpha_2)$. Indeed, if $x = j_1(a)$, $a \in \A$, then
\begin{align*}
j\alpha_{1,s}(x)=jj_1\alpha_s(a)
=j_2\alpha_s(a) = \alpha_{2,s}j(x), \quad s \in \G
\end{align*}
and since $j_1(\A)$ generates $\C_1$ as a $\ca$-algebra, $j\alpha_{1, s}(x)=\alpha_{2, s}j(x)$, for all $x \in \C_1$.

The conjugacy $j: \C_1\rightarrow \C_2$ between $(\C_1, \G, \alpha_1)$ and $(\C_2, \G, \alpha_2)$ implies that the (full) crossed product $\ca$-algebras are $*$-isomorphic \cite[Proposition 2.48]{Will}. Furthermore this isomorphism \textit{maps generators to generators}, i.e., it maps $j_1f$ onto $jj_1f=j_2f$, for any $ f \in C_c(\G , \A)$. This establishes that $\A \rtimes_{\C_1 ,j_1, \alpha_1} \G$ and $\A \rtimes_{\C_2, j_2, \alpha_2}\G$ are completely isometrically isomorphic. A similar argument, using \cite[Lemma A.16]{Es} this time, establishes the isomorphism of the operator algebras $\A \rtimes_{\C_1, j_1, \alpha_1}^r \G$ and $\A \rtimes_{\C_2, j_2, \alpha_2}^r\G$.
\end{proof}

The previous lemma will allow us to adopt a notation lighter than the notation $\A \rtimes_{\C, j, \alpha} \G$ and $\A \rtimes_{\C,j, \alpha}^r \G$ at least in the case where the $\ca$-covers are coming either from the $\ca$-envelope or from the universal $\ca$-algebra of $\A$. Indeed

\begin{lemma} \label{delicate}
 Let $(\A, \G, \alpha)$ be a dynamical system and let $(\C_i, j_i)$ be $\ca$-covers for $\A$ with either $(\C_i, j_i) \simeq \cenv(\A)$, $i=1,2$, or $(\C_i, j_i) \simeq \cmax(\A)$, $i=1,2$. Then there exist continuous group representations $\alpha_i : \G \rightarrow \Aut(\C_i)$ which extend the representations
  \[
  \G \ni s \mapsto j_i \circ \alpha_s \circ j_i^{-1} \in \Aut (j_i(\A)), \quad i=1,2.
  \]
 Furthermore
 $\A \rtimes_{\C_1 , j_1, \alpha_1} \G \simeq  \A \rtimes_{\C_2, j_2, \alpha_2}\G$ and
$\A \rtimes_{\C_1, j_1, \alpha_1}^r \G \simeq \A \rtimes_{\C_2, j_2, \alpha_2}^r\G
$, via complete isometries that map generators to generators.
 \end{lemma}

 \begin{proof} We deal with $\cmax(\A)$ and the full crossed product. Similar arguments work in all other cases as well. 
 
 The proof essentially follows from the well-known fact that any completely  isometric automorphism $\beta$ of $\A$ extends to a $*$-automorphism $\rho$ of $\cmax(\A) \equiv (\cmax(\A), j)$. Indeed, the defining property of $\cmax (\A)$ implies the existence of a $*$-homomorphism $\rho: \cmax(\A) \rightarrow \cmax(\A)$ so that $\rho\circ j = j \circ \beta$. Similarly, there exists a $*$-homomorphism $\rho': \cmax(\A) \rightarrow \cmax(\A)$ so that $\rho' \circ j = j \circ \beta^{-1}$. Hence, if $x=j(a)$, $a \in \A$ we have
 \[
\rho\rho'(x)=  \rho\rho'j(a) =\rho j\beta^{-1}(a) = j \beta \beta^{-1}(a)= j(a)=x,
 \]
 i.e., $(\rho\circ \rho')_{\mid j(\A)}= \id_{\mid j(\A)}$, and since $\A$ generates $\cmax(\A)$ as a $\ca$-algebra, $\rho \circ \rho' = \id$. Similarly $\rho' \circ \rho = \id $ and so $\rho  \in \Aut \cmax(\A)$ with $\rho\circ j = j \circ \beta$ and so $\rho_{\mid j(\A)}= (j \circ \beta \circ j^{-1})_{\mid j (\A)}$.
 
The previous paragraph implies the existence of $*$-automorphisms $\alpha_{i, s}$, $s \in \G$, which extend the maps $j_i \circ \alpha_s\circ j_i^{-1}\mid_{ j_i (\A)}$, $i=1,2$. Since
  \[
  \G \ni s \mapsto j_i \circ \alpha_s \circ j_i^{-1} \in \Aut (j_i(\A)), \quad i=1,2 
    \]
are group representation, the same is true for $\G \ni s \mapsto \alpha_{i, s}$ over a dense subalgebra of $\C_i$. Therefore the $\alpha_i: \G \rightarrow \Aut\C_i$ are group representations.   \end{proof}

Because of the previous two lemmas we can now write $\A \rtimes_{\cmax(\A), \alpha} \G$, $\A \rtimes_{\cenv(\A), \alpha} \G$ and similarly for the associated reduced crossed products. It turns out that in specific situations there are more crossed products to be associated naturally with the system $( \A, \G, \alpha)$. This is truly a feature of the non-selfadjoint world.

Our next results establish basic properties for the crossed product to be used frequently in the rest of the monograph. Both results are easy to prove in the case where $\G$ is discrete but the general case requires some agility.

\begin{lemma} \label{cai}
Let $(\A, \G, \alpha)$ be a dynamical system and let $(\C , j)$ be an $\alpha$-admissible $\ca$-cover for $\A$. Then the algebras $\A\rtimes_{\C, j, \alpha} \G$ and $\A \rtimes^r_{\C, j, \alpha} \G$ are approximately unital.
\end{lemma}

\begin{proof}
Consider the collection $\{\U_i \mid i \in \bbI\}$ of all compact neighborhoods of the identity $e \in G$, ordered by inverse set-theoretic inclusion and contained in a fixed compact set $K$. For each such neighborhood $\U_i$, choose a non-negative continuous function $w_i$ with $\supp w_i \subseteq \U_i$ and $\int w_i (s) d\mu(s)=1$.  

Set  $e_i \equiv w_i\otimes a_i$, $i \in \bbI$, where $\{ a_i\}_{i \in \bbI}$ is a contractive approximate identity for $\A$ (and therefore for $\C$). We claim that $\{ e_i\}_{i \in \bbI}$ is a left contractive approximate identity for $C_c(\G, \C)$ in the $L^1$-norm.

Indeed let $c \in \C$, $z \in C_c(\G)$ and fix an $\epsilon >0$. Then,
\[
\big(e_i(z\otimes c)\big)(s)=\int a_i\alpha_r(c)w_i(r)z(r^{-1}s)d \mu (r), \, \, \, s \in \G.
\]
Since the supports of the $w_i$ ``shrink" to $e \in \G$, we can choose the $i \in \bbI$ large enough so that the $a_i\alpha_r(c)$ are eventually $\epsilon$-close to $c$, for all $r \in \supp w_i$. Hence for such $i \in \bbI$ we have
\begin{equation} \label{eq:cai1}
\Big\| \big(e_i(z\otimes c)\big)(s) - \int c w_i(r)z(r^{-1}s) d \mu (r)\Big\| \leq \epsilon\|z\|_{\infty}
\end{equation}
for all $s \in \G$. Since left translations act continuously on $C_c(\G)$, we can also arrange for these $i \in \bbI$ to satisfy, $|z(r^{-1}s)-z(s)| \leq \epsilon$, for all $r \in \supp w_i$ and $s \in \G$. 
Hence,
\begin{equation} \label{eq:cai2}
\begin{split}
\Big\|  \int c w_i(r)z(r^{-1}s) d \mu (r) - \int c w_i(r)z(s) d \mu (r)\Big\| &\leq \epsilon\|c\|\int w_i(r)d \mu(r) \\
&=\epsilon \|c\|.
\end{split}
\end{equation}
However, $\int  c w_i(r)z(s) d \mu (r) = (z\otimes c )(s)$ and so (\ref{eq:cai1}) and (\ref{eq:cai2}) imply that 
\[
\Big\| \big(e_i(z\otimes c)\big)(s) - (z\otimes c )(s) \Big\| \leq \epsilon (\|z\|_{\infty} + \|c \|)
\]
for all $s \in G$ and sufficiently large $i \in \bbI$. From this it is easily seen that $\{ e_i\}_{i \in \bbI}$ is a left contractive approximate identity for $C_c(\G, \C)$ in the $L^1$-norm.

From the above it follows that $\{ e_i\}_{i \in \bbI}$ is a left contractive approximate identity for  $\C\rtimes_{\alpha} \G$ and $\C \cpr$. Hence by  \cite[Lemma 2.1.6]{BlLM} we have that $\{ e_i\}_{i \in \bbI}$ is also a right contractive identity and the conclusion follows.
\end{proof}

\begin{proposition}
Let $(\A, \G, \alpha)$ be a dynamical system and let $(\C , j)$ be an $\alpha$-admissible $\ca$-cover for $\A$. Then $\C \cpf$ is a $\ca$-cover for $\A\rtimes_{\C, j, \alpha} \G$ and $\C \cpr$ is a $\ca$-cover for $\A\rtimes_{\C, j, \alpha}^r \G$. 
\end{proposition}

\begin{proof} We verify the first claim only. Let $c \in \C$ and $ z \in \C_c(\G)$. We will show that if $z \otimes c \in \ca(\A \rtimes_{\C, j , \alpha} \G )$ then $z \otimes ac,  z \otimes a^*c\in \ca(\A \rtimes_{\C, j , \alpha} \G )$, for all $a \in \A$. This suffices to show that all elementary tensors in $\C_c(\G, \C)$ belong to $\ca(\A \rtimes_{\C, j , \alpha} \G )$ and the conclusion then follows from \cite[Lemma 1.87]{Will}.

Let $\{ e_i\}_{i \in \bbI}$ be the approximate identity of $\A\rtimes_{\C, j, \alpha} \G$  (Lemma~\ref{cai}) and let $(i_{\C}, i_{\G})$ be the  covariant homomorphism of $(\C, \G , \alpha)$ into $M\big(\C \rtimes_{\alpha} \G\big)$, appearing in \cite[Proposition 2.34]{Will}. Then
\[
z\otimes ac = \lim_i ( w_i\otimes aa_i)(z \otimes c) \in \ca(\A \rtimes_{\C, j , \alpha} \G ).
\]
On the other hand,
\[
z \otimes a^*c= \lim_i i_{\C}(a^*) e_i^*(z\otimes c) =\lim_i \big(e_ii_{\C}(a)\big)^*(z\otimes c).
\]
However, $e_i i_{\C}(a)(s)= w_i(s)a_i\alpha_r(a) \in \A$, for all $s \in \G$, and so $e_ii_{\C}(a) \in C_c(\G, \A)$. This implies that $z \otimes a^*c \in \ca(\A \rtimes_{\C, j , \alpha} \G )$ and the conclusion follows.
\end{proof}

The crossed product  $\A \rtimes_{\cmax(\A), \alpha} \G$ shares an important property which we describe in Proposition~\ref{Raeburnthm} below. But first we need a few definitions.

A \textit{covariant representation} of a dynamical system $(\A, \G, \alpha)$ is a triple $(\pi, \phist, \H)$ consisting of a Hilbert space $\H$, a strongly continuous unitary representation $\phist: G \rightarrow B(\H)$ and a non-degenerate, completely contractive representation  $\pi: \A \rightarrow B(\H)$ satisfying
\begin{equation*}
 \phist(s) \pi(a)=  \pi(\alpha_s(a)) \phist(s), \mbox{ for all } s \in G , a \in \A.
\end{equation*}
If we insist that the dimension of $\H$ is at most $\card(\A \times \G)$ then the collection of all covariant representations forms a set. (This is a crude requirement that can be refined further; for instance if $\A$ is separable and $\G$ is countable we can simply ask for $\H$ to be separable.) Nevertheless the direct sum of all covariant representations on a Hilbert space of dimension at most $\card(\A \rtimes \G)$ forms a representation $(\pi_{\infty}, \phist_{\infty}, \H_{\infty})$ that we call the universal covariant representation for $(\A, \G, \alpha)$.
A special class of covariant representations for $(\A, \G, \alpha)$ arises from the left regular representation $\lambda : \G \rightarrow B(L^2(\G, \mu))$. If $\pi:\A \rightarrow \H$ is a completely contractive representation of $\A$ then on the Hilbert space $L^2(\G, \H)\simeq \H\otimes L^2(\G)$, we define
\[
\bar{\pi} : \A \longrightarrow B(L^2(\G, \H)) \, \, \A \ni a \longrightarrow \bar{\pi}(a)
\]
with $\bar{\pi}(a)h(s)\equiv \pi\big(\alpha^{-1}_s(a)\big)\big(h(s)\big)$, $s \in \G$, $h \in L^2(\G, \H)$
and
\[
\lambda_{\H} : \G \longrightarrow B(\H \otimes L^2(\G)); \, \, \G \ni s \longrightarrow 1\otimes \lambda(s).
\]
A representation $( \overline{\pi}, \lambda_{\H})$ of the above form will be called a regular covariant representation for $(\A, \G, \alpha)$.

Our next result identifies a universal property of $\A \rtimes_{\cmax(\A), \alpha} \G$ and lends support to our subsequent Definition~\ref{fulldefn}.

\begin{proposition} \label{Raeburnthm}
 Let $(\A, \G, \alpha)$ be a dynamical system. Then
\begin{itemize}
\item[(i)] there exists a completely isometric non-degenerate covariant homomorphism $(i_{\A}, i_{\G})$ of $(\A, \G , \alpha)$ into $M\big(\A \rtimes_{\cmax(\A), \alpha} \G\big)$,
\item[(ii)] given a non-degenerate covariant representation $(\pi, \phist, \H)$ of $(\A, \G, \alpha)$, there is a non-degenerate representation $\pi \rtimes \phist$ of $\A \rtimes_{\cmax(\A), \alpha} \G$ such that $\pi = (\overline{\pi \rtimes \phist})\circ i_{\A}$ and $\phist = (\overline{\pi \rtimes \phist})\circ i_{\G}$, and,
\item[(iii)] $\A \rtimes_{\cmax(\A), \alpha} \G = \overline{\Span}\{ i_{\A}(a)\tilde{\i}_{\G}(z) \mid a \in \A, z \in C_c(\G) \},$ 
\end{itemize}
where 
\begin{equation} \label{iintegral}
\tilde{\i}_{\G}(z) \equiv \int_{\G} z(s) i_{\G}(s) d \mu (s), \quad \mbox{for all } z \in C_c(\G).
\end{equation}
\end{proposition}

\begin{proof}
Let $\C$ stand for $\cmax(\A)$. Before embarking with the proof note that the presence of a contractive approximate identity for $\A \cpf$ implies 
\begin{equation} \label{twomult}
M(\A \rtimes_{\C , \alpha} \G) \subseteq M \big(\C \cpf\big).
\end{equation}
Furthermore, the integral (\ref{iintegral}) is understood as in Proposition~\ref{prop integral}.

For $\C \cpf$ such a covariant representation $(i_{\C}, i_{\G})$ of $(\C, \G , \alpha)$ into $M\big( \C \cpf\big)$ exists by \cite[Proposition 2.34]{Will}. We will show that the same pair $(i_{\C}, i_{\G})$ restricted on $\A$ works for $\A \rtimes_{\C , \alpha} \G$ as well.

By \cite[Proposition 2.34]{Will},
\[
i_{\C} (c)f(s) = cf(s) \quad \mbox{ and } i_{\G}(t)f(s)=\alpha_{t}(f(t^{-1}s)),
\]
for all $f \in C_c(\G, \C)$ and $c \in \C$. From this, it is immediate that $(i_{\C}, i_{\G})$ maps $(\A, \G, \alpha)$ into $M ( \A \rtimes_{\C, \alpha} \G)$. Furthermore, $i_{\A}$ is non-degenerate because $\A \subseteq \C$ is approximate unital and $i_{\C}$ is non-degenerate. Hence (i) follows.

Corollary 2.36 in \cite{Will} shows that $i_{\C}(c)\tilde{\i}_{\G}(z) = z\otimes c$, $z \in C_c(\G)$, $ c \in \C$. This implies (iii).

It only remains to verify (ii). If $(\pi, \phist)$ is a non-degenerate covariant representation of $(\A, \G , \alpha)$, then there exists a non-degenerate $*$-representation $\rho$ of $\C$ so that $\rho \circ j = \pi$, where $j : \A \rightarrow \C$ is the canonical inclusion. But then 
\[
\begin{aligned}
\phist(s)\rho\big( j(a)\big) &= \phist(s)\pi(a)=\pi\big( (\alpha_s(a)\big)u(s) \\
&= \rho\Big( j \circ \alpha_s \circ j^{-1}\big(j(a)\big)\Big)\phist(s),
\end{aligned}
\]
for all $a \in \A$, and since $\C$ is generated by $\A$, the pair $(\rho, \phist)$ is a covariant representation for $(\C, \G , \alpha)$. Proposition 2.39 in \cite{Will} implies now the existence of a representation $\rho \rtimes \phist$, which satisfies the analogous properties of (ii) for $\C  \cpf$. If we set $\pi \rtimes \phist \equiv (\rho \rtimes \phist)_{\mid \A \rtimes_{\C, \alpha}\G}$, the conclusion follows. 
\end{proof}

The previous proposition shows that any covariant representation $( \pi , u)$ for $(\A, \G, \phi)$ ``integrates'' in a very precise sense to a completely contractive representation $\pi \rtimes u$ of  $\A \rtimes_{\cmax(\A), \alpha} \G$. Indeed, $\pi \rtimes u$ is given by the familiar formula
\[
(\pi \rtimes u)(f) = \int \pi\big( f(s)\big)u(s)d\mu (s), \quad f \in C_c(\G, \A).
\]
Our next result shows that this class of representations exhausts all the completely contractive representations of $\A \rtimes_{\cmax(\A), \alpha} \G$.

\begin{proposition} \label{classifyrepns}
Let  $(\A, \G, \phi)$ be a dynamical system and let $$\phi: \A \rtimes_{\cmax(\A), \alpha} \G \longrightarrow B(\H)$$ be a non-degenerate completely contractive representation. Then there exists a non-degenerate covariant representation $(\pi, u, \H)$ of $(\A, \G, \phi)$ so that $\phi = \pi\rtimes u$.
\end{proposition}

\begin{proof}
Since $\A \rtimes_{\cmax(\A), \alpha} \G$ is approximately unital, the representation $\phi$ is multiplier-nondegenerate, when viewing $B(\H)$ as the multiplier algebra of the compact operators (Remark~\ref{multiplier nondeg}). Let $\overline{\phi} : M(\A \rtimes_{\cmax(\A), \alpha} \G)  \rightarrow B(\H)$ be the canonical (unital) extension of $\phi$ by \cite[Proposition 2.6.12]{BlLM}. We set 
\[
\begin{split}
\pi(a) &= \overline{\phi}\big(i_{\A}(a)\big), \,\,\, a \in \A,\\
u(s) &= \overline{\phi}\big(i_{\G}(s)\big), \, \, \, s \in \G,
\end{split}
\]
where $(i_{\A}, i_{\G})$ is the covariant representation of $(\A, \G , \alpha)$ into $M\big(\A \rtimes_{\cmax(\A), \alpha} \G\big)$ appearing in Proposition~\ref{Raeburnthm}. 

Now notice that $(\pi, u)$ is a covariant representation of $(\A, \G, \phi)$. Indeed, for every $s \in \G$, $u(s) \in \B(\H)$ is a contraction with inverse the contraction $u(s^{-1})$, hence a unitary. Furthermore the map $s \mapsto u(s)$ is strictly continuous as the composition of two such maps. Finally $\pi$ is non-degenerate. Indeed $i_{\A}$ is non-degenerate so if $\{ a_i\}_{i \in \bbI}$ is a contractive approximate unit for $\A$ then $\{i_{\A}( a_i)\}_{i \in \bbI}$ is a contractive approximate unit for $\A \rtimes_{\cmax(\A), \alpha} \G$, i.e., it converges strictly to $I \in M\big(\A \rtimes_{\cmax(\A), \alpha} \G\big)$. Since $\overline{\phi}$ is strictly continuous, we obtain that $\{ \pi(a_i)\}_{i \in \bbI}$ converges strictly (and so strongly) to $I \in B(\H)$. Hence the non-degeneracy of $\pi$.

By Proposition~\ref{Raeburnthm} we obtain the representation $\pi\rtimes u$ that integrates $(\pi , u)$ and satisfies the conclusions of that result.

If $f \in C_c(\G, \A)$, then 
\begin{alignat*}{3}
 \phantom{X}&(\pi\rtimes u )(f) &&= \int \pi \big( f(s)\big)u(s) d \mu(s)&&\\
 &                       &&=\int \overline{\phi}\Big(i_{\A}\big(f(s)\big)\Big) \overline{\phi}\big(i_{\G}(s)\big) d \mu(s)&& \\
 &                  &&=\int \overline{\phi}\Big(i_{\A}\big(f(s)\big)i_{\G}(s)\Big) d \mu(s)&& \\
 &                  &&=\overline{\phi}\Big(\int i_{\A}\big(f(s)\big)i_{\G}(s) d \mu(s)\Big)&& \mbox{(by Proposition \ref{prop integral})}\\
 &                  &&=\phi (f) && \mbox{(by \cite[Corollary 2.36]{Will})}
  \end{alignat*}
  and the conclusion follows.
\end{proof}

We have gathered enough evidence for us now to justify the following definition.

\begin{definition}[Full Crossed Product] \label{fulldefn}
If $(\A, \G, \alpha)$ is a dynamical system then 
\[
\A \rtimes_{\alpha} \G \equiv  \A \rtimes_{\cmax(\A), \alpha} \G
\]
\end{definition}

In the case where $\A$ is a $\ca$-algebra, the algebra $\A \rtimes_{\alpha}\G$ is nothing else but the full crossed product $\ca$-algebra of $(\A, \G, \alpha)$. In the general case of an operator algebra $\A$, one might be tempted to define $\A \cpf$ as the relative crossed product $\A \rtimes_{\cenv(\A), \alpha} \G$, by virtue of the fact that $\cenv( \A)$ is a more tractable (perhaps more popular) object than  $\cmax (\A)$. Even though in the case where $\A$ is selfadjoint, $\A \rtimes_{\cenv(\A), \alpha} \G$ also reduces to the usual full crossed product, in the non-selfadjoint case it is not clear at all that $\A \rtimes_{\cenv(\A), \alpha} \G$ satisfies the universal properties that $\A\rtimes_{\cmax(\A), \alpha}\G$ does. Of course, any covariant representation of $( \cenv(\A), \G, \alpha)$ extends some covariant representation  of $(\A, \G, \alpha)$. The problem is that the converse may not be true, i.e., a covariant representation of $(\A, \G, \alpha)$ does not necessarily extend to a covariant representation of $( \cenv(\A), \G, \alpha)$, as it happens with $\cmax(\A)$. As it turns out, the identification $\A \cpf \simeq \A \rtimes_{\cenv(\A), \alpha} \G$ is a major open problem in this monograph, which is resolved in the case where $\G$ is amenable or when $\A$ is Dirichlet.

To make the previous paragraph more precise, let us show now that the properties of $\A\cpf$, as identified in Proposition~\ref{Raeburnthm}, actually characterize the crossed product as the universal object for covariant representations of the dynamical system $(\A, \G, \alpha)$. In the case where $\A$ is a $\ca$-algebra, this was done by Raeburn in \cite{Raeb}. Below we prove it for arbitrary operator algebras, borrowing from the ideas of \cite{Raeb} and \cite[Theorem 2.61]{Will}. 

\begin{theorem} \label{full Raeb}
Let $(\A, \G , \alpha)$ be a dynamical system. Assume that $\B$ is an approximately unital operator algebra such that
\begin{itemize}
\item[(i)] there exists a completely isometric non-degenerate covariant representation $(j_{\A}, j_{\G})$ of $(\A, \G , \alpha)$ into $M(\B)$,
\item[(ii)] given a non-degenerate covariant representation $(\pi, \phist, \H)$ of $(\A, \G, \alpha)$, there is a completely contractive, non-degenerate representation $L: \B \rightarrow B(\H)$ such that $\pi = \bar{L}\circ j_{\A}$ and $\phist = \bar{L}\circ j_{\G}$, and,
\item[(iii)] $\B = \overline{\Span}\{ j_{\A}(a)\tilde{\j}_{\G}(z) \mid a \in \A, z \in C_c(\G) \},$ 
\end{itemize}
where 
\begin{equation*} 
\tilde{\j}_{\G}(z) \equiv \int_{\G} z(s) j_{\G}(s) d \mu (s), \quad \mbox{for all } z \in C_c(\G).
\end{equation*}
Then there exists a completely isometric isomorphism $\rho: B \rightarrow \A \cpf$ such that 
\begin{equation} \label{twostooges}
\bar{\rho} \circ j_{\A} = i_{\A} \mbox{ and } \bar{\rho} \circ j_{\G} = i_{\G}
\end{equation}
where $(i_{\A}, i_{\G})$ is the covariant representation of $(\A, \G, \alpha)$ appearing in Proposition~\ref{Raeburnthm}.
\end{theorem}

\begin{proof}
We will show that the map 
\begin{equation} \label{rhodefn}
\B \ni \sum_{k} j_{\A}(a_k)\tilde{\j}_{\G}(z_k) \longrightarrow \sum_{k} i_{\A}(a_k)\tilde{\i}_{\G}(z_k) \in \A \cpf ,
\end{equation}
where $a_k \in \A$, $z_k \in C_c(\G)$, is a well-defined map, which is a complete isometry and therefore extends to the desired isomorphism $\rho: B \rightarrow \A \cpf$.

Let $\phi: \A\cpf \rightarrow B(\H)$ be a completely isometric non-degenerate representation and let $\overline{\phi}: M(\A \cpf) \rightarrow B(\H)$ its canonical extension. Let
\[
\begin{split}
\pi(a) &= \overline{\phi}\big(i_{\A}(a)\big), \,\,a \in \A,\\
u(s) &= \overline{\phi}\big(i_{\G}(s)\big), \, \,\,s \in \G.
\end{split}
\]
Then for any $a \in \A$ and $z \in C_c(\G)$ we have
\[
\begin{split}
L\big( j_{\A}(a)\tilde{\j}_{\G}(z)\big) &= \bar{L}\big( j_{\A}(a)\big)\bar{L} \big(\tilde{\j}_{\G}(z)\big) = \bar{L}\big(j_{\A}(a) \big) \int z(s) \bar{L}\big(j_{\G}(s)\big) d \mu(s)\\
                       &=\pi(a) \int z(s) u(s)d \mu(s) \\
                       &=  \overline{\phi}\big(i_{\A}(a)\big) \int z(s)  \overline{\phi}\big(i_{\G}(s)\big) d \mu(s)\\
                       &= \phi\big( i_{\A}(a)\tilde{\i}(z)\big).
                       \end{split}
                       \]
 Since $\phi$ is a complete isometry, the above shows that (\ref{rhodefn}) is a well-defined map which is a complete contraction. By reversing the roles of $\A \cpf$  and $\B$ in the above arguments, we obtain that (\ref{rhodefn}) is a complete isometry, as desired.
 
 It remains to verify (\ref{twostooges}). We indicate how to do this with the second identity and we leave the first for the reader. 
 
Fix a $z \in \C_c(\G)$ and $s \in \G$. An easy calculation using (\ref{iintegral}) reveals that $\tilde{\i}_{\G}(z)i_{\G}(s)= \tilde{\i}_{\G}(w)$, where $w \in C_c(\G)$ with $w(r) = \Delta(s^{-1})z(rs^{-1})$, $ r \in \G$. A similar calculation shows that $\tilde{\j}_{\G}(z)j_{\G}(s)= \tilde{\j}_{\G}(w)$ as well. Hence for any $a \in \A$ we have
\[
\begin{split}
\rho\big(j_{\A}(a)\tilde{\j}(z)\big)\bar{\rho}\big(j_{\G}(s)\big)&= \rho\big((j_{\A}(a)\tilde{\j}_{\G}(w)\big)\\&= i_{\A}(a)\tilde{\i}_{\G}(w) = i_{\A}(a)\tilde{\i}_{\G}(z)i_{\G}(s)\\
                        &=\rho\big(j_{\A}(a)\tilde{\j}(z)\big)i_{\G}(s).
                        \end{split}
 \]
Since the linear span of elements of the form $\rho\big(j_{\A}(a)\tilde{\j}(z)\big)$, $a \in \A$, $z \in C_c(\G)$, is dense in $\A \cpf$  and $\A \cpf$ is essential as a left ideal of $M(\A \cpf)$, we have $\bar{\rho}(j_{\G}(s))= i_{\G}(s)$, as promised. 
\end{proof}

We need the following

\begin{lemma} \label{soelementary}
Let $\A$ be a unital operator algebra and let $(C, j)$ be a $\ca$-cover for $\A$. Let $\alpha \in \Aut \C$ be a completely isometric automorphism satisfying $\alpha\big(j(\A)\big)=j( \A)$. If $\J_{\A} \subseteq \C$ denotes the Shilov ideal of $\A$, then $\alpha(\J_{\A})= \J_{\A}$. 
\end{lemma}

\begin{proof}
It is a result of Hamana~\cite{Hamana} (see \cite[Theorem 5.9]{Katssurvey} for a ``modern" proof) that if $\J \subseteq \C$ is any ideal so that the natural quotient map $q: \C \rightarrow \C /\J$ is completely isometric on $j(\A)$, then $\J \subseteq \J_{\A}$. Since $\alpha(\J_{\A})$ clearly satisfies this property, we have $\alpha(\J_{\A}) \subseteq \J_{\A}$. Similarly, $\alpha^{-1}(\J_{\A}) \subseteq \J_{\A}$ and so 
\[
\J_{\A}=\alpha\big( \alpha^{-1}(\J_{\A})\big)\subseteq \alpha(\J_{\A}) \subseteq \J_{\A},
\]
as desired.
\end{proof}

Our next result is a key step in the proof of Theorem~\ref{r=f}. In the proof, we make an essential use of the theory of maximal dilations of Dritschel and McCullough~\cite{DrMc}. The reader familiar with the earlier work of Kakariadis and Katsoulis will recognize the influence of \cite[Proposition 2.3]{KakKatJFA1} in the proof below.

\begin{lemma}\label{cis}
Let $(\A, \G, \alpha)$ be a unital dynamical system and let $(\C, j)$ be an $\alpha$-admissible $\ca$-cover for $\A$. If $\J_{\A} \subseteq \C$ denotes the Shilov ideal of $\A$,  then
\[
\A \rtimes_{\C , j, \alpha}^r \G \simeq \A/\J_{\A} \rtimes_{\C /\J_{\A}, \, q\circ j , \, \alpha}^r \G
\]
via a complete isometry that maps generators to generators.
\end{lemma}

\begin{proof}
Notice that by its maximality, the Shilov ideal $\J_{\A}$ is left invariant by the automorphisms $\alpha_s$, $s \in \G$.  Therefore we have a continuous representation $\alpha: \G \rightarrow \Aut\big(\C \slash \J_{\A}\big)$ and the crossed product $\A/\J_{\A} \rtimes_{\C /\J_{\A}, \alpha}^r \G$ is meaningful. 

The statement of the lemma asserts that the association
 \begin{equation} \label{contraction2}
 \C\slash \J_{\A}  \rtimes_{\alpha} \G \ni \sum _{i} z_i\otimes (a_i+ \J_{\A})\longmapsto \sum _{i} z_i \otimes a_i  \in \C \rtimes_{\alpha} \G, 
 \end{equation}
where $a_i\in \A$, $ z_i\in C_c( \G)$, is a well-defined map that extends to a complete isometry. (Note that the map $\A\slash \J_{\A} \ni a +\J_{\A} \mapsto a \in \A$ is a well-defined complete isometry.)

Let $\pi$ be a faithful representation of $\C$ on a Hilbert space $\H$ and let $(\bar{\pi}, \lambda_{\H})$ be the associated regular covariant representation of $(\C, \G, \alpha)$.
Consider the completely isometric map
\[
\phi:\A/\J_{\A}  \longrightarrow \B(\H): a + \J_{\A} \longmapsto \pi(a), \quad A\in \A.
\]
According to the Dritschel and McCullough result \cite[Theorem 1.2]{DrMc}, there is a
maximal dilation $(\Phi, \K)$ of $\phi$ which extends uniquely to a
representation of $\C/\J_{\A} $ such that
\[
P_{\H}\Phi(a+\J_{\A} )|_{\H}=\phi(a+\J_{\A} )=\pi(a),
\]
for all $a \in A$. Since $P_{\H\otimes L^2(\G)}=
P_{\H}\otimes I$, we have that
\[
P_{\H \otimes L^2(\G)}
\bar{\Phi}(a+\J_{\A} ))|_{\H\otimes L^2(\G)}=\bar{\pi}(a+\J_{\A} ),
\]
for all $a \in \A$. Also, $\lambda_{\K}(s)|_{\H\otimes L^2(\G)}=\lambda_{\H}(s)$, $s \in \G$, and so
\begin{align*}
\big \| \bar{\pi}\rtimes \lambda_{\H} &\big( \sum_i z_i \otimes a_i\big)\big \| = \big \|\sum_i \, \bar{\pi}(a_i) \int z_i(s)\lambda_{\H}(s)d\mu (s) \big\| \\
&= \Big\|P_{\H\otimes L^2(\G)}\left(\sum_i\,
\bar{\Phi}(a_i + \J_{\A})\int z_i(s)\lambda_{\K}(s) d\mu(s) \right)|_{\H\otimes L^2(\G)}\Big\|\\
& \leq \big\| \sum_i\,
\bar{\Phi}(a_i + \J_{\A})\int z_i(s)\lambda_{\K}(s) d\mu(s) \big\| \\
&= \big \| \bar{\Phi}\rtimes \lambda_{\K} \big( \sum_i z_i \otimes (a_i +\J_{\A}) \big)\big \| 
\end{align*}
The same is also true for all the matrix norms. Since the covariant representation $(\bar{\pi}, \lambda_{\H}, \H\otimes L^2(\G))$ norms  $\C \rtimes_{\alpha}^r \G$, the map in (\ref{contraction2}) is well defined and completely contractive.
By reversing the roles between $\A$ and $\A/\J(\A)$ in
the previous arguments, we can also prove that (\ref{contraction2}) is actually an isometry,
and the conclusion follows.
\end{proof}

The previous lemma applies only to unital dynamical systems. In order to take advantage of it in the general case, we require the following.

\begin{lemma} \label{transfer}
Let $(\A, \G, \alpha)$ be a dynamical system and assume that $\A$ does not have a unit.  Let $(\C, j)$ be an $\alpha$-admissible $\ca$-cover for $\A$. Then the operator algebras generated by $$C_c(\G, \A)\subseteq \A^1 \rtimes_{\C^1, j_1, \alpha} \G$$ and $$C_c(\G, \A)\subseteq \A^1 \rtimes_{\C^1, j_1, \alpha}^r \G$$ are isomorphic to $\A \rtimes_{\C, j , \alpha}\G$ and $\A \rtimes_{\C, j , \alpha}^r\G$ respectively via complete isometries that map generators to generators.
\end{lemma}
\begin{proof}
We address first the case of the full crossed product. Since any covariant representation $(\pi , U)$ of $(\C, \G, \alpha)$ extends to a covariant representation of $(\C^1, \G, \alpha)$, we have from \cite[Lemma 2.27]{Will} that the map 
\[
\C \cpf \supset C_c(\G, \C) \ni f \longmapsto f \in C_c(\G, \C^1) \subseteq \C^1\cpf
\]
is an isometry that extends to a $*$-injection of $\C\cpf$ into $\C^1\cpf$ that maps generators to generators. In particular, this injection maps $\A \rtimes_{\C, j , \alpha}\G$ onto the subalgebra generated by $C_c(\G, \A)\subseteq \A^1 \rtimes_{\C^1, j_1, \alpha} \G$.

 The case of the reduced crossed product follows from the fact that if $\pi:\C \rightarrow B(\H)$ is a faithful $*$-representation, then
 \[
 \overline{\, \pi_1}\rtimes\lambda_{\H}\mid_{C_c(\G, \C)}= \overline{\pi}\rtimes\lambda_{\H}\mid_{C_c(\G, \C)},
 \]
 where $\pi_1$ is the unitization of $\pi$.
\end{proof}

The following is one of the main results of this chapter and generalizes a classical result from the theory of crossed product $\ca$-algebras to the theory of arbitrary operator algebras. It shows that in the case of an amenable group $\G$, the crossed product  is a unique object. In particular, it allows us to identify $\A \rtimes_{\cmax(\A), \alpha} \G$ with $\A \rtimes_{\cenv(\A), \alpha} \G$ in a canonical way.  

\begin{theorem} \label{r=f}
 Let $(\A, \G, \alpha)$ be a dynamical system with $\G$ amenable and let $(\C, j)$ be an $\alpha$-admissible $\ca$-cover for $\A$. Then
 \[
 \A \rtimes_{\alpha} \G  \simeq  \A \rtimes_{\C , j, \alpha}\G \simeq  \A \rtimes_{\C , j, \alpha}^r \G
 \]
 via a complete isometry that maps generators to generators.
 \end{theorem}

 \begin{proof}
We begin with the case where $(\A, \G, \alpha)$ is a unital dynamical system. With the understanding that the symbol $\simeq$ stands for a complete isometry that sends generators to generators we have
 \[
  \A \rtimes_{\C , j, \alpha}\G \simeq  \A \rtimes_{\C , j, \alpha}^r \G
  \]
  because $\G$ is amenable. 
  
  \noindent On the other hand
 \begin{alignat*}{3}
 &\A \rtimes_{\C, j, \alpha}^r\G &&\simeq \A \rtimes_{\C \slash \J_{\A},\, q\circ j, \, \alpha}^r \G \quad \quad&& \mbox{(by Lemma \ref{cis})}\\
  \mbox{Also \phantom{XXXX}}&   &&   && \\
 &\A \rtimes_{\alpha} \G &&\simeq \A \rtimes_{\cmax(\A), j, \alpha} \G \quad\quad &&\mbox{(by definition)}\\
 &                       &&\simeq \A \rtimes_{\cmax(\A), j, \alpha}^r \G \quad \quad&&\mbox{(since $\G$ is amenable)}\\
 &                  &&\simeq \A \rtimes_{\cmax (\A) \slash \J_{\A},\, q\circ j, \, \alpha}^r \G \quad \quad&& \mbox{(by Lemma \ref{cis})}
  \end{alignat*}
  However both $\ca$-covers $\big( \C \slash \J_{\A} , q\circ j \big)$ and $\big( \cmax(\A) \slash \J_{\A}, q\circ j\big)$ give $\cenv(\A)$ and so Lemma \ref{delicate} implies $\A \rtimes_{\alpha} \G \simeq  
  \A \rtimes_{\C , j, \alpha}^r \G $, as desired.
  
  In the general case notice that from the above we have 
   \[
 \A^1 \rtimes_{\alpha} \G  \simeq  \A^1 \rtimes_{\C^1 , j, \alpha}\G \simeq  \A^1 \rtimes_{\C^1 , j, \alpha}^r \G
 \]
 via complete isometries that maps generators to generators. In particular these isometries map surjectively the operator algebras generated by $C_c(\G, \A)$ inside the crossed products appearing above. The conclusion follows now from Lemma~\ref{transfer}.
  \end{proof}
 
Of course, Theorem~\ref{r=f} does much more than just provide an isomorphism between relative (full) crossed products. It also allows us to utilize regular covariant representations for $(\cenv( \A), \G , \alpha)$ in order to norm the crossed product. Indeed

\begin{corollary} \label{norming}
Let $(\A, \G, \alpha)$ be a dynamical system and assume that $\G$ is amenable. If $\pi: \C \rightarrow B(\H)$ is a faithful non-degenerate $*$-representation of $\cenv(\A)$ then $\bar{\pi}\rtimes \lambda_{\H}$ is a completely isometric representation of $\A \rtimes_{\alpha} \G$.
\end{corollary}

\begin{proof} Since $\G$ is amenable, $\bar{\pi}\rtimes \lambda_{\H}$ is a faithful representation of $\cenv (\A) \rtimes_{\alpha} \G$, where $\alpha$ is the unique extension of $\G \ni s \mapsto \alpha_s \in \Aut (\A)$. By the previous results
$$\A \rtimes_{\alpha} \G \simeq \A \rtimes_{\alpha}^r \G\simeq \A \rtimes_{\cenv (\A) , \alpha}^r \G \subseteq \cenv (\A) \rtimes_{\alpha} \G$$
and the conclusion follows
\end{proof}
 
 Part of the proof of Theorem~\ref{r=f} establishes the fact that all relative reduced crossed products coincide with each other, even for non-amenable $\G$. Stated formally
 
 \begin{corollary} \label{allrcoincide}
 Let   $(\A, \G, \alpha)$ be a dynamical system, with $\G$ an arbitrary locally compact group, and let $(\C , j)$ be an $\alpha$-admissible $\ca$-cover for $\A$. Then,
\[
\A \rtimes_{\C, j, \alpha}^r \G  \simeq \A\rtimes_{\cenv(\A), \alpha}^r \G \simeq  \A\rtimes_{\cmax(\A), \alpha}^r \G
\]
 via complete isometries that map generators to generators.
 \end{corollary}
 
 \begin{proof}
If $\A$ is unital than the result follows from Lemma~\ref{cis} as in the proof of Theorem~\ref{r=f}. If $\A$ is non-unital, then 
\[
\A^1 \rtimes_{\C^1, j, \alpha}^r \G  \simeq \A^1\rtimes_{\cenv(\A)^1, \alpha}^r \G \simeq  \A^1\rtimes_{\cmax(\A)^1, \alpha}^r \G
\]
 via complete isometries that map generators to generators. These isometries map surjectively the operator algebras generated by $C_c(\G, \A)$ inside the reduced crossed products appearing above. The conclusion follows again from Lemma~\ref{transfer}.
 \end{proof}
 
 In light of Corollary~\ref{allrcoincide} we give the following definition.
 
\begin{definition}[Reduced Crossed Product] \label{reddefn}
If $(\A, \G, \alpha)$ is a dynamical system then the reduced crossed product of $( \A, \G , \alpha)$ is the operator algebra
\[
A \rtimes_{\alpha}^r \G \equiv \A\rtimes_{\cenv(\A), \alpha}^r \G 
\]
\end{definition}

\begin{remark}
(i) Since $\A\rtimes_{\cenv(\A), \alpha}^r \G \simeq  \A\rtimes_{\cmax(\A), \alpha}^r \G$, it follows that any regular covariant representation of $(\A , \G , \alpha)$ integrates to a continuous representation of $\A \rtimes_{\alpha}^r \G$. One can actually view $\A \cpr$ as the universal object for the regular covariant representations of $(\A, \G, \alpha)$.

(ii) If $(\A , \G , \alpha)$ is a $\ca$-dynamical system then it is well known that any regular covariant representation $(\overline{\pi}, \lambda_{\H})$ integrates to a faithful representation of $\A \cpr$, provided that $\pi$ is faithful. This remains true for arbitrary dynamical systems under the additional requirement that $\pi$ is a maximal, completely isometric map for $\A$. (Note that for a $\ca$-algebra $\A$, any faithful $*$-representation is automatically maximal and completely isometric.)
\end{remark}

We will now use the theory we have developed so far to obtain von Neumann type inequalities, where the role of the disc algebra is being played now by the crossed product $\A \rtimes_{\alpha} \G$. First we obtain a covariant version of a theorem of Naimark and Sz.-Nagy that applies to arbitrary operator algebras.

Let $\G$ be a group and let $\phi: \G \rightarrow B(\H)$. We say that $\phi$ is \textit{completely positive definite} if for every finite set of elements $s_1, s_2, \dots, s_n$ of $\G$, the operator matrix $(\phi(s_{i}^{-1}s_j))_{ij}$ is positive; if $\phi(e)=I$ then $\phi$ is said to be \textit{unital}.

Note that for a completely positive definite map $\phi: \G \rightarrow B(\H)$, the matrix $\left(\begin{smallmatrix} \phi(e)&\phi(s)\\\phi(s^{-1})&\phi(e)\end{smallmatrix}\right)$, $s \in \G$, is automatically positive and so 
\begin{equation} \label{Naimarktrivial}
\phi(s)^* = \phi(s^{-1}), \mbox{ for all } s \in \G.
\end{equation}

We need the following

\begin{lemma} \label{elementary}
Let $A, B \in B(\H)$, $B\geq 0$, be commuting operators. Then
\[
|\sca{ABx, x}|\leq \|A\|\sca{Bx , x},
\]
for any $x \in \H$.
\end{lemma}

\begin{proof}
Note that,
\begin{align*}
|\sca{ABx, x}|^2&= \left|\sca{B^{1/2}A B^{1/2}x, x}\right|^2 =\left|\sca{AB^{1/2}x, B^{1/2}x}\right|^2 \\
&\leq \sca{B^{1/2}A^*A B^{1/2}x, x}\sca{B x, x}\\
&\leq\|A\|^2\sca{Bx,x}^2
\end{align*}
as desired
\end{proof}

In the case where $\A$ is a $\ca$-algebra, the following result was established by McAsey and Muhly in \cite[Proposition 4.2]{McM}. In the generality appearing below, the result is new and its proof requires new arguments.

\begin{theorem} \label{Naimark}
Let $\A$ be a unital operator algebra, let $\G$ be a group and let $(\A, \G, \alpha)$ be a dynamical system. Let $\phi: \G \rightarrow B(\H)$ be a unital, strongly continuous and completely positive definite map and let $\rho: \A \rightarrow B(\H)$ be a unital completely contractive map satisfying
\begin{equation} \label{covariance}
 \phi(s) \rho(a)=  \rho(\alpha_s(a)) \phi(s), \mbox{ for all } s \in \G , a \in \A.
\end{equation}
Then there exists a Hilbert space $\K\supset \H$, a strongly continuous unitary representation
 $\hat{\phi}: \G \rightarrow B(\K)$ and a completely contractive map  $\hat{\rho}: \A \rightarrow B(\K)$ so that
 \[
\rho(a)= P \hat{\rho}(a)\mid_{P}, \quad \phi(s)= P \hat{\phi}(s)\mid_{P},
\]
 and
 \[
 \hat{\phi}(s)\hat{\rho}(a)= \hat{\rho}(\alpha_s(a))\hat{\phi}(s) , a \in \A, s \in \G,
\]
where $P$ is the orthogonal projection on $\H$. Furthermore, $\H$ reduces $\hat{\rho}(\A)$. In the case where $\rho$ is multiplicative, $\hat{\rho}$ is multiplicative as well.
\end{theorem}

\begin{proof} Since $\G$ acts completely isometrically on $\A$, this action extends to $\cenv(\A)$. Similarly, since $\rho$ is unital, it extends to a completely positive map on $\cenv(\A)$. We reserve the same symbols for these extensions. Note that these extensions do not satisfy (\ref{covariance}), but their restrictions on the operator system $\S(\A) \equiv \overline{\A + \A^*}$ do. Indeed, in (\ref{covariance}) replace $s$ with $s^{-1}$, $a$ with $\alpha_s(a)$ and use (\ref{Naimarktrivial}) to obtain $\phi(s)^*\rho(\alpha_s(a))=\rho(a)\phi(s)^*$
and by taking adjoints 
\[
\rho(\alpha_s(a^*))\phi(s)=\phi(s)\rho(a^*), \mbox{ for all } s \in \G, a \in \A,
\]
as desired. For the rest of the proof we concentrate on that system.

 We start by adopting the ideas of \cite[Theorem 4.8]{Paulsen} in our context. Consider the vector space $c_{00}(\G, \H)$ of finitely supported functions from $\G$ to $\H$ and define a bilinear function on this space by
\[
\big< \sum_{s'} h'_{s'}\chi_{s'} , \sum_s h_s\chi_s \big> = \sum_{s', s}\sca{\phi(s^{-1}s')h'_{s'}, h_s}.
\]
As in the proof of \cite[Theorem 4.8]{Paulsen}, we observe that $\sca{h, h}\geq 0$ and that the set $\N = \{h \in c_{00}(\G, \H)\mid \sca{h, h} = 0\}$ is a subspace of $c_{00}(\G, \H)$. We let $\K$ be the completion of $c_{00}(\G, \H)\slash \N $ with respect to the induced inner product and we identify $\H$ as a subspace of $\K$, via the isometry $V$ that satisfies $h \mapsto h\chi_e$.

Let $\hat{\phi}: \G \rightarrow B(\K)$ be left translation, i.e.,
\[
(\hat{\phi}(s)h)(s')= h(s^{-1}s').
\]
It is easy to see that $\hat{\phi}$ is a unitary representation and $\phi(s) = V^*\hat{\phi}V$. Since $V$ is an isometry, we simply write $\hat{\phi}(s)= P_\H \phi(s)\mid_{\H}$.

Defining $\hat{\rho}$ and verifying its properties requires more care. If $a \in \S(\A)$ then we define
\[
\hat{\rho}(a)\big(\sum_s h_s\chi_s + \N\big)=  \sum_s \rho\big(\alpha_s^{-1}(a)\big)h_s\chi_s + \N
\]
We need to verify that $\hat{\rho}$ is well defined. Assume that $\sum_{l=1}^{m} h_l\chi_{s_l} \in \N$, i.e.,
\[
\sca{Bh, h}=0
\]
where
\[
h = (h_1, h_2, \dots , h_m)^{T} \in \H^m \quad \mbox{ and } \quad B=(\phi(s_{k}^{-1}s_l))_{kl}.
\]
Now if
\[
C= \left( \begin{array}{cccc}
\rho(\alpha_{s_1}^{-1}(a)) & 0 & \dots & 0\\
0& \rho(\alpha_{s_2}^{-1}(a)) & \dots &0\\
\vdots & \vdots & \ddots &\vdots \\
0 &0 & \dots &\rho(\alpha_{s_m}^{-1}(a))\\
\end{array} \right)
\]
then the covariance condition (\ref{covariance}) implies that $B$ and $C$ commute. Hence
\begin{align*}
\big< \sum_{l=1}^m \rho\big(\alpha_{s_l}^{-1}(a)\big)&h_l\chi_{s_l} , \sum_{k=1}^m \rho\big(\alpha_{s_k}^{-1}(a)\big)h_k \chi_{s_k}  \big> \\
&=\sca{C^*BC h, h} =\sca{B^{1/2}C^*CB^{1/2} h, h} \\
&\leq \| C\|^2 \sca{Bh, h} = 0,
\end{align*}
as desired.

We now verify that $\hat{\rho}$ is completely contractive; this will require an application of Schwarz's inequality. Let $(a_{ij})_{ij}\in M_r(\A)$ be a contraction and we are to verify that $\left(\hat{\rho}(a_{ij})\right)_{ij}$ is also a contraction.

Start by noticing that if $s_1, s_2, \dots s_m \in \G$,
\[
A= \left( \begin{array}{cccc}
\big[  \rho\big(\alpha_{s_1}^{-1}(a_{ij}
)\big) \big]_{ij}& 0 & \dots & 0\\
0& \big[\rho\big(\alpha_{s_2}^{-1}(a_{ij}
)\big) \big]_{ij} & \dots &0\\
\vdots & \vdots & \ddots &\vdots \\
0 &0 & \dots &\big[\rho\big(\alpha_{s_m}^{-1}(a_{ij}
)\big) \big]_{ij}\\
\end{array} \right)
\]
is in $M_{mr}(\rho(\A))$
and $B= \big[ \phi(s_k^{-1}s_l)I_r\big]_{kl} \in M_{mr}(B(\H))$, then (\ref{covariance}) implies that $A$ and  $B$ commute. Furthermore, since $\rho \circ \alpha_{s_l}$ is completely contractive, an application of Schwarz's inequality implies
\begin{align*}
 \big[\rho\big(\alpha_{s_l}^{-1}(a_{ij})\big) \big]_{ij}^* \big[\rho\big(\alpha_{s_l}^{-1}(a_{ij})\big) \big]_{ij} &\leq( \rho \circ \alpha_{s_l}^{-1}) ^{(r)}\left(\big[(a_{ij}) \big]_{ij}^*\big[(a_{ij}) \big]_{ij}\right) \\
&\leq ( \rho \circ \alpha_{s_l}^{-1}) ^{(r)}(I_r) = I_r
\end{align*}
and so $A^*A\leq I_{mr}$, i.e., A is a contraction.

Now let $h = (h_1 +\N , h_2 + \N , \dots h_r +\N)^T \in \Big(c_{00}(\G, \H)\slash \N\Big)^r$ with $h_i= \sum_{k=1}^m h_{i k }\chi_{s_k}$. We calculate
\begin{align*}
\Big< [\hat{\rho}(a_{ij})]_{ij}h, h  \Big> &=\sum_{i, j =1}^{r}\big< \hat{\rho}(a_{ij})(h_j +\N), (h_i + \N)\big>\\
&=\sum_{i,j=1}^{r}\sum_{k,l=1}^{m}\Big< \rho\big(\alpha_{s_k}^{-1}(a_{ij})\big)\phi\big(s_k^{-1}s_l\big)h_{jl},h_{ik}\Big>\\
&=\big< ABx, x\big>,
\end{align*}
where $x = (x_1, x_2, \dots , x_m)^T$ with $x_l= (h_{1 l }, h_{2 l}, \dots h_{r l})^T$, $l=1,2, \dots m$. An application of Lemma~\ref{elementary} shows now that
\begin{align*}
\big| \big< [\hat{\rho}(a_{ij})]_{ij}h, h  \big> \big|&= |\left< AB x, x \right> |  \leq \|A\| \left<Bx , x \right> \\                               & \leq  \left<Bx , x \right> =  \left<  h, h \right>
 \end{align*}
             and so $\left(\hat{\rho}(a_{ij})\right)_{ij}$ is a contraction, as desired. Hence $\hat{\rho}$ is completely contractive (and so completely positive).

             It remains to verify that $\hat{\rho} (\A)$ reduces $\H$; here lies the reason why we extended the original dynamical system on $\S(\A)$. It is clear that $\hat{\rho}(\S(\A))$ leaves $\H$ invariant and since $\hat{\rho}$ is completely positive, $\H$ reduces $\hat{\rho}(\S(\A))$.
(If we had chosen to define $\hat{\rho}$ only on $\A$, then we would only have that $\hat{\rho }(\A)$ leaves $\H$ invariant.) \end{proof}

Note that in the proof of the above theorem, the only reason why we ask for $\A$ to be unital is to guarantee that the unital completely contractive map $\rho$ extends to a completely positive map on $\cenv(\A)$. If $\rho$ is assumed to be multiplicative, such an extension exists without that requirement, because of Meyer's result \cite[Corollary 3.3]{Mey}. This is implicitly used below in obtaining the promised von Neumann inequality.

\begin{corollary} \label{vonNeumann}
Let $(\A, \G, \alpha)$ be a unital dynamical system and assume that $\G$ is a locally compact amenable group. Let $\phi: \G \rightarrow B(\H)$ be a unital, strongly continuous and completely positive definite map and let $\rho: \A \rightarrow B(\H)$ be a completely contractive representation satisfying
\begin{equation} \label{covarianceNeumann}
 \phi(s) \rho(a)=  \rho\big(\alpha_s(a)\big) \phi(s), \mbox{ for all } s \in \G , a \in \A.
\end{equation}
Then, for any $f \in C_c(\G, \A)$, we have 
\begin{equation} \label{cntnsvonNeum}
\Big\| \int \rho\big(f(s)\big)\phi(s) d\mu(s) \Big\| \leq \Big\| \int \bar{\pi}\big(f(s)\big) \lambda_{\H}(s)d\mu(s)  \Big\|,
\end{equation}
where $\pi : \cenv(\A) \rightarrow B(\H)$ is a faithful $*$-representation and $(\bar{\pi}, \lambda_{\H})$ the associated regular covariant representation of $( \cenv(\A), \G, \alpha)$.
\end{corollary}

\begin{proof} By Theorem~\ref{Naimark}, there exists a Hilbert space $\K\supseteq \H$ and a covariant representation representation $(\hat{\rho}, \hat{\phi})$  of  $(\A, \G, \alpha)$, whose compression on $\H$ gives $(\rho, \phi)$. Hence 
\[
\Big\| \int \rho\big(f(s)\big)\phi(s) d\mu(s) \Big\| \leq \Big\| \int \hat{\rho}\big(f(s)\big)\hat{\phi}(s) d\mu(s)  \Big\|.
\]
On the other hand, the representation $(\hat{\rho}, \hat{\phi})$ extends to a covariant representation of the dynamical system $(\cmax(\A), \G, \alpha)$. (See the last paragraph of the proof of Proposition \ref{Raeburnthm}). Hence,
\[
 \Big\| \int \hat{\rho}\big(f(s)\big)\hat{\phi}(s) d\mu(s)  \Big\| \leq\|f\| _{\cmax(\A)\cpf}.
\]
Theorem~\ref{r=f} shows however that  on $C_c(\G, \A)$ all relative crossed product norms coincide. In particular
\[
\|f\| _{\cmax(\A)\cpf} = \| f \|_{\cenv(\A) \cpr}
\]
and the conclusion follows.
\end{proof}

\begin{remark} (i) Corollary~\ref{vonNeumann} achieves its most pleasing form in the case where $\G$ is discrete, as in that case (\ref{cntnsvonNeum}) becomes an inequality involving finite sums instead of integrals, i.e., 
\[
\Big\| \sum_{s} \rho(a_s)\phi(s) \Big\| \leq \Big\| \sum_s\bar{\pi}(a_s) \lambda_{\H}(s) \Big\|,
\]
where $a_s \in \A$ and $s$ ranges over a finite subset of $\G$.

(ii) We have defined $(\pi , u, \H)$ to be a covariant representation of $( \A, \G, \alpha)$ provided that 
\[
 \phist(s) \pi(a)=  \pi(\alpha_s(a)) \phist(s), \mbox{ for all } s \in \G , a \in \A.
\]
This is of course equivalent to 
\[
 \pi(\alpha_s(a))  =  \phist(s) \pi(a)u^*(s), \mbox{ for all } s \in \G , a \in \A.
\]
It is important to note that  there we have no analogue of Theorem~\ref{Naimark} nor Corollary~\ref{vonNeumann} for the second set of covariance relations.
\end{remark}

The reader that has followed us this far should recognize now why we choose to define the crossed product $\A \rtimes_{\alpha} \G$ as a universal object with regards to \textit{arbitrary} representations of $\A$ (Definition \ref{fulldefn}). It is true that had we chosen to work only with the relative crossed product $\A \rtimes_{\cenv(\A),\alpha} \G$, we would not need to work so hard with the various relative crossed products, including $\A \rtimes_{\cmax (\A), \alpha} \G$. However, since the ``allowable" representations of $\A$ would have been only the $\cenv(\A)$-extendable ones, the von Neumann inequality of Corollary ~\ref{vonNeumann} would have been unattainable. This added flexibility in our definition for $\A \rtimes_{\alpha} \G$ is truly invaluable.

Corollary~\ref{vonNeumann} also raises the question whether $\cenv(\A)\cpf$ is the ``best choice" in our von Neumann inequality. In other words, we wonder what is the $\ca$-envelope of $\A \rtimes_{\alpha} \G$ and $\A \rtimes_{\alpha}^r \G$. Clearly, Lemma~\ref{cis} implies that $\cenv (\A \rtimes_{\alpha}^r \G)$ is a quotient of $\cenv(\A) \rtimes_{\alpha}^r \G$ but beyond that, we don't know too much. This is going to be a recurrent theme in this monograph. It turns out that even in special cases, the problem of identifying the $\ca$-envelope of the crossed product is intimately related to problems in $\ca$-algebra theory which are currently open, such as the Hao-Ng isomorphism problem. We will have to say more about that later in this monograph.

For the moment, we deal with the case where $\G$ is an abelian group and $\A$ is a unital operator algebra. The case where $\G$ is discrete follows easily from the work we have done so far and from the ideas of either \cite{KakKatJFA1} in the $\bbZ$ case or more directly from \cite[Theorem 3.3]{DFK}, by choosing $P = \G$, $\tilde{\alpha} = \alpha$ and transposing the covariance relations. In the generality appearing below, the result is new and paves the way for exploring non-selfadjoint versions of Takai duality.

 \begin{theorem} \label{abelianenv}
 Let $(\A, \G, \alpha)$ be a dynamical system. If $\G$ is an abelian locally compact group, then
 \[
 \cenv(\A \rtimes_{\alpha} \G) = \cenv(\A) \rtimes_{\alpha} \G.
 \]
 \end{theorem}

 \begin{proof}
 Let $\C$ denote the $\ca$-envelope of $\A$. Let $e_i$, $i \in \bbI$, be the common contractive approximate identity of $\A \cpf$ and $\C \cpf$, as in Lemma~\ref{cai}. 
 
 By way of contradiction assume that $\{ 0 \} \neq \fJ \subseteq \C \cpf$ is the Shilov ideal for $\A \cpf$.  Since $\fJ$ is invariant by the dual action $\hat{\alpha}$, \cite[Corollary 2.2]{GootL} (or \cite[Corollary 3.4 (i)]{Nielsen} for non-separable systems) implies the existence of an $\alpha$-invariant ideal $\J \subseteq \C$ so that 
 \[
\fJ = \J \cpf.
 \]
 Now note that $\J \subseteq M ( \C \cpf)$ and furthermore, 
 \begin{equation} \label {Jstuff}
 \J (\C \cpf) \subseteq \J \cpf \subseteq \fJ.
 \end{equation} 
 If $L_x \in B(\C \cpf)$, $ x \in M(\C\cpf)$, stands for the left multiplication operator, then for arbitrary $a \in \A$, $ j \in \J$  we have
  \begin{alignat*}{2}
 &\| a -j\|&&\geq \sup_i \|L_{a -j}(e_i) \| =\sup_i \| ae_i -je_i\|\\
 &                       &&\geq \sup_i \|ae_i\| \quad  \mbox{(by (\ref{Jstuff}) and because $\fJ$ is a boundary ideal)}\\
 &                  &&= \sup_i \|L_{ae_i}\| = \sup \{ \|ae_ix\|\mid x \in \C\cpf, \|x\|=1,  i \in \bbI \}  \\
 &                  &&= \| L_{a}\|= \|a \| ,
  \end{alignat*}
  where $\{ e_i\}_{i \in \bbI}$ is the contractive approximate unit of $\A \cpf$ appearing in Lemma~\ref{cai}. A matricial variation of the above argument shows that $$\|a-j\|\geq\|a\|,$$ for arbitrary $a \in M_n(\A)$ and $j \in M_n(\J)$.
Therefore it follows that $\J \subseteq \C$ is a boundary ideal for $\A$. Since $\C = \cenv(\A)$, we obtain $\J = \{0\}$. But this implies that $\fJ  = \{ 0 \}$, which is a contradiction. 
 \end{proof}
 
 \begin{remark} Certainly the extension of \cite[Corollary 2.2]{GootL} to non-separable systems does not require the heavy machinery of coactions, as it happens in \cite[Corollary 3.4 (i)]{Nielsen}. Nevertheless we have not able to locate an appropriate reference in the literature that does not involve coactions. A  standard reference, even for the case of a non-separable system, seems to be \cite[Theorem 3.4]{GootL2}. However that result too involves the use of coactions. Furthermore, even though the proof of \cite[Theorem 3.4]{GootL2} works for arbitrary dynamical systems, the authors of \cite[Theorem 3.4]{GootL2}  make the blanket assumption that all $\ca$-dynamical systems appearing in their monograph are separable. 
 \end{remark}
 
 In Chapter~\ref{Hao} we will use the above theorem in order to give a proof of the Hao-Ng Theorem~\cite{HN} for locally compact abelian groups.
 
\chapter[Takai Duality]{Maximal C$^*$-covers, Iterated Crossed Products and Takai Duality} \label{Takai}

Even though most of the non-selfadjoint operator algebras currently under investigation are actually unital, we have gone to great lengths to build a theory of crossed products that encompasses non-unital algebras as well. There is a good reason for that and this becomes apparent in this chapter. Both the context of an iterated crossed product and the non-selfadjoint Takai duality presented here would be meaningless had we not incorporated non-unital algebras in our theory.

We begin with an important identity.

\begin{theorem} \label{cmaxthm}
Let  $(\A, \G, \alpha)$ be a dynamical system. Then
\begin{equation*} 
\cmax \big(\A \cpf\big)\simeq \cmax (\A)\cpf .
\end{equation*}
\end{theorem}

\begin{proof}
Let $\phi: \A \cpf  \rightarrow B(\H)$ be a completely contractive (perhaps degenerate) representation. Since $\phi(\A)$ is approximately unital, the subspace $  [ \phi(\A)(\H) ]$ reduces $\phi(\A)$. Therefore, by restricting on $  [ \phi(\A)(\H) ]$, we may assume that $\phi$ is actually non-degenerate.  (See \cite[Lemma 2.1.9]{BlLM} for more details.)

By Proposition~\ref{classifyrepns}, there exists a covariant representation $(\pi, u, \H)$ of $(\A, \G, \phi)$ so that $\phi = \pi\rtimes u$. Extend $\pi$ to a $\ca$-representation $\hat{\pi}: \cmax(\A) \rightarrow B(\H)$. 

We claim that $(\hat{\pi}, u, \H)$ is a covariant representation of $(\cmax (\A), \G, \phi)$. By taking adjoints in the covariance equation
 \[
\phist(s^{-1}) \pi(a)=  \pi(\alpha_s^{-1}(a)) \phist(s^{-1})
\]
and then setting $a=\alpha_s(b)$, we obtain $\phist(s) \pi(b)^*=  \pi(\alpha_s(b))^* \phist(s)$, i.e.,
\[
\tilde{\pi}(b^*)\phist(s)= \phist(s)\tilde{\pi}(\alpha_s(b)^*)= \phist(s)\tilde{\pi}(\alpha_s(b^*)) ,
\]
and the conclusion follows.
Furthermore the $\ca$-representation
\[
\hat{\pi}\rtimes u : \cmax (\A)\cpf \rightarrow B(\H)
\]
extends $\phi = \pi \rtimes u$. 

This shows that $\cmax (\A)\cpf$ satisfies the universal property for $\cmax \big(\A \cpf\big)$ and the conclusion follows.
\end{proof}

Let $\A$ be an operator algebra. Let $K, H$  be locally compact groups and consider continuous actions $\beta: K \rightarrow \Aut \A$ and $\delta : H \rightarrow \Aut( \A \rtimes_{\beta} K)$. The iterated crossed product $(\A \rtimes _{\beta}K)\rtimes _{\delta}H$ can be described as follows.

By Lemma~\ref{delicate} both $\beta$ and $\delta$ extend to actions $\beta: K \rightarrow \Aut \cmax(\A)$ and $\delta : H \rightarrow \Aut\big( \cmax(\A \rtimes_{\beta} K)\big)$ respectively, denoted by the same symbols for convenience. Now, Theorem~\ref{cmaxthm} shows that 
\[
\cmax(\A \rtimes_{\beta} K) \simeq \big( \cmax(\A) \rtimes_{\beta} K, j \big),
\]
where $j : \A\rtimes _{\beta}K \rightarrow \cmax(\A ) \rtimes_{\beta} K $ is the canonical map arising from the ``inclusion" $\A \subseteq \cmax(\A)$. Therefore we may identify  $(\A \rtimes _{\beta}K)\rtimes _{\delta}H$ with the norm closed subalgebra of $\big(\cmax(\A )\rtimes _{\beta}K\big)\rtimes _{\delta}H$ generated by $C_c(H, \A \rtimes_\beta K) \subseteq \big(\cmax(\A )\rtimes _{\beta}K\big)\rtimes _{\delta}H$.

In the case where both $K$ and $H$ are abelian there is a more convenient description of the iterated crossed product.

\begin{proposition} \label{doesntmatter}
Let $\A$ be an operator algebra. Let $K, H$  be locally compact abelian groups and consider continuous actions $\beta: K \rightarrow \Aut \A$ and $\delta : H \rightarrow \Aut( \A \rtimes_{\beta} K)$. Then the iterated crossed product  $(\A \rtimes _{\beta}K)\rtimes _{\delta}H$ is canonically and completely isometrically isomorphic with the norm closed subalgebra of $\big(\cenv(\A )\rtimes _{\beta}K\big)\rtimes _{\delta}H$ generated by $C_c(H, \A \rtimes_\beta K) \subseteq \big(\cenv(\A )\rtimes _{\beta}K\big)\rtimes _{\delta}H$.
\end{proposition}

\begin{proof}
By Theorem~\ref{r=f}, we have 
\[
(\A \rtimes _{\beta}K)\rtimes _{\delta}H \simeq (\A \rtimes _{\beta}K)\rtimes _{\cenv(\A\rtimes _{\beta}K), \delta} H.
\]
However, Theorem~\ref{abelianenv} shows that 
\[
\cenv(\A\rtimes _{\beta}K )\simeq \big( \cenv(\A)\rtimes _{\beta}K, j\big),
\]
where $j : \A\rtimes _{\beta}K \rightarrow \cenv(\A) \rtimes_{\beta} K $ is the canonical map arising from the ``inclusion" $\A \subseteq \cenv(\A)$. This implies the desired identification.
\end{proof} 

Before embarking with the Takai duality, we need a technical result. Let $\C$ be a $\ca$-algebra, $H, K$ locally compact groups and $\beta: K \rightarrow \C$, $\delta: H \rightarrow \C \rtimes_{\beta} K$ continuous actions. Then we can view $C_c(H \times K, \C)$ as a dense subspace of the iterated crossed product
\[
\big( \C \rtimes_{\beta} K\big)\rtimes_{\delta} H
\]
by associating to a ``kernel" $F \in C_c(H \times K, \C)$, the function $\lambda_{F} \in C_c(H, \C \rtimes_{\beta}K)$ defined by
\begin{equation} \label{defn;lambda}
\lambda_F(h)(k)\equiv F(h, k), \,\, h \in H, k \in K.
\end{equation}
Assuming a compatibility condition for $\delta$, one can show (see \cite[p. 191]{Will}) that actually the subspace 
\[
\{ \lambda_F\mid F \in  C_c(H \times K, \C)\}
\]forms a $*$-subalgebra of the iterated crossed product. The compatibility condition requires that $C_c(K, \C) \subseteq \C \rtimes_{\beta} K$ is invariant for $\delta$, and that 
\begin{equation} \label{compatible}
(h, h', k) \mapsto \delta_h\big( \lambda_F(h')\big)(k)
\end{equation}
is continuous with compact support in $h'$ and $k$. (For instance, if $\supp \delta(\lambda_F(h))\subseteq \supp\lambda_F(h)$, for all $h \in H$, then  (\ref{compatible}) is satisfied.) Actually one can show that for functions $\lambda_{F_i} \in C_c(H, \C \rtimes_{\beta}K)$, $i=1,2$, we have
\begin{equation} \label{product formula} 
\big(\lambda_{F_1}\lambda_{F_2}\big)(h', k') = \int_{H}\int_{K} \lambda_{F_1}(h)(k)\beta_{k}\Big(\delta_h\big(\lambda_{F_2}(h^{-1}h')\big)(k^{-1}k')\Big) d \mu_H d\mu_K
\end{equation}
How does this transfer to non-selfadjoint algebras? Assume now that the systems $(\A, K, \beta)$ and $(\A\rtimes_{\beta} K , H, \delta)$ are as in the beginning of this chapter and let $\C = \cmax(\A)$. Assume further that the compatibility condition is satisfied by $\delta$, regarding both its action on $C_c(K, \C) \subseteq \C \rtimes_{\beta} K$ and on $C_c(K, \A) \subseteq \C \rtimes_{\beta} K$.\footnote{In this case we simply require that $C_c(K, \A) \subseteq \C \rtimes_{\beta} K$ is invariant for $\delta$.}

\begin{lemma} \label{dense Takai}
If $\A$, $\C$, $\beta$ and $\delta$ are as in the paragraph above, then the set 
\begin{equation} \label{Takai subalg}
\{ \lambda_F\mid F \in  C_c(H \times K, \A)\} 
\end{equation}
forms a dense subalgebra of the iterated crossed product $(\A \rtimes _{\beta}K)\rtimes _{\delta}H$. 
\end{lemma}
\begin{proof} Indeed, (\ref{product formula}) shows that the set in (\ref{Takai subalg}) is a subalgebra of $(\A \rtimes _{\beta}K)\rtimes _{\delta}H$. The density follows from the fact that kernels of the form 
\[
F(h, k) =az(h)w(k), \,\, a \in \A, z \in C_c(H), w \in C_c(K)
\] give $\lambda_F= (a\otimes w) \otimes z$ and such elements form a total subset of $(\A \rtimes _{\beta}K)\rtimes _{\delta}H$. 
\end{proof}

A particular case of an iterated crossed product comes from the dual action of a locally compact abelian group $\G$ on the crossed product $\A \cpf$. Here we have a dynamical system $(\A, \G, \alpha)$, with $\G$ abelian, and we let  $K= \G$, $\beta = \alpha$, $H = \hat{\G}$ and $\delta = \hat{\alpha}$. The dual action $\hat{\alpha}$ is defined on $C_c(\G, \A)$ by $\hat{\alpha}_{\gamma}(f)(s)= \overline{\gamma(s)}f(s)$, $f \in C_c(\G, \A)$, $\gamma \in \hat{\G}$. (By Theorem~\ref{r=f}, it does not matter whether we consider $C_c(\G, \A)$ as a subalgebra of $\A\cpf$ or any other relative crossed product.)  

Let $\rho: \G \rightarrow B\big( L^2(\G, \mu)\big)$ be the right-regular representation of $\G$ and consider the dynamical system $\Ad \rho \colon \G \rightarrow \K\big( L^2(\G)\big)$ where $(\Ad \rho)_s \equiv \rho(s)T \rho(s^{-1})$, $s \in \G$. Then for the dynamical system $(\A, \alpha , \G)$, we can form the tensor product dynamical system $\alpha \otimes \Ad \rho \colon \G \rightarrow \Aut \big( \A \otimes  \K\big( L^2(\G)\big)$.

For $\ca$-algebras, the following is known as the Takai duality Theorem \cite{Takai}. We establish its validity for crossed products of arbitrary operator algebras.

\begin{theorem}[Takai duality] \label{Takai duality}
 Let $(\A , \G, \alpha)$ be a dynamical system with $\G$ a locally compact abelian group. Then 
\begin{equation} \label{Takai formula}
\big( \A \cpf\big)\rtimes_{\hat{\alpha}}\hat{\G} \simeq \A \otimes \K \big( L^2 (\G)\big),
\end{equation}
where $\K \big( L^2 (\G)\big)$ denotes the compact operators on $ L^2 (\G)$ and $\A \otimes \K \big( L^2 (\G)\big)$ is the subalgebra of $\cenv(\A) \otimes \K \big( L^2 (\G)\big)$ generated by the appropriate elementary tensors.

Furthermore, the complete isomorphism 
\[
\Phi \colon \
\big( \A \cpf\big)\rtimes_{\hat{\alpha}}\hat{\G} \longrightarrow \A \otimes \K \big( L^2 (\G)\big),
\]
which implements \textup{(\ref{Takai formula})} can be chosen to be equivariant for the double dual action $\hat{\hat{\alpha}} \colon \G \rightarrow \Aut \big((\A \cpf\big)\rtimes_{\hat{\alpha}}\hat{\G}\big) $ and the action $\alpha \otimes \Ad \rho \colon\G \rightarrow \Aut \big( \A \otimes  \K\big( L^2(\G)\big)$.
\end{theorem}

\begin{proof} The proof follows verbatim the plan laid down by Williams in \cite[Theorem 7.1]{Will}. What we need to do here is to keep track of where our non-selfadjoint operator algebra is mapped under the various maps appearing in Williams' proof. (For the record, Williams attributes his proof to Raeburn \cite{Raeb}, with an extra contribution by S. Echterhoff.)

Let $\C \equiv \cmax(\A)$. In \cite[Lemma 7.2]{Will}, it is shown that there exists an isomorphism 
\[
\Phi_1: \big( \C \cpf\big)\rtimes_{\hat{\alpha}}\hat{\G} \longrightarrow \big( \C\rtimes_{\id} \hat{\G}\big)\rtimes_{\hat{\id}^{-1}\otimes \alpha } \G.
\]
 Here $\C\rtimes_{\id} \hat{\G} \simeq  \ca(\hat{\G}) \otimes \C $ and the action $\hat{\id}^{-1}\otimes \alpha $ of $\G$ (which also satisfies the compatibility condition) is given by 
\[
(\hat{\id}^{-1}\otimes \alpha)_s(f)(\gamma) =\gamma(s) \alpha_s\big(f(\gamma)\big),
\]
where $f \in C_c(\hat{\G},  \C)$, $ s \in \G$ and $\gamma \in \hat{\G}$. Actually, $\Phi_1$ is constructed so that on kernels $F \in C_c(\hat{\G} \times \G, \C)$ it acts as
	\begin{equation} \label{Phi1action}
\Phi_1(F)(s, \gamma)= \gamma(s)F(\gamma, s), \,\, s \in \G, \gamma \in \hat{\G},
\end{equation}
in the sense that $\Phi_1(\lambda_F) = \lambda_{\Phi_1(F)}$.
Therefore $\Phi_1$ maps the linear space 
\begin{equation} \label{domain}
\{ \lambda_F \mid F \in C_c(\hat{\G}\times \G, \A)\} \subseteq \big( \A \cpf\big)\rtimes_{\hat{\alpha}}\hat{\G}
\end{equation}
onto the linear space 
\begin{equation} \label{range}
\{ \lambda_F \mid F \in C_c(\G \times \hat{\G}, \A)\} \subseteq  \big( \A\rtimes_{\id} \hat{\G}\big)\rtimes_{\hat{\id}^{-1}\otimes \alpha } \G.
\end{equation}
Recall that both $\hat{\alpha}$ and $\hat{\id}^{-1}\otimes \alpha$ satisfy the compatibility condition and so two applications of Lemma~\ref{dense Takai} show that the algebras appearing on the left side of (\ref{domain}) and (\ref{range}) are dense in the algebras appearing in the right sides of these relations. Hence we have a completely isometric surjection
\begin{equation} \label{Phi1}
\widetilde{\Phi}_1: \big( \A \cpf\big)\rtimes_{\hat{\alpha}}\hat{\G} \longrightarrow \big( \A\rtimes_{\id} \hat{\G}\big)\rtimes_{\hat{\id}^{-1}\otimes \alpha } \G.
\end{equation}

In \cite[Lemma 7.3]{Will} it is shown that there exists isomorphism 
\[
\Phi_2: \big( \C\rtimes_{\id} \hat{\G}\big)\rtimes_{\hat{\id}^{-1}\otimes \alpha } \G \longrightarrow C_0(\G, \C) \rtimes_{\lt \otimes \alpha} \G
\]
Here $(\lt \otimes \alpha)_s(f)(r)= \alpha_s\big( f(s^{-1}r)\big)$, $ f \in C_0(\G, \C)\simeq C_0(\G) \otimes \C$. By its construction, $\Phi_2$ satisfies
\[
(c \otimes \phi ) \otimes z \xmapsto{\phantom{xxx}\Phi_2\phantom{xxx}} (c \otimes \hat{\phi} ) \otimes z,
\]
where $\phi \in C_c(\hat{\G})$, $z \in C_c(\G)$ and $\hat{\phi}$ denotes the Fourier transform of $\phi$. Clearly $\Phi_2$ maps $\big( \A\rtimes_{\id} \hat{\G}\big)\rtimes_{\hat{\id}^{-1}\otimes \alpha } \G $ onto $C_0(\G, \A) \rtimes_{\lt \otimes \alpha} \G$ and so we have a complete isomorphism
\begin{equation} \label{Phi2}
\widetilde{\Phi}_2: \big( \A\rtimes_{\id} \hat{\G}\big)\rtimes_{\hat{\id}^{-1}\otimes \alpha } \G \longrightarrow C_0(\G, \A) \rtimes_{\lt \otimes \alpha} \G
\end{equation}
Now \cite[Lemma 7.4]{Will} provides an isomorphism 
\[
\Phi_3: C_0(\G, \C) \rtimes_{\lt \otimes \alpha} \G
 \longrightarrow C_0(\G, \C) \rtimes_{\lt \otimes \id} \G, 
\]
which satisfies 
\[
\Phi_3\big((a\otimes z)\otimes w\big)=\phi_3(a\otimes z)\otimes w,
\] 
where $z, w \in \C_c(\G)$ and $\phi_3 (a\otimes z)(s)= \alpha_s^{-1}(a)z(s)$, $s \in \G$. Clearly we have a complete isometry
\begin{equation} \label{Phi3}
\widetilde{\Phi}_3: C_0(\G, \A) \rtimes_{\lt \otimes \alpha} \G
 \longrightarrow C_0(\G, \A) \rtimes_{\lt \otimes \id} \G, 
\end{equation}
Combining (\ref{Phi1}), (\ref{Phi2}) and (\ref{Phi3}), we obtain 
\begin{equation} \label{just}
\big( \A \cpf\big)\rtimes_{\hat{\alpha}}\hat{\G}  \simeq C_0(\G, \A) \rtimes_{\lt \otimes \id} \G
\end{equation}
via the complete isometry $\widetilde{\Phi}_3\circ\widetilde{\Phi}_2 \circ\widetilde{\Phi}_1$. However
\[
\begin{split}
\phantom{x} C_0(\G, \C) \rtimes_{\lt \otimes \id} \G &\simeq \big(  C_0(\G) \otimes \C \big)  \rtimes_{\lt \otimes \id} \G \\
&\simeq \big( C_0(\G)  \rtimes_{\lt } \G\big) \otimes \C \\
&\simeq \K\big(L^2(\G)\big)  \otimes \C 
\end{split}
\]
by the Stone-von Neumann Theorem \cite[Theorem 4.24]{Will}. Now these isomorphisms preserve $\A$-valued functions, i.e., 
\[
C_0(\G, \A) \rtimes_{\lt \otimes \id} \G \simeq  \A \otimes K\big(L^2(\G)\big) .
\]
This combined with (\ref{just}) completes the proof of the first paragraph of the theorem.

In \cite[Lemma 7.5]{Will} it is shown that there exists an equivariant isomorphism $\phi_4$ from 
\[
\big(C_0(\G)\rtimes_{\lt}\G , \G , \rt\otimes \id\big) \mbox{ onto } \big(\K(L^2(\G)), \G, \Ad \rho\big)
\]
and so we have an equivariant isomorphism $\phi_4 \otimes \id$ from 
\[
\Big(\big(C_0(\G)\rtimes_{\lt}\G\big)\otimes \C , \G , (\rt\otimes \id)\otimes \alpha \Big) \mbox{ onto } \Big( \C \otimes \K\big(L^2(\G)\big) , \G,\alpha \otimes \Ad \rho  \Big),
\]
where $\C = \cenv(\A)$. Therefore \cite[Lemma 2.75]{Will} implies the existence of  an equivariant isomorphism $\Phi_4$ from 
\[
\big( C_0(\G, \C)\rtimes_{\lt \otimes \id}\G , \G , (\rt\otimes \alpha) \otimes \id \big) \mbox{ onto } \big(\C \otimes \K(L^2(\G)), \G, \alpha \otimes \Ad \rho  \big).
\]
Note that $\Phi_4$ preserves the non-selfadjoint subalgebras $C_0(\G, \A)\rtimes_{\lt \otimes \id}\G$ and $\A \otimes \K(L^2(\G))$ and so it establishes an equivariant isomorphism $\widetilde{\Phi}_4$ for the non-selfadjoint dynamical systems
\[
\big( C_0(\G, \A)\rtimes_{\lt \otimes \id}\G , \G , (\rt\otimes \alpha) \otimes \id \big) \mbox{ and } \big(\A \otimes \K(L^2(\G)), \G, \alpha \otimes \Ad \rho  \big).
\]

Finally let $\Phi= \Phi_3\circ \Phi_2 \circ \Phi_1$, where $\Phi_1, \Phi_2 , \Phi_3$ are as earlier in the proof. In the proof of \cite[Theorem 7.1]{Will}, it is shown that $\Phi$ establishes an equivariant isomorphism from 
\[
\big( ( \C \cpf)\rtimes_{\hat{\alpha}}\hat{\G}, \G, \hat{\hat{\alpha}}\big) \mbox{ onto } \big( C_0(\G, \C)\rtimes_{\lt \otimes \id}\G , \G , (\rt\otimes \alpha) \otimes \id \big)
\]
and so $\widetilde{\Phi}\equiv  \widetilde{\Phi}_3\circ  \widetilde{\Phi}_2 \circ  \widetilde{\Phi}_1$ establishes an equivariant isomorphism from 
\[
\big( ( \A \cpf)\rtimes_{\hat{\alpha}}\hat{\G}, \G, \hat{\hat{\alpha}}\big) \mbox{ onto } \big( C_0(\G, \A)\rtimes_{\lt \otimes \id}\G , \G , (\rt\otimes \alpha) \otimes \id \big).
\]
Composing $ \widetilde{\Phi}_4\circ  \widetilde{\Phi}$ gives the desired equivariant isomorphism from 
\[
\big( ( \A \cpf)\rtimes_{\hat{\alpha}}\hat{\G}, \G, \hat{\hat{\alpha}}\big) \mbox{ onto }  \big(\A \otimes \K(L^2(\G)), \G, \alpha \otimes \Ad \rho  \big).
\]
This completes the proof of the second paragraph of the theorem.
\end{proof}

\chapter{Crossed Products and the Dirichlet Property} \label{Dirichlet Sect}

A far more illuminating, but prohibitively longer title for this monograph should be ``Dirichlet algebras, tensor algebras and the crossed product of an operator algebra by a locally compact group". Indeed the initial motivation for this monograph came from our desire to understand when a Dirichlet operator algebra fails to be the tensor algebra of a $\ca$-correspondence. In principle, examples of such algebras should abound but remarkably, up until the recent paper of Kakariadis \cite{Kak}, none was mentioned in the literature. In this monograph we manage to come up with many additional examples (see Theorem ~\ref{crossedtensor}) and the apparatus for proliferating such examples is the crossed product of an operator algebra. In this chapter we produce the first such class of examples, with additional ones to come in later chapters. (See Theorem~\ref{infinite div}.)

Actually, we do even more here. In \cite{DKDoc} Davidson and Katsoulis introduced the class of semi-Dirichlet algebras. The semi-Dirichlet property is a property satisfied by all tensor algebras and the premise of \cite{DKDoc} is that this is the actual property that allows for such a successful dilation and representation theory for the tensor algebras. Indeed in \cite{DKDoc} the authors verified that claim by recasting many of the tensor algebra results in the generality of semi-Dirichlet algebras. What was not clear in \cite{DKDoc} was whether there exist ``natural" examples of  semi-Dirichlet algebras beyond the classes of tensor and Dirichlet algebras. It turns out that the crossed product is the right tool for generating new examples of semi-Dirichlet algebras from old ones, as Theorem~\ref{generatingSemiDir} indicates. By also gaining a good understanding on Dirichlet algebras and their crossed products (Theorems~\ref{Dirichletenvr} and \ref{Dirichletenv}) we are able to answer a related question of Ken Davidson: we produce the first examples of semi-Dirichlet algebras which are neither Dirichlet algebras nor tensor algebras (Theorem \ref{neither}).

\begin{definition}
Let $\B$ be an operator algebra and let $\cenv(\B) \simeq (\C, j)$.  Then $\B$ is said to be \textit{Dirichlet} iff $$\C = \overline{j(\B) +j(\B)^*}\equiv \S(\B).$$
\end{definition}

Many of the applications of the crossed product in this monograph involve Dirichlet operator algebras. Our first priority is to show that whenever $\A$ is Dirichlet, $\A \rtimes_{\alpha} \G$ and $\A \rtimes_{\alpha}^r \G$ are Dirichlet and also calculate the $\ca$-envelope in that important case. 

First we need the following lemma which gives a workable test for verifying the Dirichlet property.
Its proof usually follows as an application of a result of Effros and Ruan \cite[Proposition 3.1]{ER}, which asserts that completely isometric unital surjections between unital  operator algebras are always multiplicative. Below we give a new proof that avoids \cite[Proposition 3.1]{ER} but uses instead the existence of maximal dilations.

 \begin{lemma} \label{useful}
  Let $\B$ be an operator algebra contained in a $\ca$-algebra $\C$ and assume that $\S(\B) = \C$. Then, $\cenv(\B)\simeq (C, j)$, where $j:\B \rightarrow \C$ denotes the inclusion map.
  \end{lemma}

  \begin{proof}
 Assume first that $\B$ is unital and let $\C$ act on a Hilbert space $\H$. Consider the diagram
\[
\xymatrix{\C \ar[r]^{\rho_{*}\phantom{LL}}&B(\K)\ \ar[d]^{c}\\
\B \ar[r]_{j} \ar[ru]^{\rho}  \ar[u]^{j} & \C}
\]
where $\rho$ is a maximal dilation of $j$ on a Hilbert space $\K \supset \H$, $c: B(\K)\rightarrow \H$ is the compression on $\H$ and $\rho_*$ is the extension of $\rho$ to a $*$-homomorphism on $\C$ so that the above diagram commutes.

Since $\rho$ is a maximal dilation of the complete isometry $j: \B \rightarrow \C$, we have that $$\cenv (\B) = \ca(\rho (\B) )= \ca(\rho_{*} (B)) = \rho_*(\C).$$ Therefore it suffices to show that $\rho_{*}$  is a complete isometry, i.e., it is injective.

Assume that $\rho_{*}\big( j(b_1) + j(b_2)^*\big)=0$. Then
\begin{align*}
j(b_1) + j(b_2)^*&= c\big(\rho(b_1)\big) +c\big(\rho(b_2)\big)^*=c\big(\rho(b_1)\big) +c\big(\rho(b_2)^*\big) \\
&=c\big(\rho(b_1)+\rho(b_2)^*\big)= c\Big(\rho_*(j(b_1))+\rho_*(j(b_2))^*\Big) \\
&=c\Big(\rho_*\big(j(b_1)+j(b_2)^*\big)\Big)=0
\end{align*}
as desired.

If $\B$ does not have a unit, then let $j_1: \B^1 \rightarrow \C^1$ be the unitization of the inclusion map. Clearly the pair $(\C^1, j_1)$ satisfies the requirements of the lemma for the unital algebra $\B^1$ and so $\cenv(\B^1) = (\C^1, j_1)$. Since ${j_1}\mid_{\B}=j$ and $\ca(j_1(\B))= \C$, we conclude that $\cenv(\B) = (\C, j)$. 
\end{proof}

First we deal with the reduced crossed product.

\begin{theorem} \label{Dirichletenvr}
 Let $(\A, \G, \alpha)$ be a dynamical system and assume that $\A$ is a Dirichlet operator algebra. Then $\A \rtimes_{\alpha}^r \G$ is a Dirichlet operator algebra and
 \[
 \cenv(\A \rtimes_{\alpha}^r \G) = \cenv(\A) \rtimes_{\alpha}^r \G.
 \]
 \end{theorem}

 \begin{proof}
 From Definition~\ref{reddefn} we have
 \[
 \A \rtimes_{\alpha}^r \G \equiv \A \rtimes_{\cenv (\A ) , \alpha}^r \G
 \subseteq \cenv (\A ) \rtimes_{\alpha}^r \G
 \]
 Furthermore, since the elementary tensors are dense in $\C_c(\G, \A)$, it is easily seen that
 \[
 \S \big( \A \rtimes_{\cenv (\A ) , \alpha}^r \G \big) \simeq \cenv (\A ) \rtimes_{\alpha}^r \G.
 \]
 Hence the conclusion follows from Lemma~\ref{useful}.
 \end{proof}

 The case of the full crossed product of a Dirichlet operator algebra requires more work.

In what follows, if $(\A, \G, \alpha)$ is a dynamical system and $\A \subseteq \S\subseteq \cenv(\A)$ a unital operator system left invariant by the action of $\G$, then a covariant representation of $(\S, \G, \alpha)$ consists of a Hilbert space $\H$, a unitary representation $\phist: \G \rightarrow B(\H)$ and a completely contractive map  $\pi: \S \rightarrow B(\H)$ satisfying
$ \phist(s) \pi(a)=  \pi(\alpha_s(a)) \phist(s)$, $\mbox{ for all } s \in \G , a \in \S$.

 \begin{lemma} \label{forward}
 Let $(\A, \G, \alpha)$ be a unital dynamical system and let $(\S(\A), \G, \alpha)$ be the restriction of the natural extension $( \cenv(\A), \G, \alpha)$ on $\S(\A) = \overline{\A + \A^*} \subseteq \cenv(\A)$. Then any covariant representation $(\pi, \phist, \H)$ of $(\A, \G, \alpha)$ admits an extension to a covariant representation $(\tilde{\pi}, \phist, \H)$ of $(\S(\A), \G, \alpha)$.
 \end{lemma}
 \begin {proof} By \cite[Proposition 3.5]{Paulsen} the map
 \[
 \tilde{\pi} :  \A + \A^* \longrightarrow B(\H);\,\, a+b^*\longmapsto \pi(a)+\pi(b)^*, \quad a, b\in \A
 \]
 is well defined and extends to a completely contractive map on $\S(\A)$. By taking adjoints in the covariance equation
 \[
\phist(s^{-1}) \pi(a)=  \pi(\alpha_s^{-1}(a)) \phist(s^{-1})
\]
and then setting $a=\alpha_s(b)$, we obtain $\phist(s) \pi(b)^*=  \pi(\alpha_s(b))^* \phist(s)$, i.e.,
\[
\tilde{\pi}(b^*)\phist(s)= \phist(s)\tilde{\pi}(\alpha_s(b)^*)= \phist(s)\tilde{\pi}(\alpha_s(b^*)) ,
\]
and the conclusion follows.
 \end{proof}

 \begin{theorem} \label{Dirichletenv}
 Let $(\A, \G, \alpha)$ be a dynamical system and assume that $\A$ is a Dirichlet operator algebra. Then $\A \rtimes_{\alpha} \G$ is a Dirichlet operator algebra and
 \[
 \cenv(\A \rtimes_{\alpha} \G) \simeq \cenv(\A) \rtimes_{\alpha} \G.
 \]
 Furthermore, $\A \rtimes_{\alpha} \G \simeq \A \rtimes_{\cenv(\A), \alpha} \G$.
 \end{theorem}

 \begin{proof}
 We deal first with the unital case. We will show that the map
 \begin{equation} \label{contraction}
 \cenv(\A) \rtimes_{\alpha} \G \ni f \longmapsto f \in \cmax (A ) \rtimes_{\alpha} \G, \quad f \in C_c(\G, \A)
 \end{equation}
is a complete contraction (and therefore a complete isometry). Hence $\A \rtimes_{\alpha} \G$ embeds completely isometrically in $ \cenv(\A) \rtimes_{\alpha} \G$ via a map that maps generators to generators. Lemma~\ref{useful} then implies the conclusion.

  Let $(\pi, \phist, \H)$ be a covariant representation of $(\A, \G, \alpha)$. By the previous lemma, it admits an extension to a covariant representation $(\tilde{\pi}, \phist, \H)$ of $(\S(\A)=\cenv(\A), \G, \alpha)$. Note however that the map $\tilde{\pi}$ may not be multiplicative.

 We now claim that $(\tilde{\pi}, \phist, \H)$ admits a covariant Stinespring dilation, $(\hat{\pi}, \hat{\phist}, \K)$, so that $\hat{\phist}(\G)$ reduces $\H$.

The process for constructing that dilation is standard \cite{Jorg, PaulsExt}. Indeed start with the algebraic tensor product $\cenv(\A) \otimes \H$ with the positive semi-definite bilinear form coming from setting
\[
\left< a\otimes x, b \otimes y \right>= \sca{\tilde{\pi}(b^*a)x,y}
\]
for $a,b \in \A$ and $x,y \in \H$. If $\N = \{f \in \cenv(\A)\otimes \H\mid \sca{f, f} = 0\}$ then $\K_0 \equiv \cenv(\A) \otimes \H\slash \N$ becomes a pre Hilbert space, whose completion $\K$ is of dimension less than $\card(\A \times \G)$. The original Hilbert space is identified as a subspace of $\K$ via the isometry $\H \ni x \mapsto 1\otimes x \in \K$; let $P$ be the orthogonal projection onto (that copy of) $\H$

On $\K_0$ we define maps $\hat{\pi}(a) $, $a \in \A$, and $\hat{\phist}(s)$ by \[
\hat{\pi}(a)\Big(\sum a_i \otimes x_i\Big) = \sum (a a_i)\otimes x_i
 \]
  and
  \[
  \hat{\phist}(s)\Big(\sum a_i \otimes x_i\Big)=\sum \alpha_s(a_i) \otimes \phist(s)x_i
\]
respectively. We leave it to the reader to verify that $\hat{\pi}$ is well defined and bounded; this is done as in \cite[page 45]{Paulsen}. Note that $\hat{\phist}(\G)$ leaves $ \H\subseteq \K$ invariant and so $P$ commutes with $\hat{\phist}(\G)$. Furthermore if $a, b \in \A$ and $x , y \in \H$, then the calculation
\begin{align*}
\sca{ \hat{\phist}(s)\big(a\otimes x\big) , \hat{\phist}(s)\big(b \otimes y\big)}&= \sca{\alpha_s(a)\otimes \phist(s)x , \alpha_s(b)\otimes \phist(s) y}\\
&= \sca{ \tilde{\pi}\big(\alpha_s(b^*a)\big)\phist(s)x, \phist(s)y } \\
&=\sca{\tilde{\pi}(b^*a)x, y}= \sca{a\otimes x , b \otimes y}
\end{align*}
shows that $\hat{\phist}(s)$ is an isometry with inverse $\hat{\phist}(s^{-1})$, $s \in \G$, and thus a unitary. The strong continuity of $s \mapsto \hat{\phist}(s)$ is easy to verify.

Returning to (\ref{contraction}), given $f \in C_c(\G, \A)$, we have
\begin{align*}
\big\| (\pi\rtimes\phist) \big( f \big) \big\|&=\Big\|\int \pi\big(f(s)\big)\phist(s) d\mu(s) \Big\| \\
&= \Big \|\int P\hat{\pi}\big(f(s)\big)P\hat{\phist}(s)P d\mu(s)\Big \|\\
&=\Big\|\int P\hat{\pi}\big(f(s)\big)\hat{\phist}(s)P d \mu (s)\Big\| \\
& = \Big\|P \Big( \int \hat{\pi}\big(f(s)\big)\hat{\phist}(s) d \mu (s)\Big) P\Big\| \\
& \leq \big\| (\hat{\pi}\rtimes\hat{\phist})(f) \big\| \leq \|f\|
\end{align*}
where the last norm is calculated in $\cenv(\A) \cpf$. Since the covariant representation $(\pi, \phist, \H)$ of $(\A, \G, \alpha)$ is arbitrary, the map in (\ref{contraction}) is a contraction. A similar calculation holds at the matricial level and the conclusion follows.

Assume now that $\A$ is not unital. Since its unitization $\A^1$ is Dirichlet, the unital case above applies thus showing that $$\A^1 \rtimes_{\alpha} \G\equiv \A^1 \rtimes_{\cmax(\A)^1, \alpha} \G \simeq \A^1 \rtimes_{\cenv(\A)^1, \alpha} \G.$$ An application of Lemma~\ref{transfer} shows now that 
\[
\A \rtimes_{\cmax(\A), \alpha} \G \simeq \A \rtimes_{\cenv(\A), \alpha} \G .
\]
From this it is immediate that 
\[
 \cenv(\A \rtimes_{\alpha} \G) \simeq \cenv\big(\A \rtimes_{\cenv(\A), \alpha} \G\big) \simeq \cenv(\A) \rtimes_{\alpha} \G.
 \]
 since $\A \rtimes_{\cenv(\A), \alpha} \G \subseteq \cenv(\A)\cpf$ is Dirichlet.
 \end{proof}

In \cite{DKDoc}, Davidson and Katsoulis introduced a new class of operator algebras.

\begin{definition}
Let $\B$ be an operator algebra and let $\cenv(\B) \simeq (\C, j)$.  Then $\B$ is said to be \textit{semi-Dirichlet} iff $$j(\B)^*j(\B) \subseteq  \overline{j(\B) +j(\B)^*}\equiv \S\big(j(\B)\big).$$
\end{definition}

The name is justified by the fact that both $\B$ and $\B^*$ are semi-Dirichlet if and only if $\B$ is Dirichlet  \cite[Proposition 4.2]{DKDoc}.
As in the Dirichlet case, where $\S(\B)$ being a C$^*$-algebra implied that $\B$ was Dirichlet, we remove the necessity of working in the C$^*$-envelope.

\begin{lemma}
Let $\B$ be an operator algebra and let $(\C, j)$ be a $\ca$-cover of $\B$. If 
\[
j(\B)^*j(\B) \subseteq \S\big(j(\B)\big) \subseteq \C,
\]
then $\B$ is semi-Dirichlet.
\end{lemma}
\begin{proof}
Identify $\B$ with $j(\B)$  and let $\phi : \C \rightarrow \cenv(\B)$ be the surjective  $*$-homomorphism that maps $\B$ completely isometrically. It is immediate that
\[
\phi(\B)^*\phi(\B) \subseteq \phi(\S(\B)) \subseteq \S(\phi(\B)) \subseteq \cenv(\B).
\] 
Therefore, $\B$ is semi-Dirichlet.
\end{proof}

\begin{theorem} \label{generatingSemiDir}
 Let $(\A, \G, \alpha)$ be a unital dynamical system. If $\A$ is a semi-Dirichlet operator algebra then so is $\A \rtimes_\alpha^r \G$.
 \end{theorem}
 \begin{proof} 
 
 Let $z, w \in C_c(\G)$ with $\supp z =K$ and $\supp w = L$.  Let $a, b\in \A\subseteq \cenv(\A)$. Since $\A$ is semi-Dirichlet, there exist sequences $\{c_n \}_{n=1}^{\infty}$, $\{ d_n \}_{n=1}^{\infty}$ in $\A$ so that 
 \[
 a ^*b=\lim_n (c_n^*+d_n)
 \]
 Let 
 \[
 f_n \equiv  (z\otimes c_n)^*  (w_0\otimes 1)^* +(z_0\otimes 1)  (w\otimes d_n), \quad n \in \bbN,
 \]
 where,
 \begin{align*}
 z_0(s)&=\Delta(s^{-1})\overline{z(s^{-1})}\\
 w_0(s)&=\Delta(s)\overline{w(s^{-1})}, \quad \quad s \in \G.
 \end{align*}
 Clearly $f_n \in C_c(\G, \A)^* + C_c(\G, \A)$. We will show that $\{ f_n \}_{n =1}^{\infty}$ approximates $(z\otimes a)^*(w\otimes b)$ in the crossed product norm. 
 
 Note that
 \begin{align*} 
\big((z\otimes a)^* (w\otimes b)\big)(s)&=\int \Delta(r^{-1})\overline{z(r^{-1})}\alpha_{r}(a^*)w(r^{-1}s)\alpha_{r}(b)d\mu(r) \\
& =\int \Delta(r^{-1})\overline{z(r^{-1})}w(r^{-1}s)\alpha_{r}(a^*b)d\mu(r).
   \end{align*}
   On the other hand,
   \begin{equation}
   f_n(s) = \int \Delta(r^{-1}) \overline{z(r^{-1})}w(r^{-1}s)\alpha_{r}(c_n^* +d_n)d\mu(r)
   \end{equation}
   and so
    \[
 \big\|f_n(s) - \big((z\otimes a)^*(w\otimes b)\big)(s)\big\| \leq \| c_n^* +d_n - a^*b\| \|z\|_{\infty}\|w\|_{\infty}\mu(K^{-1}),
 \]
  for any $s \in \G$. Furthermore, $\supp f_n \subseteq   K^{-1}L$, $n \in \bbN$, which is a compact set. Hence, $\{f_n\}_{n=1}^{\infty}$ converges to $(z\otimes a)^*(w\otimes b)$ in the inductive limit topology \cite[Remark 1.86]{Will} and so in the $L^{1}$-norm. This suffices to prove the desired approximation.
  
 We have shown that  
 \[
(z\otimes a)^*  (w\otimes b) \in \overline{C_c(\G, \A)^* + C_c(\G, \A)}.
 \]
  Similarly, 
 \[
 \Big(\sum_{i=1}^n z_i\otimes a_i\Big)^*\Big(\sum_{j=1}^m w_j\otimes b_j\Big) \in \overline{C_c(\G, \A)^* + C_c(\G, \A)}.
 \]
Since the linear span of the elementary tensors is dense in $C_c(\G, \A)$ \cite[Lemma 1.87]{Will} we have
\[
\big(\A \rtimes_{\cenv(\A), \alpha}^r \G\big)^*\big(\A \rtimes_{\cenv(\A), \alpha}^r \G\big) \subset \S(\A \rtimes_{\cenv(\A), \alpha}^r \G).
\]
By the previous lemma, $\A \rtimes_\alpha^r \G$ is semi-Dirichlet.
  \end{proof}

Outside of the amenable case it is not known whether the full crossed product preserves the semi-Dirichlet property. Nevertheless, the following holds for arbitrary locally compact groups, with a proof similar to that of Theorem~\ref{generatingSemiDir}.

\begin{corollary}
Let $(\A, \G, \alpha)$ be a unital dynamical system. If $\A$ is a semi-Dirichlet operator algebra then so is $\A \rtimes_{\cenv (\A), \alpha} \G$.
\end{corollary}

We have built enough machinery now to present our first applications. It was an open question in \cite{DKDoc} whether all semi-Dirichlet algebras are tensor algebras of $\ca$-correspondences. Apparently, any Dirichlet algebra that fails to be a tensor algebra would serve as a counterexample to the question of Davidson and Katsoulis but no such examples were available at that time. It was Kakariadis in \cite{Kak} that produced the first example of a Dirichlet operator algebra which is not \textit{completely isometrically isomorphic} to the tensor algebra of a $\ca$-correspondence.

In what follows we use the crossed product of operator algebras to produce new examples of Dirichlet and semi-Dirichlet algebras which are not tensor algebras. Actually our algebras are not isomorphic to tensor algebras even by \textit{isometric} isomorphisms, thus improving Kakariadis' result. These are our first non-trivial examples of crossed products of operator algebras, with more to follow in later chapters. But first we have to resolve a subtle issue regarding the diagonal of a crossed product.

\begin{definition} If $\A$ is an operator algebra then the diagonal of $\A$ is the largest $\ca$-algebra contained in $\A$.
\end{definition}

 If $\A$ is contained in a $\ca$-algebra $\C$, then the diagonal of $\A$ is simply equal to $\A \cap  \A^*\subseteq \C$. We retain that notation for the diagonal of $\A$, without making any reference to the containing $\ca$-algebra $\C$.

\begin {proposition} \label{diagonal descr}
 Let $(\A, \G, \alpha)$ be a dynamical system and assume that $\G$ is a discrete amenable group. Then,
 \begin{equation} \label{diagonal}
\A \rtimes_{\alpha} \G \cap\big( \A \rtimes_{\alpha} \G \big)^*
= \ca\Big( \Big \{ \sum_g a_g U_g \mid\ a_g \in \A \cap \A^* , g \in \G \Big\}\Big).
\end{equation}
\end{proposition}

\begin{proof}
Consider $\A \rtimes_{\alpha} \G$ as a subset of $\cenv(\A) \rtimes_{\alpha} \G$; clearly the set $\A \rtimes_{\alpha} \G \cap (\A \rtimes_{\alpha} \G )^*$ contains all monomials $a_gU_g$, $a_g \in \A$, $g \in \G$, where $U_g \in M(\A \rtimes_{\alpha} \G)$ are the universal unitaries implementing the action of $\alpha$ on $\A$. Hence the inclusion $\supseteq$ in (\ref{diagonal}) is obvious.

Conversely let $X \in \A \rtimes_{\alpha} \G \cap (\A \rtimes_{\alpha} \G )^*$. Using an approximation argument involving finite polynomials in  $ \A \rtimes_{\alpha} \G$ approximating either $X$ or $X^*$, we see that $\Phi_g(X) \in \A \cap \A^*$, $g \in \G,$
where $\{ \Phi_g (X)\}_{g \in \G}$ denotes the Fourier coefficients of $X \in \cenv(\A) \rtimes_{\alpha} \G$. By Proposition~\ref{Ceasaroapprox}, $X$ can be approximated by finite polynomials with coefficients in $\{ \Phi_g (X)\}_{g \in \G}$ and $\{ U_g\}_{g \in \G}$, which completes the proof.
\end{proof}

It is not known to us whether an analogue of Proposition~\ref{diagonal descr} holds for the diagonal of $\A \rtimes_{\alpha} \G$, when $\G$ is not necessarily discrete and amenable.
  
  Recall that the non trivial conformal homeomorphisms of the unit disc $\bbD$ are classified as either \textit{elliptic, parabolic} or \textit{hyperbolic} depending on the nature of their fixed points. An elliptic conformal homeomorphism has only one fixed point in the interior of $\bbD$; such maps are conjugate via a M\"obius transformation to a rotation. The hyperbolic transformations have two fixed points which are both located on the boundary of $\bbD$. The parabolic transformations have only one fixed point located on the boundary of $\bbD$. (See \cite[Section 2.3]{CM} or \cite{Burckel} for a more detailed exposition and proof of these facts.)

\begin{theorem} \label{crossedtensor}
Let $\G$ be a discrete amenable group and let $\alpha: \G \rightarrow \Aut \big( \bbA(\bbD)\big)$ be a non-trivial representation. Assume that the common fixed points of the M\"obius transformations associated with $\{ \alpha_g\}_{g \in \G}$ do not form a singleton. Then $\bbA(\bbD) \rtimes_{\alpha} \G$ is a Dirichlet algebra which is not isometrically isomorphic to the tensor algebra of any $\ca$-correspondence.
\end{theorem}

\begin{proof}
By way of contradiction assume that there exists isometric isomorphism $\sigma: \bbA(\bbD) \rtimes_{\alpha} \G \rightarrow \T_X^+$, for some $\ca$-correspondence $(X, C)$.

By Proposition~\ref{diagonal descr} we have
\begin{equation} \label{ADdiagonal}
\bbA(\bbD) \rtimes_{\alpha} \G \cap\big( \bbA(\bbD) \rtimes_{\alpha} \G \big)^*
= \ca\Big( \Big \{ \sum_g s_g U_g \mid\ s_g \in \bbC , g \in \G \Big\}\Big)\simeq\ca(\G),
\end{equation}
where $U_g$ are the universal unitaries in $\bbA(\bbD) \rtimes_{\alpha} \G$.

By its universality,  $\ca(\G)$ admits a (non-zero) multiplicative linear functional $\rho$. Let $\fM_{\rho}$ be the collection of all (necessarily contractive) multiplicative linear functionals on $\bbA(\bbD) \rtimes_{\alpha} \G$ whose restriction on $\ca(\G)$ agrees with $\rho$.

\vspace{.05in}
\noindent \textit{Claim:} Either $\fM_{\rho} = \emptyset$ or $\fM_{\rho}$ contains exactly two elements.

\vspace{.05in}
\noindent Indeed any multiplicative form $\rho'$ on $\bbA(\bbD) \rtimes_{\alpha} \G$ is determined by its action on $\bbA(\bbD)$ and $\{U_g\}_{g \in \G}$. If it so happens that $\rho' \in \fM_{\rho}$, then (\ref{ADdiagonal}) implies that $\rho'$ is only determined by its action on $\bbA(\bbD)$ and therefore by its value on $f_0(z)=z$, $z \in \bbT$. If $\rho'(f_0) = z_0$, then the covariance relation $U_{g}f_0 = (f_0\circ\alpha_{g})U_{g}$ implies
\begin{align*}
\rho'(U_{g})z_0 & =\rho'(U_{g})\rho'(f_0)=\rho'(U_{g}f_0)\\
                    &=\rho'\big( (f_0\circ\alpha_{g})U_{g}\big)=\big(f_0\circ \alpha_{g}\big)(z_0)\rho(U_{g})\\
                    &=\alpha_{g}(z_0)\rho'(U_{g})
\end{align*}
for all $g \in \G$. Since $\rho'(U_{g}) \neq 0$, we obtain $z_0=\alpha_{g}(z_0)$ and so $z_0$ is a fixed point for all $\alpha_{g}$, $g \in \G$. If such points do not exist, then $\fM_{\rho} = \emptyset$. Otherwise, our assumptions imply that there exist exactly two common fixed points. Hence there are exactly two choices for $\rho'$, which both materialize by the universality of $\bbA(\bbD) \rtimes_{\alpha} \G$ (Proposition~\ref{Raeburnthm}). Hence $| \fM_{\rho} |=2$, as desired.

 \vspace{.05in}

 Corollary~2.10 in \cite{AS} (see also \cite[Proposition 3.1]{DavKatAn}) implies that the isomorphism $\sigma$ maps the diagonal of $\bbA(\bbD) \rtimes_{\alpha} \G$ onto the diagonal of $\T_X^+$. Hence the induced isomorphism $\sigma^*$ onto the spaces of multiplicative linear functionals satisfies $\sigma^*(\fM_{\rho}) = \fM_{\hat{\rho}}$ for some multiplicative linear functional  $\hat{\rho}$ on $C$. By the Claim above $|\fM_{\hat{\rho}}|$ is either $0$ or $2$. But this contradicts Proposition~\ref{extension} and the conclusion follows.
\end{proof}

As we saw in the proof of Theorem~\ref{crossedtensor}, under the assumptions of that theorem there are two choices for the common fixed points of $\{ \alpha_g\}_{g \in \G}$: either there are no such points or otherwise they form a two-point set. Let us show that both choices do materialize under an amenable action.

\begin{remark} (i) Let $\G = \bbZ$, let $\alpha$ be a hyperbolic M\"obius transformation of the disc and let $\alpha_n=\alpha^{(n)}$, $n \in \bbZ$. In that case the common fixed points form a two-point set. 

\vspace{.05in}

(ii) Let $z_1, z_2 \in \bbT$ be distinct points and consider two M\"obius transformations $\alpha_1, \alpha_2$ of the unit disc $\bbD$. Choose $\alpha_1$ so that it fixes both $z_1, z_2$ without being the identity self map on $\bbD$. Choose $\alpha_2$ so that it intertwines $z_1$ and $z_2$. Clearly the group $\G$ generated by these transformations has no common fixed points. However, the set $\{z_1, z_2\}$ is invariant by both generators and so $\G$ is amenable by \cite[Theorem]{Nebbia}. Choose $\alpha : \G \rightarrow \Aut \big((\bbA ( \bbD )\big)$ to be the identity representation.
\end{remark}

In particular the above remark implies that whenever $\alpha$ is a non-trivial hyperbolic automorphism of $\bbA(\bbD)$, then $\bbA(\bbD) \rtimes_{\alpha}\bbZ$ is not a tensor algebra. It is instructive to observe that in the case where $\alpha$ is elliptic then $\bbA(\bbD) \rtimes_{\alpha}\bbZ \simeq C(\bbT) \times_{\alpha} \bbZ^+$, which is indeed a tensor algebra. We will have more to say about this later in the monograph. 

We can now extend the previous result into a multivariable context. Recall, for $d\geq 2$, the non-commutative disk algebra $\mathfrak A_d$ is the universal unital operator algebra generated by a row contraction $[T_1 \cdots T_d]$ \cite{Pop}.  The maximal ideal space is $M(\mathfrak A_d) \simeq \overline{\bbB}_d$ and so every automorphism $\varphi$ of $\mathfrak A_d$ induces an automorphism $\varphi^*$ of $\overline{\bbB}_d$ by composition $\varphi^*(\rho) = \rho\circ\varphi$. It is established in \cite{DavPit, Pop2} that the isometric automorphisms of $\mathfrak A_d$ are in bijective correspondence with $\Aut(\bbB_d)$  which turn out to be unitarily implemented and thus completely isometric automorphisms.

In the same way as the disk there are automorphisms of $\overline\bbB_d$ that fix exactly two points, see \cite[Example 2.3.2]{Rudin}. Therefore, in exactly the same way as the proof of Theorem \ref{crossedtensor}, we can now produce semi-Dirichlet algebras that are not isometrically isomorphic to a tensor algebra of any C$^*$-correspondence,  thus providing new examples for the theory in \cite{DKDoc}, not covered by the tensor algebra literature.

\begin{theorem} \label{ncdiscexample}
Let $\G$ be an amenable discrete group and let $\alpha : \G \rightarrow \Aut \big( \mathfrak A_d\big)$ be a representation. Assume that the common fixed points of the transformations associated with $\{ \alpha_g\}_{g \in \G}$ form a finite set which is not a singleton. Then $\mathfrak A_d \rtimes_\alpha \G$ is a semi-Dirichlet algebra which is not isomorphic to the tensor algebra of any C$^*$-correspondence.
\end{theorem}

In the case where $\G$ is abelian, we can say something more definitive about $\mathfrak A_d \rtimes_\alpha \G$. Indeed in that case, Theorem~\ref{abelianenv} shows that $\cenv (\mathfrak A_d \rtimes_\alpha \G) \simeq \O_d \cpf$. It is easy to see now that $\mathfrak A_d \rtimes_\alpha \G$ is not a Dirichlet algebra, thus showing that $\mathfrak A_d \rtimes_\alpha \G$ is a semi-Dirichlet algebra which is neither a tensor algebra nor a Dirichlet algebra. This answers a question of Ken Davidson that was communicated to both authors on several occasions. Stated formally

\begin{corollary} \label{neither}
There exist semi-Dirichlet algebras which are neither Dirichlet nor isometrically isomorphic to the tensor algebra of any $\ca$-correspondence.
\end{corollary}

\chapter{Crossed Products and Semisimplicity} \label{semis}

In this chapter we consider the semisimplicity of crossed products by locally compact abelian groups. Recall from Theorem \ref{r=f} that there is a unique crossed product for such groups. 

We begin by reminding the reader of the definition of the Jacobson Radical of a (not necessarily unital) ring.

\begin{definition}
Let $\R$ be a ring. The Jacobson radical $\Rad \R$ is defined as the interchapter of all maximal regular right ideals of $\R$. (A right ideal $\I \subseteq \R$ is regular if there exists $e \in \R$ such that $ex - x \in \I$,  for all $x \in \R$.)
\end{definition}

An element $x$ in a ring $R$ is called right quasi-regular if there exists $y \in \R$ such that $x+y+xy=0$. It can be shown that $ x \in \Rad \R$ if and only if $xy$ is right quasi-regular for all $y \in \R$. This is the same as $1+xy$ being right invertible in $\R^1$ for all $y \in \R$. 

In the case where $\R$ is a Banach algebra we have
\[
\begin{split}
\Rad \R &= \{ x \in \R \mid \lim_n \|(xy)^n\|^{1/n} = 0, \mbox{ for all } y \in \R\} \\
		&= \{ x \in \R \mid \lim_n \|(yx)^n\|^{1/n} = 0, \mbox{ for all } y \in \R\}.
		\end{split}
		\]

A ring $\R$ is called semisimple iff $\Rad \R =\{0\}$.	

The study of the various radicals is a central topic of investigation in Abstract Algebra and Banach Algebra theory. In Operator Algebras, the Jacobson radical  and the semisimplicity of operator algebras have been under investigation since the very beginnings of the theory. In his seminal paper \cite{Rin}, Ringrose characterized the radical of a nest algebra, a work that influenced many subsequent investigations in the area of reflexive operator algebras. Around the same time, Arveson and Josephson \cite{ArvJ} raised the question of when the semicrossed product of a commutative $\ca$-algebra by $\bbZ^+$ is semisimple. This problem received a good deal of attention as well  \cite{M, Pet, Pet2} and it was finally solved in 2001 by Donsig, Katavolos and Manoussos \cite{DKM}, building on earlier ideas of Donsig \cite{Don}.

In Theorem~\ref{firstsemisimple} we discover that the semisimplicity of an operator algebra is a property preserved under crossed products by discrete abelian groups. This provides a huge supply of semisimple operator algebras and also raises the question of whether or not the converse is true. In order to investigate this, we go back to a class of operator algebras that has been investigated quite extensively by Davidson, Donsig, Hopenwasser, Hudson, Katsoulis, Larson, Peters, Muhly, Pitts, Poon, Power, Solel and others: triangular approximately finite-dimensional (abbr. TAF) operator algebras \cite{DavKatAdv, Don, DonH, DonHJFA, Hu, LS, Pow}. This is one of the main topics of study of this chapter. Another is the semisimplicity of crossed products by compact abelian groups. This is done in Section~\ref{compact abelian semisimple} where we uncover some interesting behavior.

In a recent paper \cite{DFK}, Davidson Fuller and Kakariadis make a comprehensive study of semicrossed products of operator algebras by discrete abelian groups. It turns out that our ideas on the semisimplicity of crossed products by abelian groups are also applicable on semicrossed products as well. We devote a whole section on this topic at the end of this chapter. 

 \section{Crossed products by discrete abelian groups}
 
 We begin with the following motivating result.

\begin{theorem}\label{firstsemisimple}
Let $(\A, \G, \alpha)$ be a dynamical system with $\G$ a discrete abelian group. If $\A$ is semisimple then $\A \rtimes_\alpha \G$ is semisimple.
\end{theorem}
\begin{proof}
Assume that the crossed product is not semisimple and so there exists $0 \neq S \in \Rad \A \rtimes_\alpha \G$. Any automorphism of  $\A \rtimes_\alpha \G$ fixes the Jacobson radical, which is \textit{closed}, and so the discussion at the end of section~\ref{abelian} implies that $\Phi_g(S)U_g  \in \Rad \A \rtimes_\alpha \G$ for all $g\in \G$.

Now if $\{ e_i\}_{i \in \bbI}$ is a contractive approximate unit for $\A$, then
\[
\Phi_g(S) = \lim_i \big( \Phi_g(S) U_g \big)\alpha_g^{-1}(e_i)U_g^*,
\]
and so $\Phi_g(S) \in \Rad \A \rtimes_\alpha \G$, for all $g \in \G$. By Proposition \ref{Ceasaroapprox} since $S\neq 0$ there is a $g\in \G$ such that $\Phi_g(S) \neq 0$ and so 
 \[
0 \neq \Phi_g(S) \in (\Rad \A \rtimes_\alpha \G) \cap \A \subseteq \Rad \A.
\]
Therefore, $\A$ is not semisimple.
\end{proof}

Naturally, one asks whether the converse of the above result is true. This brings us to the study of crossed products and semisimplicity in the context of strongly maximal TAF algebras with regular $*$-extendable embeddings.  Studying this class alone will provide us with a good idea of the richness of the theory. As we will see, even very ``elementary" automorphisms, i.e., quasi-inner automorphisms, can be used to generate crossed product algebras with interesting properties. Below we give some pertinent definitions and a few instructive examples. We direct the reader to \cite{Pow} for a comprehensive treatment of non-selfadjoint AF algebras.

Let $\fA = \varinjlim (\fA_n, \rho_n)$ be a unital AF $\ca$-algebra via regular embeddings \cite[Section~\ref{Dirichlet Sect}.9]{Pow} and further assume that $\rho_n(\A_n ) \subseteq \A_{n+1}$, $n=1, 2, \dots, $ where $\A_n$ denotes the subalgebra of upper triangular matrices in $\fA_n$. The limit algebra $\A = \varinjlim (\A_n, \rho_n)$ is said to be a \textit{strongly maximal TAF algebra}. In the case of a strongly maximal TAF algebra $\A = \varinjlim (\A_n, \rho_n)$ the \textit{diagonal} $ \C \equiv \A\cap \A^*$ of $\A$ satisfies
$$\C = \varinjlim (\C_n, \rho_n), \mbox{ where }\C_n = \A_n \cap \A^*_n, \, n = 1, 2, \dots.$$
It is easy to see that the Shilov ideal of $\A = \varinjlim (\A_n, \rho_n) \subseteq  \fA$ is the zero ideal and so the $\ca$-algebra $\fA=\varinjlim (\fA_n, \rho_n)$ together with the inclusion map gives the $\ca$-envelope of $\A$.
 
\begin{definition} \label{standard}
Let $\{e_{ij}\}_{i,j=1}^{n}$ denote the usual matrix unit system of the algebra $M_n(\bbC)$ of $n \times n$ complex matrices. An embedding $\sigma:M_n(\bbC) \rightarrow M_{mn}(\bbC)$ is said to be \textit{standard} if it satisfies $\sigma(e_{ij})=\sum_{k=0}^{m-1} e_{i +kn, j+kn}$, for all $i,j$. That is, 
\[
\sigma(A) =  \underbrace{A \oplus \dots \oplus A}_{m} \in M_{mn}(\bbC), \ \forall A\in M_n(\bbC).
\]
\end{definition}

\begin{example} \label{firstSmithrepn}
Let $\A_{\sigma} = \varinjlim ( \A_n, \sigma_n)$ be a standard limit algebra, i.e., each $\A_n$ is isomorphic to the $k_n \times k_n$ upper triangular matrices $\T_{k_n} \subseteq M_{k_n}(\bbC)$ and $\sigma_n: M_{k_n}(\bbC)\rightarrow M_{k_{n+1}}(\bbC)$ are the standard embeddings. Let $\fA_{\sigma} = \cenv(\A_{\sigma})$ be the associated UHF $\ca$-algebra.

For each $ z \in \bbT$, we define an automorphism $\psi_z: \fA_{\sigma} \rightarrow \fA_{\sigma}$, which acts on matrix units as $\psi_z(e_{i j}^{n_k})= z^{j-i}e_{i j}^{n_k}$. Assume further that $z = e^{2\pi i\theta}$, with $\theta \in [0, 1)$ irrational. We denote the corresponding crossed product $\ca$-algebra as $\fA_{\sigma}\rtimes_{\theta} \bbZ$ and the associated non-selfadjoint algebras as $\A_{\sigma}\rtimes_{\theta} \bbZ^+$ and $\A_{\sigma}\rtimes_{\theta} \bbZ$. These are analogues of the familiar irrational rotation $\ca$-algebras and their non-selfadjoint counterparts.
\end{example}

Of course, there is nothing special in this discussion about the standard embedding. If $\fA_{\sigma} = \varinjlim ( \fA_n, \rho_n)$ is any other presentation of $\fA_{\sigma}$ via regular embeddings, then one has a commutative diagram
\[
\begin{CD}
\fA_1 @>\rho_1>> \fA_2 @>\rho_2>>\fA_3@>>>  \dots\phantom{....} @. \fA \\
@VVV @VVV @VVV  @. @VV \Psi V \\
\fA_1 @>\sigma_1>> \fA_2 @>\sigma_2>>\fA_3@>>>  \dots\phantom{....} @. \fA
\end{CD}
\]
where the vertical maps are conjugations by permutation unitaries. The composition $\Psi^{-1}\circ \psi_{z}\circ \Psi$ allows us to define now an automorphism on the non-selfadjoint algebra $\A= \varinjlim ( \A_n, \rho_n)$, that twists each matrix unit by a (not necessarily positive) power of $z = e^{2\pi i\theta}$. (This automorphism is actually an example of a \textit{quasi-inner} automorphism, i.e., an automorphism that maps $\A_n$ onto $\A_n$, for all $n \in \bbN$.)

By Theorem~\ref{Dirichletenv}, $\cenv(\A_{\sigma}\rtimes_{\theta} \bbZ)\simeq \fA_{\sigma}\rtimes_{\theta} \bbZ$. The K-theory of that $\ca$-algebra is easy to calculate and it demonstrates how far removed $\A_{\sigma}\rtimes_{\theta} \bbZ$ is from its TAF generator.

\begin{proposition}
Let $\fA = \varinjlim (\fA_n, \rho_n)$ be an AF $\ca$-algebra and $\psi: \fA \rightarrow \fA$ a quasi-inner automorphism. Then, $K_0(\fA \rtimes_{\psi}\bbZ) = K_0(\fA)$ and $K_1(\fA \rtimes_{\psi}\bbZ) \simeq K_0(\fA)$.
\end{proposition}

\begin{proof} This follows from an application of the Pimsner-Voiculescu exact sequence
\begin{equation} \label{PV}
\begin{CD}
K_0(\fA) @>\id_* - \psi_*>> K_0 (\fA) @> i_*>> K_0 (\fA \rtimes_{\psi}\bbZ )\\
@AAA  @. @VV V \\
K_1(\fA \rtimes_{\psi}\bbZ ) @<< i_*< K_1(\fA) @<<\id_* - \psi_*< K_1 (\fA)
\end{CD}
\end{equation}
where $i: \fA \rightarrow \fA \rtimes_{\psi}\bbZ $ denotes the inclusion map. Since $\psi$ is quasi-inner, $\psi_*=\id_*$ on $K_i(\fA_n)$, $i=0,1$, $ n \in \bbN$. By the continuity of the $K_i$ functors (Theorem 6.3.2 and Proposition 8.2.7 in \cite{Rordam}), we obtain $\psi_*=\id_*$ on $K_i(\fA)$, $i=0,1$. Hence the upper $i_*$ is injective. Furthermore $K_1(\fA)=0$ \cite[Excercise 8.7]{Rordam} and so the right vertical map is the zero map. Therefore the upper $i_*$ is also surjective and so $K_0 (\fA) \simeq K_0(\fA\rtimes_{\psi} \bbZ)$. 

On the other hand the left side of (\ref{PV}) collapses to 
\begin{equation*}
\begin{CD}
0@>>>   K_1(\fA \rtimes_{\psi}\bbZ )   @>>>                           K_0(\fA) @>>> 0
\end{CD}
\end{equation*}
and so $K_1(\fA \rtimes_{\psi}\bbZ) \simeq K_0(\fA)$.
\end{proof}

By Kishimoto's Theorem \cite[Theorem 3.1]{Kis}, the $\ca$-algebra $\fA_{\sigma}\rtimes_{\theta} \bbZ$ is simple and therefore any of its representations is necessarily faithful. This allows us to give a good picture for $\A_{\sigma}\rtimes_{\theta} \bbZ$.

\begin{example}
Let $\A_{\sigma} = \varinjlim ( \A_n, \sigma_n)$, $\theta \in [0, 1]$ and $\A_{\sigma}\rtimes_{\theta} \bbZ$  be as in Example~\ref{firstSmithrepn}. 

Let $\{ e_n\}_{n \in \bbN}$ be an orthonormal basis for a Hilbert space $\H$. An operator $A\in B(\H)$ is said to be $k$-periodic if its matrix representation with respect to $\{ e_n\}_{n \in \bbN}$ consists of a $k \times k$-matrix which is repeated infinitely along the diagonal. The collection of all $k$-periodic matrices is denoted as $\A_k'$. Clearly the collection $\{ \A_{k_n}'\}_{n \in \bbN}$ is an increasing collection of finite dimensional factors that provides a faithful representation for $\fA_{\sigma}$.

Consider now the diagonal unitary operator $U_{\theta} \in B(\H)$ with $U_{\theta}e_n = e^{2\pi i \theta n} e_n$, $n \in \bbN$. Then the algebra generated by $\bigcup_{n \in \bbN} \A_{k_n}$ and $\{U_{\theta}^m\}_{m \in \bbZ}$ is isomorphic to $\A_{\sigma}\rtimes_{\theta} \bbZ$.

As we will see, the semisimplicity of $\A_{\sigma}\rtimes_{\theta} \bbZ$ is easy to establish. The same statement for $\A_{\sigma}\rtimes_{\theta} \bbZ^+$ requires more work.
\end{example}

The semisimplicity of strongly maximal TAF algebras was characterized by Donsig \cite[Theorem 4]{Don}. Donsig showed that a strongly maximal TAF algebra $\A$ is semisimple iff any matrix unit $e\in \A$ has a \textit{link}, i.e., $e\A e\neq\{0\}$ (Donsig's criterion). It is easy to see that any strongly maximal TAF algebra $\A = \varinjlim (\A_n, \rho_n)$ for which the standard embedding appears infinitely many times satisfies the above and is therefore semisimple.

\begin{definition}
Let $\A$ be a strongly maximal TAF algebra.
The dynamical system $(\A, \G, \alpha)$ is said to be \textit{linking} if for every matrix unit $e\in \A$ there exists a group element $g\in G$ such that $e\A \alpha_g(e) \neq \{0\}$.
\end{definition}

By Donsig's criterion if $\A$ is semisimple then $(\A, \G, \alpha)$ is linking. The following example shows that there are other linking dynamical systems.

\begin{example} \label{linking example}
Let $\A_n = \bbC \oplus \T_{2n}$ and define the embeddings $\rho_n : \A_n \rightarrow \A_{n+1}$ by 
\[
\rho_n(x \oplus A) \ = \ x \oplus \left[\begin{array}{ccc} x \\ & A \\ &&x \end{array}\right].
\]
Then $\A = \varinjlim \A_n$ is a strongly maximal TAF algebra that is not semisimple. Consider the following map $\psi:\A_n \rightarrow \A_{n+1}$ given by
\[
\psi(x \oplus A) \ =  \ x \oplus \left[\begin{array}{ccc} x \\ & x \\ && A \end{array}\right].
\]
You can see that  $\psi\circ\rho_n = \rho_{n+1}\circ\psi$ on $\A_n$ and so $\psi$ is a well-defined map on $\cup \A_n$. By considering that 
\[
\psi^{-1}(x \oplus A) \ = \ x \oplus \left[\begin{array}{ccc} A \\ & x \\ && x \end{array}\right]
\]
one gets $\psi\circ\psi^{-1} = \psi^{-1}\circ\psi = \rho_{n+1}\circ\rho_n$ on $\A_n$. Hence, $\psi$ extends to be an isometric automorphism of $\A$. Finally, for every $e_{i,j}^{2n} \in \A_n, i\neq j$ 
\[
e_{i,j}^{(2n)} \left[\begin{array}{ccc} 0_{2n}  \\ & 0_{2n} & e_{j,i}^{(2n)}\\ && 0_{2n} \end{array}\right] \psi^{(2n)}(e_{i,j}^{(2n)})
\]
\[
= \ \left[\begin{array}{ccc} 0_{2n} \\ & e_{i,j}^{(2n)}  \\ && 0_{2n} \end{array}\right]\left[\begin{array}{ccc} 0_{2n} \\ & 0_{2n} & e_{j,i}^{(2n)} \\ && 0_{2n} \end{array}\right]  \left[\begin{array}{ccc}  0_{2n} \\ & 0_{2n} \\ && e_{i,j}^{(2n)} \end{array}\right] 
\]
\[
= \left[\begin{array}{ccc} 0_{2n}  \\ & 0_{2n} & e_{i,j}^{(2n)}\\ && 0_{2n} \end{array}\right] .
\]
Therefore, $(\A, \bbZ, \psi)$ is a linking dynamical system.
 \end{example}

 \vspace{.05in}

The following theorem and the previous example establish that the converse of Theorem \ref{firstsemisimple} is not true in general.

\begin{theorem} \label{mainsemisimple}
 Let $\A$ be a strongly maximal TAF algebra and $\G$ a discrete abelian group. The dynamical system $(\A, \G, \alpha)$ is linking if and only if $\A \rtimes_\alpha G$ is semisimple. \end{theorem}

\begin{proof}
Assume that $(\A, \G, \alpha)$ is not linking. This means that there exists a matrix unit $e\in \A$ such that $e\A \alpha_g(e) = \{0\}$ for all $g\in \G$. For every $g\in \G$ and $a\in\A$ we have
\begin{align*}
(eaU_g)^2 &= eaU_geaU_g \\
& = ea \alpha_g(e)U_gaU_g \\
&= 0 U_gaU_g = 0.
\end{align*}
In the same way for any $g_1,\cdots, g_n\in \G$ and $a_1,\cdots, a_n\in \A$
\[
(e \sum_{i=1}^n a_i U_{g_i})^2 = 0.
\]
Therefore, $e \in \Rad \A\rtimes_\alpha \G$.

 \vspace{.12in}

Conversely, assume that $(\A, \G, \alpha)$ is linking. By way of contradiction, assume that $\Rad \A\rtimes_\alpha \G$ contains a non-zero element.
As in the proof of Theorem \ref{firstsemisimple} this implies that there is a nonzero element 
\[
a\in \A \cap \Rad \A \rtimes_\alpha \G \equiv \J.
\]
It is easy to see that $\J$ is a non-zero closed ideal of $\A$. By \cite[Theorem 4.7]{Pow}, $\J$ is inductive and so it is generated by the matrix units it contains. Hence there exists at least one non-diagonal matrix unit $e \in \J\cap \A_{r_1}$.

Start with $e_1=e$. Since $(\A, \G, \alpha)$ is linking, there exist $g_1 \in \G$, $b_1 \in \A_{r_2}$ matrix unit and summands $e_1^{r_2}, e_2^{r_2} \in \A_{r_2}$ of $e_1$, so that 
\begin{equation} \label{e2}
0\neq e\alpha_{g_1}(b_1 e) = e_1^{r_2}\alpha_{g_1}( b_1e_2^{r_2}) \equiv e_2 \in \Rad \A\rtimes_\alpha \G.
\end{equation}

 \vspace{.12in}
\noindent \textit{Claim 1:} The summands $e_1^{r_2}, e_2^{r_2} \in \A_{r_2}$ of $e_1$ are distinct.

\vspace{.1in}
\noindent 
Indeed if $f\equiv e_1^{r_2} = e_2^{r_2}$, then notice that (\ref{e2}) implies that the normalizing partial isometry $\alpha_{g_1}(b_1)$ maps inside the initial space of $f$ and $f$ maps inside the initial space of $b_1$. Hence $f\alpha_{g_1}\big(b_1f\alpha_{g_1}(b_1) \big) \neq 0$ and so 
\[
(fU_gb_1)^2 = \alpha_{g_1}\big(b_1f\alpha_{g_1}(b_1) \big)U_{g_1}^2 \neq 0.
\] 
Similarly $$f\alpha_{g_1}\Big(b_1f\alpha_{g_1}\big(b_1f\alpha_{g_1}(b_1) \big) \Big)\neq 0$$ and so $(fU_{g_1}b_1)^3\neq 0$. Continuing in this fashion we conclude that $(fU_{g_1}b_1)^n\neq 0$, for all $n \in \bbN$. However, $(fU_{g_1}b_1)^n$ is the product of a normalizing partial isometry with the unitary $U_{g_1}^n$ and so $\| (fU_{g_1}b_1)^n\|=1$, which contradicts the fact the $f \in \Rad \A\rtimes_\alpha \G$.

  \vspace{.12in}

Note that at this point we cannot conclude that $e_2$ is a matrix unit nor a sum of such units because of the generic nature of $\alpha_{g_1}$. However by multiplying $b_1$ with a suitable diagonal unitary, we may arrange so that $\alpha_{g_1}(b_1e_2^{r_2})$ is a sum of matrix units with orthogonal initial and final spaces, i.e., a pure normalizing partial isometry. Hence multiply both sides of (\ref{e2}) with a suitable diagonal unitary and so $e_2$ becomes a pure normalizing partial isometry by being a product of two such partial  isometries.
Note that $b_1$ in this situation is no longer a matrix unit but instead a normalizing partial isometry which may not be even pure.

Since $e_2 \in \Rad \A\rtimes_\alpha \G$, we can find now $g_2 \in \G$, a normalizing partial isometry $b_2 \in \A$ and summands $e_1^{r_3}, e_2^{r_3} \in \A_{r_3}$ of one of the matrix units in $e_2$ so that 
\[
0\neq e_2\alpha_{g_2}(b_2 e_2) = e_1^{r_3}\alpha_{g_2}(b_2 e_2^{r_3}) \equiv e_3 \in \Rad \A\rtimes_\alpha \G 
\]
is a sum of matrix units. Since $e_2 \in \Rad \A\rtimes_\alpha \G$ an argument identical to that of Claim 1 shows that the summands $e_1^{r_3}, e_2^{r_3}$ are distinct. And so on.

 Continuing in this fashion we obtain the sequences $\{ e_m\}_{m =1}^{\infty}$, $\{ e_1^{r_m}\}_{m =1}^{\infty}$, $\{ e_2^{r_m}\}_{m =1}^{\infty}$, $\{b_m\}_{m=1}^{\infty}$ and $\{g_m\}_{m=1}^{\infty}$ so that
\begin{equation} \label{em}
e_{m+1}= e_m\alpha_{g_m}(b_me_m) = e_1^{r_{m+1}}\alpha_{g_m}( b_me_2^{r_{m+1}}), m \in \bbN.
\end{equation}
Note that the sequences $\{ e_1^{r_m}\}_{m =1}^{\infty}$ and $\{ e_2^{r_m}\}_{m =1}^{\infty}$ are distinct term by term.
 
 In an analogy to the above construction to be understood shortly, we now construct a sequence $X_m \in  \A\rtimes_{\alpha} \G$, $m= 1, 2, \dots$, as follows
 \begin{align*}
 X_1 &=eU_{h_1}b_1 \\
 X_2&= \big(eU_{h_1}b_1\big)(e U_{h_2}b_2\big) \big(eU_{h_1}b_1\big) \\
 &\vdots \\
 X_m&=X_{m-1}\big(eU_{h_m}b_m\big) X_{m-1} , \quad m \in \bbN,
  \end{align*}
  where $h_1=g_1$ and $h_m=g_m(g_1g_2\dots g_{m-1})^{-1}$, for $m\geq2$.
  
 \vspace{.12in}
\noindent \textit{Claim 2:} $X_me=e_{m+1}U_{g_1g_2\dots g_m}, m \in \bbN$.

\vspace{.1in}
\noindent The claim is indeed true for $m=1$. Assuming its validity for $m-1$, we have that
\begin{align*}
X_m e&= \big(X_{m-1}e\big)U_{h_m}b_m \big( X_{m-1}e \big) \\
	&= e_mU_{g_1g_2\dots g_{m-1}}U_{h_m}b_m e_mU_{g_1g_2\dots g_{m-1}}\\
	&=e_mU_{g_m}b_m e_mU_{g_1g_2\dots g_{m-1}}\\
	&=e_m\alpha_{g_m}(b_me_m)U_{g_m}U_{g_1g_2\dots g_{m-1}} = e_{m+1}U_{g_1g_2\dots g_m}
\end{align*}
as desired.

  \vspace{.12in}
  
  Set
  \begin{equation} \label{Rad1}
  eB \equiv  e\left( \sum_{i=1}^\infty \frac{1}{2^i}U_{h_i}b_i\right)
 \in \Rad \A\rtimes_{\alpha} \G,
  \end{equation}
  where $h_1,h_2, \dots \in \G$ are as above.
  We will show that
  \begin{equation} \label{normestimate1}
  \| (eB)^{2^m-1} \| \geq 1/2^{2^{m+1}}, \quad m \in \bbN.
  \end{equation}
  This will imply that the spectral radius of $eB$ is
  \[
  \lim_{m\rightarrow \infty} \|(eB)\|^{1/2^m-1} \geq \lim_{m\rightarrow \infty} \left(\frac{1}{2^{2^{m+1}}}\right)^{1/2^m-1} = 1/4
  \]
  and so $eB$ is not quasinilpotent, thus contradicting (\ref{Rad1}).

   To establish (\ref{normestimate1}), fix an $m \in \bbN$ and note that $(eB)^{2^m-1}e$ can be written as an infinite sum of the form
\small \begin{equation} \label{summand1}
\begin{split}
  &\sum_{k  = (k_1, \dots   k_{2^m -1} )\in \bbN^{2^m -1}} \, \, \left(\prod_{l=1}^{2^m-1} \, \, \frac{e}{2^{k_l}}U_{h_{k_l}}b_{k_l}\right)e=\\
  &\phantom{xx} = \sum \, \, 2^{-p_k} (e U_{h_{k_1}}b_{k_1}) (e U_{h_{k_2}}b_{k_2}) \dots  (e U_{h_{k_{2^m -1}}}b_{k_{2^m -1 }})e\\
  &\phantom{xx} = \sum\, \, 2^{-p_k} e\alpha_{h_{k_1}}\big(b_{k_1}e\alpha_{h_{k_2}}\big(b_{k_2} \dots e\alpha_{h_{k_{2^m-1}}}\big(b_{k_{2^m-1}}e\big)\big)\big)U_{h_{k_1}\dots h_{k_{2^m-1}}}
   \end{split}
   \end{equation}
 where $p_k$ are suitable exponents. 
 
 Note that for a specific $k = (k_1, \dots   k_{2^m -1} )\in \bbN^{2^m -1}$ the corresponding summand in (\ref{summand1}) is 
\[
2^{-p_k}X_me= 2^{-p_k}e_{m+1}U_{g_1g_2\dots g_m}
\]
because of Claim 2 above. The complication we are facing now is that the terms appearing in (\ref{summand1}) are not necessarily \textit{positive} multiples of sums of matrix units. (This is the case for instance when the automorphisms $\alpha_g$ are actually gauge automorphisms.)  In order to bypass this problem and actually show that the norm of the sum in (\ref{summand1}) is as large as the norm of each one of its terms, we capitalize on our careful choice of the partial isometries $b_m$, $ m\in \bbN$. We need to establish the following two claims.
  
   \vspace{.12in}
\noindent \textit{Claim 3:} $b_ib_j^* = 0$ for $i\neq j$.

\vspace{.1in}
\noindent
  Note that
   \begin{equation} \label{nested}
e_1^{r_2} ( e_1^{r_2})^*\geq \dots \geq e_1^{r_i}( e_1^{r_i})^*\geq e_1^{r_{i+1}}( e_1^{r_{i+1}})^* \geq \dots
 \end{equation}
Since $b_i^*b_{i} \leq e_2^{r_{i+1}}(e_2^{r_{i+1}})^*$ and $e_2^{r_{i+1}}\neq e_1^{r_{i+1}}$ we have 
\[
b_i^*b_{i} \perp e_1^{r_{i+1}}( e_1^{r_{i+1}})^*
\] 
and so by (\ref{nested})
\begin{equation} \label{oneleg}
b_i^*b_{i} \perp e_1^{r_{i +l}}(e_1^{r_{i+l}})^*, \quad l = 1, 2, \dots.
\end{equation}
 On the other hand
 \[
 b_i^*b_{i} \leq e_2^{r_{i+1}}(e_2^{r_{i+1}})^* \leq e_ie_i^* \leq e_1^{r_i}(e_1^{r_i})^*
 \]
 and so  replacing $i$ with $i+l$ in the above, we obtain
 \begin{equation} \label{otherleg}
 b_{i+l}^*b_{i+l}  \leq e_1^{r_{i +l}}(e_1^{r_{i+l}})^*, \quad l = 1, 2, \dots.
  \end{equation}
  By (\ref{oneleg}) and (\ref{otherleg}), $b_{i+l}^*b_{i+l}  \perp  b_i^*b_{i} $, $l = 1,2, \dots$, which proves the claim.

   \vspace{.12in}
\noindent \textit{Claim 4:} Different choices for the index $k = (k_1, k_2, \dots k_{2^m-1}) $ produce terms in (\ref{summand1}) with orthogonal domains.

\vspace{.1in}
\noindent We will establish this for the case of two factors and will leave the details of the general case to the reader.

Indeed let
\[
X= eU_{h_{k_1}}b_{k_1} eU_{h_{k_2}}b_{k_2}e \ \ {\rm and} \ \ Y=  eU_{h_{l_1}}b_{l_1} eU_{h_{l_2}}b_{l_2}e
\]
 and assume that $XY^* \neq 0$. Since $b_k1$ is a normalizing partial isometry there exists a projection $p \in \A^* \cap \A$ so that $b_{k_2}ee^*= pb_{k_2}$. Then,
\begin{align*}
XY^*&=  eU_{h_{k_1}}b_{k_1} eU_{h_{k_2}}b_{k_2}ee^*b_{l_2}^*U_{h_{l_2}}^*e^*b_{l_1}^*U_{h_{l_1}}^*e^*\\
&=eU_{h_{k_1}}b_{k_1} eU_{h_{k_2}}pb_{k_2}b_{l_2}^*U_{h_{l_2}}^*e^*b_{l_1}^*U_{h_{l_1}}^*e^*.
\end{align*}
Since, $XY^*\neq 0$, Claim 3 implies that $k_2 = l_2$. Hence, 
\[
p'= eU_{g_{k_2}}pb_{k_2}b_{k_2}^*U_{g_{l_2}}^*e^* \in \A^* \cap \A
\] 
is a diagonal projection. Now there exists a projection $p'' \in \A^* \cap \A$ so that $b_{k_1}p'= p''b_{k_1}$. Hence
\begin{align*}
XY^* & = eb_{k_1}U_{g_{k_1}} p'U_{g_{l_1}}^*b_{l_1}^*e^*
\\ &= eU_{g_{k_1}}p''b_{k_1}b_{l_1}^*U_{g_{l_1}}^*e^*
\end{align*}
Another application of Claim 3 implies $k_1=l_1$, as desired.

  \vspace{.12in}

 Claim 4 shows now that $\| (eB)^{2^m-1}e\|$ is at least as large as the norm of each non-zero term in (\ref{summand1}). This shows now that 
\begin{equation} \label{one more}
\begin{split}
\| (eB)^{2^m-1}\|& \geq \| (eB)^{2^m-1}e\|\geq  \|a_{g_1g_2\dots g_n}U_{g_1g_2\dots g_n}\|\\
                   &\geq \|2^{-p_k}e_{m+1}U_{g_1g_2\dots g_{m}}\| \\
                      &= 2^{-p_k}.
                      \end{split}
\end{equation}
Note that the multi-index $k = (k_1, \dots   k_{2^m -1} )\in \bbN^{2^m -1}$ appearing above is given by
 
 \[ \begin{array}{c}
  k_{2^{m-1}} = m \\
  k_{2^{m-2}}= k_{3\cdot 2^{m-2}} = m-1\\
   k_{2^{m-3}}= k_{3\cdot2^{m-3}}= k_{5\cdot2^{m-3}}= k_{7\cdot2^{m-3}}=m-2\\
   \cdots \cdots \cdots \\
  k_1= k_3=k_5 = \dots = k_{2^m-1}=1,
  \end{array}
\]
and so
\[
p_{k}=m+2(m-1) +2^2(m-2)+\dots +2^{m-1} = 2^{m+1}-m -2 \]
by an easy inductive argument. Hence $p_k \leq 2^{m+1}$, $ m \in \bbN$. This fact, together with (\ref{one more}),  implies (\ref{normestimate1}) and leads to the the desired contradiction. Hence $\A\rtimes_{\alpha} \G$ is semisimple.
 \end{proof}
 
 \begin{remark} \label{promise}
 In order to show that the converse of Theorem~\ref{firstsemisimple} fails, it suffices to prove Theorem~\ref{mainsemisimple} only in the case where the automorphisms $\alpha_g$ map matrix units to sums of matrix units, as Example~\ref{linking example} clearly indicates. As it turns out, the proof of Theorem~\ref{mainsemisimple} simplifies considerably in that case. 
  \end{remark}

 If we specialize the automorphisms or the algebras in the previous result we do have the converse of Theorem \ref{firstsemisimple}.

 \begin{corollary}\label{quasi-inner}
 Let $\A$ be a strongly maximal TAF algebra and $\G$ a discrete abelian group acting on $\A$ by quasi-inner isometric automorphisms. $\A$ is semisimple if and only if $\A \rtimes_\alpha \G$ is semisimple.
 \end{corollary}
 \begin{proof}
 A quasi-inner automorphism acts on a matrix unit $e$ by multiplying $e$ with some unimodular scalar. By Donsig's criterion, this fact implies that $(\A, \G, \alpha)$ is linking if and only if $\A$ is semisimple.
 \end{proof}

 A strongly maximal TAF algebra $\A = \varinjlim (\A_n, \rho_n) $ is said to be a TUHF algebra if each one of the finite dimensional algebras $\A_n$ is isomorphic to the upper triangular matrices  $\T_{k_n} \subseteq M_{k_n}(\bbC)$, $n \in \bbN$.
 
 \begin{theorem}\label{TUHFsemisimple}
 Let $(\A, \G, \alpha)$ be a dynamical system with $\A$ a strongly maximal TUHF algebra and $\G$ a discrete abelian group. $\A$ is semisimple if and only if $\A \rtimes_\alpha \G$ is semisimple.
 \end{theorem}
 \begin{proof}
 In light of Theorems \ref{firstsemisimple} and \ref{mainsemisimple} we only need to establish that $(\A, \G, \alpha)$ linking implies that $\A$ is semisimple. This is accomplished by careful bookkeeping of indices.
 
 Assume that $(\A, \G, \alpha)$ is a linking dynamical system with $\A$ not semisimple. By Donsig's criterion there is a matrix unit $e\in \T_n$ which cannot be linked in $\A$, i.e., $e\A e = \{0\}$. Therefore if $e_1^{(n)}, e_n^{(n)}$ are the first and last diagonal matrix units in $\T_n$, then $e_n^{(n)}\A e_1^{(n)} = \{0\}$ or otherwise by multiplying $e_n^{(n)}\A e_1^{(n)} $ from the left and right with appropriate matrix units in $\T_n$ we would get $e\A e \neq \{0\}$. Since $e_n^{(n)}\A e_1^{(n)} = \{0\}$ we get $e_{1,n}^{(n)}\A e_{1,n}^{(n)} = \{0\}$ and so $e_{1,n}^{(n)} \in \Rad \A$ because it generates a nilpotent ideal.
 
  \vspace{.12in}
\noindent \textit{Claim 1:} There exists an $n_1\in \bbN$ and an index $1 < k < n_1$ such that 
\[
e_{n_1}^{(n_1)}\A e_k^{(n_1)} \ = \ e_k^{(n_1)}\A e_1^{(n_1)} \ = \ \{0\}.
\]

\vspace{.1in} 
By linking there exists a $g_1 \in \G$ such that $e_{1,n}^{(n)}\A \alpha_{g_1}(e_{1,n}^{(n)}) \neq \{0\}$ which is the same as $e_n^{(n)}\A \alpha_{g_1}(e_1^{(n)}) \neq \{0\}$.
By inductivity there exists an $n_1\in \bbN$ such that $e_n^{(n)}\T_{n_1} \alpha_{g_1}(e_1^{(n)}) \neq \{0\}$ and $\alpha_{g_1}(\T_n) \subset \T_{n_1}$. Hence, 
\begin{align*}
e_1^{(n)} = \sum_{i=1}^{n_1/n} e_{j_i}^{(n_1)}, & & \alpha_{g_1}(e_1^{(n)}) = \sum_{i=1}^{n_1/n} e_{j_i'}^{(n_1)}, 
\\  e_n^{(n)} = \sum_{i=1}^{n_1/n} e_{l_i}^{(n_1)}, & \ \ \ \ \ \ \ \ \  \ \ {\rm and} \ &\alpha_{g_1}(e_n^{(n)}) = \sum_{i=1}^{n_1/n} e_{l_i'}^{(n_1)},
\end{align*}
where $1=j_1 < \cdots < j_{n_1/n}$, $1=j_1' < \cdots < j_{n_1/n}'$, $l_1 < \cdots < l_{n_1/n} = n_1$ and $l_1' < \cdots < l_{n_1/n}' = n_1$.
Now 
\[
e_n^{(n)}\A e_1^{(n)} = \{0\} \ \Rightarrow \ e_{l_1}^{(n_1)}\A e_1^{(n_1)} = \{0\}, \ \ {\rm and}
\]
\[
e_n^{(n)}\A e_1^{(n)} = \{0\} \ \Rightarrow \ \alpha_{g_1}(e_n^{(n)})\A \alpha_{g_1}(e_1^{(n)}) = \{0\} \ \Rightarrow \ e_{n_1}^{(n_1)}\A e_{j_{n_1/n}'}^{(n_1)} = \{0\}.
\]
As well, 
\[
e_n^{(n)}\T_{n_1} \alpha_{g_1}(e_1^{(n)}) \neq \{0\} \ \Rightarrow \ l_1 \leq j_{n_1/n}'.
\]
Finally, let $k = l_1$. We already have $e_k^{(n_1)}\A e_1^{(n_1)} = \{0\}$ and note that
\[
e_{n_1}^{(n_1)}\A e_{j_{n_1/n}'}^{(n_1)} = \{0\} \ \Rightarrow \ e_{n_1}^{(n_1)}\A e_{k}^{(n_1)} = e_{n_1}^{(n_1)}\A e_{l_1, j_{n_1/n}'}^{(n_1)}e_{j_{n_1/n}'}^{(n_1)} = \{0\}.
\]
Therefore, the claim is verified.

\vspace{.12in}
\noindent \textit{Claim 2:} Suppose $\rho : \T_{m_1} \rightarrow \T_{m_2}$ is a unital regular $*$-extendable embedding.
If $\rho(e_k^{(n_1)}) = \sum_{i=1}^{n_2/n_1} e_{k_i}^{(n_2)}$ with $k_1 < \cdots < k_{n_2/n_1}$ then $k_1 \leq (k-1) n_2/n_1 + 1$ and $k_{n_2/n_1} \geq k n_2/n_1$. 

\vspace{.1in} 
This follows from the ordered partition theory of \cite{Ram} due to the rigid structure of such embeddings.

\vspace{.12in} Let $n_1, k$ be those found in Claim 1.
By linking there exists $g_2\in \G$ such that $e_{1, n_1}^{(n_1)}\A \alpha_{g_2}(e_{1,n_1}^{(n_1)}) \neq \{0\}$. Thus, there exists $n_2\in \bbN$ such that
$e_{n_1}^{(n_1)}\T_{n_2}\alpha_{g_2}(e_1^{(n_1)}) \neq \{0\}$ and 
\[
\alpha_{g_2}(e_1^{(n_1)}) = \sum_{i=1}^{n_2/n_1} e_{j_i'}^{(n_2)} \ \ e_{n_1}^{(n_1)} = \sum_{i=1}^{n_2/n_1} e_{l_i}^{(n_2)}
\]
\[
e_k^{(n_1)} = \sum_{i=1}^{n_2/n_1} e_{k_i}^{(n_2)}, \ \ \alpha_{g_2}(e_k^{(n_1)}) = \sum_{i=1}^{n_2/n_1} e_{k_i'}^{(n_2)},
\]
where the indices are again in increasing order.
Now
\[
e_{n_1}^{(n_1)}\A e_k^{(n_1)} = \{0\}  \ \Rightarrow \ k_{n_2/n_1} < l_1, \ {\rm and}
\]
\[
e_k^{(n_1)}\A e_1^{(n_1)} = \{0\} \ \Rightarrow \ \alpha_{g_2}(e_k^{(n_1)})\A \alpha_{g_2}(e_1^{(n_1)}) = \{0\} \ \Rightarrow \ j_{n_2/n_1}' < k_1'.
\]
By $e_{n_1}^{(n_1)}\T_{n_2}\alpha_{g_2}(e_1^{(n_1)}) \neq \{0\}$, Claim 2 and the above inequalities we have that
\[
k n_2/n_1 \leq k_{n_2/n_1} < l_1 \leq j_{n_2/n_1}' < k_1' \leq (k-1)n_2/n_1 - 1,
\]
which is a contradiction.
Therefore, if $(\A, \G, \alpha)$ is linking then $\A$ is semisimple.
 \end{proof}
 \section{Crossed products by compact abelian groups} \label{compact abelian semisimple}
 
 Our previous results on the semimplicity of crossed products by discrete abelian groups raise the question of what happen in other cases. Here we address the semisimplicity of crossed products by compact abelian groups. Remarkably the situation reverses. The key ingredient in our study is non-selfadjoint Takai duality.
 
 We need the following.
 
 \begin{lemma} \label{tensorcompact}
Let $\A$ be an operator algebra and let $\K (\H)$ denote the compact operators acting on a Hilbert space $\H$. If $\A \otimes \K(\H)$ is semisimple, then $\A$ is semisimple.
\end{lemma}

\begin{proof}
Let $\{ \xi_i \}_{i \in \bbI}$ be an orthonormal basis for $\H$ and let $e_{i j}$ be the rank one operator mapping $\xi_j$ to $\xi_i$, $ i, j \in \bbI$. Assume that $\A$ acts on some Hilbert space $\H'$. Then $\A \otimes \K(\H)$ is generated as an operator algebra by all elementary tensors of the form $a \otimes e_{ ij} \in B(\H' \otimes \H)$, $ a \in \A$, $ i, j \in \bbI$.

By way of contradiction, assume that $0 \neq x \in \Rad \A$. Fix an $i_0 \in \bbI$ and let $X \equiv x\otimes e_{i_0 i_0} \in \A \otimes \K(\H)$. An easy calculation shows that given $A \in \A \otimes\K(\H)$, there exists $a \in \A$ so that 
\[
a\otimes e_{i_0 i_0}  = (I \otimes e_{i_0 i_0}) A (I \otimes e_{i_0 i_0} )
\]
and so 
\begin{align*}
(AX)^n &= A(x\otimes  e_{i_0 i_0})  \Big( (I \otimes e_{i_0 i_0}) A (I \otimes e_{i_0 i_0} ) (x\otimes e_{i_0 i_0})\Big)^{n-1} \\
    &=A(x\otimes  e_{i_0 i_0})  \big( (ax)^{n-1}\otimes e_{i_0 i_0}\big)
\end{align*}
for all $n \in \bbN$.
Hence 
\[
\begin{split}
\lim_n \|(AX)^n\|^{1/n}&\leq \lim_n \| A(x\otimes  e_{i_0 i_0}) \|^{1/n} \limsup_n  \|(ax)^{n-1}\|^{1/n} \\
                                               &=\limsup_n  \|(ax)^{n}\|^{1/n} =0
                                               \end{split}
\]
because $x \in \Rad \A$. Hence $0 \neq X \in \Rad \A\otimes \K( \H)$, which is the desired contradiction.
\end{proof}

\begin{theorem} \label{Tsemis}
Let $(\A, \G , \alpha)$ be a dynamical system, with $\G$ a compact abelian group. If $\A \rtimes_{\alpha} \G$ is semisimple, then $\A$ is semisimple.
\end{theorem}

\begin{proof} Assume that $\A \rtimes_{\alpha} \G$ is semisimple. Then Theorem~\ref{firstsemisimple} implies that $\big(\A \rtimes_{\alpha}\G \big)\rtimes_{\hat{\alpha}} \hat{\G}$ is semisimple. By Takai duality, $\A \otimes \K \big( L^2(\G)\big)$ is semisimple and so by Lemma ~\ref{tensorcompact}, $\A$ is semisimple, as desired
\end{proof}

Let us see now that the converse of the above theorem is not necessarily true. Therefore, Theorem~\ref{firstsemisimple} does not extend beyond discrete abelian groups.

\begin{example} \label{Texam}
\textit{A dynamical system $(\B, \bbT , \beta)$, with $\B$ a semisimple operator algebra, for which $\B \rtimes_{\beta} \bbT$ is not semisimple.}

We will employ again our previous results and Takai duality. In Example~\ref{linking example} we saw a linking dynamical system $(\A, \bbZ, \alpha)$ for which $\A$ is not semisimple. Since $(\A, \bbZ, \alpha)$ is linking, we have by Theorem~\ref{mainsemisimple} that the algebra $\B \equiv \A\rtimes_{\alpha} \bbZ$ is semisimple. Let $\beta\equiv \hat{\alpha}$. Then,
\[
\B \rtimes_{\beta} \bbT =\big(\A \rtimes_{\alpha}\bbZ\big) \rtimes _{\hat{\alpha}}\bbT \simeq \A\otimes \K(\ell^2(\bbZ)),
\]
which is not semismple, by Lemma~\ref{tensorcompact}.
\end{example}
\section{More examples of crossed product Dirichlet algebras}

In Chapter~\ref{Dirichlet Sect} we promised additional examples of crossed products which are Dirichlet algebras and yet fail to be isometrically isomorphic to any tensor algebra. We remind the reader that the existence of such algebras was an open problem in \cite{DKDoc} that was only solved recently by Kakariadis \cite{Kak}. 

\begin{definition}
Let $\A=  \varinjlim (\A_n, \rho_n)$ be a strongly maximal TAF algebra and let $\A_0 \equiv \varinjlim (\Rad \A_n, \rho_n) \subseteq \A$. We say that $\A$ is fractal-like if $\A_0 = [\A_0 ^2\,]$.
\end{definition}

The familiar refinement and alternation limit algebras \cite{Pow} are examples of fractal-like limit algebras.

 \begin{theorem}  \label{infinite div}
Let $\A$ be a strongly maximal TAF algebra and let $\psi: \A \rightarrow \A$ be an isometric quasi-inner automorphism. If $\A$ is fractal-like, then $\A \rtimes_{\psi} \bbZ$ is a Dirichlet algebra which is not isometrically isomorphic to the tensor algebra of any $\ca$-correspondence.
\end{theorem}

\begin{proof} Note that
\[
(\A \rtimes_{\psi} \bbZ )\cap (\A \rtimes_{\psi} \bbZ)^* = \{ \sum_{i = -\infty}^{\infty} c_iU^i \mid c_i \in \ \A \cap \A^*, \, \, i \in \bbZ \}.
\]
Since $\psi$ is quasi-inner, $\A \cap \A^*$ is left elementwise invariant by $\psi$ and so $(\A \rtimes_{\psi} \bbZ )\cap (\A \rtimes_{\psi} \bbZ)^*$ is a commutative $\ca$-algebra.

The conclusion will follow if we verify that an operator algebra $\B$ containing a copy of a fractal-like TAF algebra $\A= \varinjlim (\A_n, \rho_n)$, cannot be isometrically isomorphic to the tensor algebra of a \textit{commutative} $\ca$-algebra $\C$.

By way of contradiction, assume that there exists a $\C$-bimodule $X$ so that each element $b \in \B$ admits a Fourier series $b =  c +\sum_{j=1}^{\infty} \,  \xi_j$, with $c \in \C$ and $\xi_j \in X^j$, $j=1,2, \dots$. Note that if $ e \in A_n$ is any off-diagonal matrix unit then the $\C$-coefficient in its Fourier series is equal to $0$, since such an $e$ is nilpotent of order $2$. Let $j_0$ be the smallest positive integer so that $e =  \sum_{j=j_0}^{\infty} \,  \xi_j$, for some off diagonal matrix unit $e$. However $e$ can be written as a finite sum of products of the form $e=e_1e_2$, where $e_1, e_2 \in \A$ are off-diagonal matrix units. But the minimality of $j_0$ implies that each product
$e_1e_2$ has a Fourier series starting from $2j_0$, which is a contradiction.
\end{proof}

It is worthwhile noticing that the above arguments also show that any fractal-like strongly maximal TAF algebra cannot be isomorphic to the tensor algebra of a $\ca$-correspondence. This theme has been further explored in \cite{KR}, where we characterize all triangular limit algebras that happen to be isometrically isomorphic to tensor algebras of $\ca$-correspondences.
 
 \section{Semicrossed products and semisimplicity}
 
 It is instructive to see what happens in the semicrossed product case. This can be taken as further evidence that the crossed product is perhaps a nicer non-selfadjoint object than the semicrossed product. 
 
 Let $(\A, \G, \alpha)$ be a dynamical system with $\G$ a discrete abelian group. Suppose $P$ is a positive spanning cone of $\G$, that is, $P$ is a unital semigroup such that $P \cap P^{-1} = \{1\}$ and $PP^{-1} = \G$, using multiplicative notation.
 
Define the (unitary) semicrossed product of the dynamical system \\ $(\A, P, \alpha)$ as 
\[
\A \times_\alpha P = \overline{\alg}\{aU_s : a\in \A, s\in P\}.
\]
This definition is left-right flipped from the usual one and would really be the definition for the unitary semicrossed product of $(\A, P^{-1}, \alpha)$. 
Another important note is that by \cite[Theorem 3.3.1]{DFK} this semicrossed product is completely isometrically isomorphic to the isometric semicrossed product.

There is no version of Theorem \ref{firstsemisimple} as it is no longer true in this context. To see this we again turn to strongly maximal TAF algebras.

\begin{definition}
Let $\A$ be a strongly maximal TAF algebra.
The dynamical system $(\A, P, \alpha)$ is said to be \textit{linking} if for every matrix unit $e\in \A$ and every $t\in P$ there exists an $s\in P$ such that $e\A \alpha_{st}\big(e\big) \neq \{0\}$.
\end{definition}

\begin{proposition}
Let $(\A, P, \alpha)$ be a dynamical system with $P$ totally ordered. If for every matrix unit $e\in \A$ there is an $s\in P\setminus \{1\}$ such that $e\A \alpha_{s}(e)\neq \{0\}$ then $(\A, P, \alpha)$ is linking.
\end{proposition}
\begin{proof}
Let $e\in \A$ be a matrix unit. By hypothesis there exists $s_1 \in P\setminus \{1\}$ such that $e \A \alpha_{s_1}(e) \neq \{0\}$. This is an inductive object, hence there exists $f_1\in \A$ such that $e f_1 \alpha_{s_1}(e)$ is a matrix unit. Again by the hypothesis, there exists $s_2\in P\setminus \{1\}$ such that
\[
\{0\} \ \neq \ e_1f_1\alpha_{s_1}(e_1)\A \alpha_{s_2}(e_1f_1\alpha_{s_1}(e_1)) \ \subset \ e_1\A \alpha_{s_2s_1}(e_1) \neq \{0\}.
\]
Repeating this argument implies that there are an infinite number of semigroup elements $s \in P$ such that $e \A \alpha_{s}(e) \neq \{0\}$. Therefore, for $t\in P$, discrete and totally ordered imply that there exists $s\in P$ such that $st$ is a semigroup element in this infinite set. Hence, $e \A \alpha_{st}(e) \neq \{0\}$.
\end{proof}

Note that if $\A$ is semisimple then $(\A, P, \alpha)$ is not necessarily linking. In particular, consider the following example.

\begin{example}
 Let
 \[
 \A_n = \bbC \oplus \T_2 \oplus \cdots \oplus \T_{2^{n-2}} \oplus \T_{2^{n-1}} \oplus \T_{2^{n-2}} \oplus \cdots \oplus \T_2 \oplus \bbC
 \]
  and define the embeddings $\rho_n: \A_n \rightarrow \A_{n+1}$ by
  \[
  \rho_1(A_1) = A_1 \oplus \left[\begin{array}{cc} A_1 \\ & A_1\end{array}\right] \oplus A_1 = A_1 \oplus (I_2 \otimes A_1) \oplus A_1
  \]
  and for $n\geq 2$
 \[
\rho_n\left(\bigoplus_{i=1}^{2n-1} A_i\right) \ = \ A_1 \oplus \left(\bigoplus_{i=1}^{2n-1} I_2\otimes A_i \right) \oplus A_{2n-1}.
 \]
   These embeddings are associated with the following Bratteli diagram

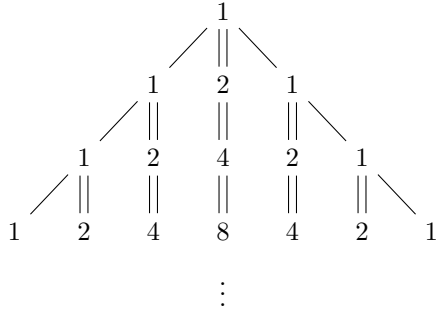
\begin{center}
\begin{tikzpicture}[node distance=0.5cm]

\node (Top) {1};
\node (midrow1) [below=of Top] {2};
\node (leftrow1) [left=of midrow1] {1};
\node (rightrow1) [right=of midrow1] {1};

\node (midrow2) [below=of midrow1] {4};
\node (leftmidrow2) [left=of midrow2] {2};
\node (rightmidrow2) [right=of midrow2] {2};
\node (leftrow2) [left=of leftmidrow2] {1};
\node (rightrow2) [right=of rightmidrow2] {1};

\node (midrow3) [below=of midrow2] {8};
\node (leftmidrow3) [left=of midrow3] {4};
\node (leftleftmidrow3) [left=of leftmidrow3] {2};
\node (leftrow3) [left=of leftleftmidrow3] {1};
\node (rightmidrow3) [right=of midrow3] {4};
\node (rightrightmidrow3) [right=of rightmidrow3] {2};
\node (rightrow3) [right=of rightrightmidrow3] {1};

\node (ellipsis) [below=3pt of midrow3] {$\vdots$};

\draw[transform canvas={xshift=-1.5pt}] (Top) -- (midrow1);
\draw[transform canvas={xshift=1.5pt}] (Top) -- (midrow1);
\draw (Top) -- (leftrow1);
\draw (Top) -- (rightrow1);

\draw (leftrow1) -- (leftrow2);
\draw[transform canvas={xshift=-1.5pt}] (leftrow1) -- (leftmidrow2);
\draw[transform canvas={xshift=1.5pt}] (leftrow1) -- (leftmidrow2);
\draw[transform canvas={xshift=-1.5pt}] (midrow1) -- (midrow2);
\draw[transform canvas={xshift=1.5pt}] (midrow1) -- (midrow2);
\draw[transform canvas={xshift=-1.5pt}] (rightrow1) -- (rightmidrow2);
\draw[transform canvas={xshift=1.5pt}] (rightrow1) -- (rightmidrow2);
\draw (rightrow1) -- (rightrow2);

\draw (leftrow2) -- (leftrow3);
\draw[transform canvas={xshift=-1.5pt}] (leftrow2) -- (leftleftmidrow3);
\draw[transform canvas={xshift=1.5pt}] (leftrow2) -- (leftleftmidrow3);
\draw[transform canvas={xshift=-1.5pt}] (leftmidrow2) -- (leftmidrow3);
\draw[transform canvas={xshift=1.5pt}] (leftmidrow2) -- (leftmidrow3);
\draw[transform canvas={xshift=-1.5pt}] (midrow2) -- (midrow3);
\draw[transform canvas={xshift=1.5pt}] (midrow2) -- (midrow3);
\draw[transform canvas={xshift=-1.5pt}] (rightmidrow2) -- (rightmidrow3);
\draw[transform canvas={xshift=1.5pt}] (rightmidrow2) -- (rightmidrow3);
\draw[transform canvas={xshift=-1.5pt}] (rightrow2) -- (rightrightmidrow3);
\draw[transform canvas={xshift=1.5pt}] (rightrow2) -- (rightrightmidrow3);
\draw (rightrow2) -- (rightrow3);

\end{tikzpicture}
\end{center}

 Then $\A = \varinjlim \A_n$ is a semisimple strongly maximal TAF algebra. However, consider the following shift-like map $\psi: \A \rightarrow \A$ which takes $\A_n$ into $\A_{n+1}$ by
 \begin{align*}
 \psi\Big(\bigoplus_{i=1}^{2n-1}  A_i\Big) \ = & A_1 \oplus (I_2 \otimes A_1) \oplus (I_4 \otimes A_1) \oplus (I_4 \otimes A_2) \otimes \cdots 
 \\ &\oplus (I_4 \otimes A_{n-1}) 
  \oplus A_{n} \oplus A_{n+1} \oplus \cdots \oplus A_{2n-1}.
 \end{align*}
 This is well defined with the $\rho_n$ embeddings and thus we define
\begin{align*}
  \psi^{-1}\Big(\bigoplus_{i=1}^{2n-1}  & A_i  \Big)   = A_1 \oplus A_2 \oplus \cdots \oplus A_{n} \oplus (I_4\otimes A_{n+1}) \oplus \cdots \quad \phantom{c}
 \\ &  \oplus (I_4 \otimes A_{2n-3})
     \oplus (I_4 \otimes A_{2n-2}) \oplus (I_2 \otimes A_{2n-1}) \oplus A_{2n-1}.
 \end{align*}
From these definitions we calculate that
\[
\psi^{-1}\circ\psi\Big(\bigoplus_{i=1}^{2n-1} A_i\Big) \ = \ \rho_{n+1}\circ\rho_n\Big(\bigoplus_{i=1}^{2n-1} A_i\Big).
\]
Thus, $\psi$ is an isometric automorphism of $\bigcup_{n=1}^\infty \A_n$ and so extends to an isometric automorphism of $\A$. 
 
 Now consider $e_{1,2} \in \T_2 \subset \A_2$. It is immediate that $e_{1,2}\A\psi^{(k)}(e_{1,2}) = \{0\}$ for all $k\geq 1$.
 Therefore, $(\A, \bbZ^+, \psi)$ is not linking even though $\A$ is semisimple.
 \end{example}

\begin{theorem}\label{semigroupsemisimple}
 Let $\A$ be a strongly maximal TAF algebra and $P$ a semigroup that is a positive spanning cone of a discrete abelian group. The dynamical system $(\A, P, \alpha)$ is linking if and only if $\A \times_\alpha P$ is semisimple.
\end{theorem}

\begin{proof}
Assume that $(\A, P, \alpha)$ is not linking. This means that there exists a matrix unit $e\in \A$ and a $t\in P$ such that $e\A \alpha_{st}(e) = \{0\}$ for all $s\in P$. For every $s\in P$ and $a\in \A$ we have
\begin{align*}
(eU_taU_s)^2 &= eU_taU_seU_taU_s
\\ &= e\alpha_{t}\big(a\big)\alpha_{st}\big(e\big)\alpha_{st^2}\big(a\big)U_{s^2t^2}
\\ &= 0\alpha_{st^2}\big(a\big)U_{s^2t^2}
\\ &= 0.
\end{align*}
In the same way for any $s_1,\dots, s_n \in P$ and $a_1,\cdots, a_n\in \A$
\[
(eU_t\sum_{i=1}^n a_iU_{s_i})^2 = 0.
\]
Therefore, $eU_t \in \Rad \A \times_\alpha P$ and so the semicrossed product is not semisimple.

Conversely, suppose that $(\A, P, \alpha)$ is linking. This will follow in a similar manner as the proof of the converse in Theorem \ref{mainsemisimple}. One only needs to be careful at a few points since we are dealing with a semigroup instead of a group.

Assume that $\A \times_\alpha P$ is not semisimple. Thus, there is a non-zero $a\in \Rad \A \times_\alpha P$. 
Since we are working in a discrete abelian group we can use the Fourier theory discussed after Proposition \ref{Ceasaroapprox}. In light of this, let $\G = PP^{-1}$ and $\hat{\G}$ the Pontryagin dual of $\G$. The gauge actions $\{\psi_\gamma\}_{\gamma\in \hat{\G}}$ restrict to gauge automorphisms on $\A \times_\alpha P$ and so ideals in this algebra are left invariant by the gauge actions.
Hence, $\Rad \A \times_\alpha P$ is a closed linear space in $\A \times_\alpha P \subset \A \times_\alpha \G$,  which is left invariant by the gauge action $\{\psi_\gamma\}_{\gamma\in \hat{\G}}$. Therefore, $a_sU_s = \Phi_s(a) \in \Rad \A \times_\alpha P$ for all $s\in P$ (being careful to note that this $\Phi_s$ was defined differently).

By Proposition \ref{Ceasaroapprox} there exists $s_0\in P$ such that $\A U_{s_0} \cap \Rad \A \times_\alpha P \neq \{0\}$. This set is inductive and so there exists a matrix unit $e\in \A_{r_1}$ such that $eU_{s_0}$ is in the radical.

Start now with $e_1 \equiv e\in \A_{r_1}$. By linking there exists $s_1'\in P$ such that $e \A \alpha_{s_1's_0}(e) \neq \{0\}$. Define $s_1 = s_1's_0 \in P$. By inductivity there is a $b_1\in \A_{r_2}$ such that $e_1\alpha_{s_1}(b_1e_1)$ is a matrix unit in $\A_{r_2}$. Because $\A$ has regular embeddings and since any isometric automorphism preserves the normalizer there exists $e_1^{r_2}, e_2^{r_2}$ summands of $e_1$ such that $e_1^{r_2}$ and $\alpha_{s_1}(e_2^{r_2})$ are matrix units in $\A_{r_2}$. This allows that $b_1$ can be taken to be a normalizing partial isometry and 
  \[
 e_2 \equiv e_1\alpha_{s_1}(b_1e_1) = e_1^{r_2}\alpha_{s_1}(b_1e_2^{r_2}).
  \]
  If $e_1^{r_2} = e_2^{r_2} \equiv f$ then notice that $f^*f = \alpha_{s_1}(b_1b_1^*)$ and $ff^* = b_1^*b_1$. This implies that 
  \begin{align*}
  (eU_{s_1}b_1)^n &= eU_{s_1}b_1eU_{s_1}b_1\cdots eU_{s_1}b_1 
  \\ &= e\alpha_{s_1}(b_1e\alpha_{s_1}(b_1e\cdots \alpha_{s_1}(b_1)))U_{s_1}^n
  \\ & = f\alpha_{s_1}(b_1f\alpha_{s_1}(b_1f\cdots \alpha_{s_1}(b_1)))U_{s_1}^n
  \end{align*}
  is a partial isometry times a unitary and so $eU_{s_0}U_{s_1'}b_1 = eU_{s_1}b_1$ is not quasinilpotent, a contradiction to $eU_{s_0}$ being in the radical. Therefore, $e_1^{r_2} \neq e_2^{r_2}$ which allows us to choose $r_2, b_1, e_1^{r_2}$ and $e_2^{r_2}$ again such that $e_1^{r_2}$ and $e_2^{r_2}$ are distinct summands of $e$. We remark for later in the proof that this gives  
  \begin{equation}\label{distinctsummandsfirstcase2}
  e_1^{r_2}(e_1^{r_2})^* \perp e_2^{r_2}(e_2^{r_2})^*.
  \end{equation}
 
 Continuing this way, we get a sequence of matrix units $\{e_m\}_{m =1}^{\infty}$,  $e_m \in \A_{r_m}$, a sequence of partial isometries $\{b_m\}_{m=1}^\infty$,  and semigroup elements $\{s_m\}_{m=1}^\infty, s_m = s_m's_m''s_0\in P$, with
  \[
  e_{m+1} \equiv e_m \alpha_{s_m}(b_me_m) = e_1^{r_{m+1}} \alpha_{s_m}(b_me_2^{r_{m+1}}) \neq 0
  \]
  where $e_{1}^{r_{m+1}}, e_{2}^{r_{m+1}}$ are summands of $e_{m}$ and
  \begin{equation}\label{semigrouptweak2}
  s_m'' = \prod_{i=1}^{m-1} s_{m-i}^i \in P.
  \end{equation}
  By linking $s_m' \in P$ is chosen such that $e_m\A \alpha_{s_m's_m''s_0}(e) \neq \{0\}$.

  Again we need to consider if $e_1^{r_{m+1}} = e_2^{r_{m+1}} \equiv f$. First, by the recursive definition of $e_m$ we have 
  \begin{align*}
 eU_{s_0}B &\equiv  eU_{s_0}\alpha_{s_0^{-1}s_1}\big(b_1e\alpha_{s_2}\big(b_2e_2\alpha_{s_3}\big(b_3 \dots b_m\big)\big)\big) U_{s_m'}  
 \\& = \alpha_{s_1s_2\dots s_{m-1}}\big(e_ms_m(b_m)\big)U_{s_m},
  \end{align*}
 noting that $s_0^{-1}s_1 = s_1' \in P$.
  Hence, 
  \begin{align*}
  (eU_{s_0}B)^n & = \left(\alpha_{s_1s_2\dots s_{m-1}}\big(e_m\alpha_{s_m}(b_m)\big)U_{s_m}\right)^n
  \\ &= \alpha_{s_1s_2\dots s_{m-1}}\big(e_m\alpha_{s_m}\big(b_me_m\alpha_{s_m}\big(b_m \dots e_m\alpha_{s_m}\big(b_m\big)\big)\big)\big)U_{s_m}^n
  \\ &= \alpha_{s_1s_2\dots s_{m-1}}\big(f\alpha_{s_m}\big(b_mf\alpha_{s_m}\big(b_m \dots f\alpha_{s_m}\big(b_m\big)\big)\big)\big)U_{s_m}^n
  \end{align*}
  is again the product of a partial isometry and a unitary and so $eU_{s_0}B$ is not quasinilpotent, a contradiction. Therefore, in the same way as before we can choose $r_{m+1}$, $b_m$, $e_1^{r_{m+1}}$ and $e_2^{r_{m+1}}$ such that 
  \begin{equation}\label{distinctsummands2}
  e_1^{r_{m+1}}(e_1^{r_{m+1}})^* \perp e_2^{r_{m+1}}(e_2^{r_{m+1}})^*.
  \end{equation}

 Set
  \begin{equation} \label{Rad2}
  eU_{s_0}B \equiv  eU_{s_0}\left( \sum_{i=1}^\infty \frac{1}{2^i}U_{t_i'}b_i\right) = e\left( \sum_{i=1}^\infty \frac{1}{2^i}U_{t_i}b_i\right)
 \in \Rad \A\times_{\alpha} P,
  \end{equation}
where the semigroup elements $t_i$ will be defined later.

  We will show that
  \begin{equation} \label{normestimate2}
  \| (eU_{s_0}B)^{2^m} \| \geq 1/2^{2^{m+1}}, \quad m \in \bbN.
  \end{equation}
  This will imply that the spectral radius of $eU_{s_0}B$ is
  \[
  \lim_{m\rightarrow \infty} \|(eU_{s_0}B)^{2^m}\|^{1/2^m} \geq \lim_{m\rightarrow \infty} \left(\frac{1}{2^{2^m+1}}\right)^{1/2^m} = 1/2
  \]
  and so $eU_{s_0}B$ is not quasinilpotent, thus contradicting (\ref{Rad2}).

   To establish this contradiction, fix an $m \in \bbN$ and note that $(eU_{s_0}B)^{2^m}$ can be written as an infinite sum of the form
 \begin{multline} \label{summand2}
   \sum_{k = (k_1, k_2, \dots   k_{2^m} )\in \bbN^{2^m}} \, \,  \big(\frac{e}{2^{k_1}} U_{t_{k_1}}b_{k_1}\big)\big(\frac{e}{2^{k_2}}U_{t_{k_2}}b_{k_2}\big)\dots \big(\frac {e}{2^{k_{2^m}}}U_{t_{k_{2^m}}}b_{k_{2^m}}\big)\\
 = \sum_{k = (k_1, k_2, \dots   k_{2^m} )\in \bbN^{2^m}}  \,\, 2^{-p_k} e\alpha_{t_{k_1}}\big(b_{k_1}e\alpha_{t_{k_2}}\big(b_{k_2} \dots e\alpha_{t_{k_{2^m}}}\big(b_{k_{2^m}}\big)\big)\big)U_{t_{k_1}\dots t_{k_{2^m}}},
 \end{multline}
 where $p_k$ are suitable exponents.

   \vspace{.12in}

The following two claims remain unchanged from the proof of Theorem \ref{mainsemisimple}.
  
   \vspace{.12in}
\noindent \textit{Claim 1:} $b_ib_j^* = 0$ for $i\neq j$.

   \vspace{.12in}
\noindent \textit{Claim 2:} Different choices for the index $k = (k_1, k_2, \dots k_{2^n}) $ produce terms in (\ref{summand2}) with orthogonal domains.

\vspace{.12in}
\noindent \textit{Claim 3:} For any $m \in \bbN$, there is a choice of indices $k_1, k_2, \dots k_{2^{m}-1}$ and group elements $t_{k_1},\cdots, t_{k_{2^m-1}} \in \G = PP^{-1}$ such that
\[
e_{m+1} = e \alpha_{t_{k_1}}\big(b_{k_1} e\alpha_{t_{k_2}}\big(b_{k_2} \dots \alpha_{t_{k_{2^m-1}}}\big(b_{k_{2^m-1}}e\big)\big)\big).
\]
\vspace{.05in}
\noindent
This follows by induction. The case $m=2$ follows from the definition of $e_2$. Assume that the claim is true for $m \in \bbN$, i.e.,
\begin{equation} \label{e_m2 }
e_m = e \alpha_{t_{k_1}}\big(b_{k_1} e\alpha_{t_{k_2}}\big(b_{k_2} \dots \alpha_{t_{k_{2^{m-1}-1}}}\big(b_{k_{2^{m-1}-1}}e\big)\big)\big).
\end{equation}
Then, for $t_{k_{2^{m-1}}} = s_mt_{k_1}^{-1}\dots t_{k_{2^{m-1}-1}}^{-1}$, remembering that $\G$ is abelian, we have 
 \begin{align*}
 e_{m+1}  &= e_m\alpha_{s_m}(b_me_m)
 \\ &= e \alpha_{t_{k_1}}\big(b_{k_1}  \dots \alpha_{t_{k_{2^{m-1}-1}}}\big(b_{k_{2^{m-1}-1}}e\big)\big)
 \\ & \phantom{XXXXXXXX} \alpha_{s_m}\big(b_me \alpha_{t_{k_1}}\big( b_{k_1} \dots \alpha_{t_{k_{2^{m-1}-1}}}\big(b_{k_{2^{m-1}-1}}e\big)\big)\big)
 \\ &= e \alpha_{t_{k_1}}\big(b_{k_1}  \dots \alpha_{t_{k_{2^{m-1}-1}}}\big(b_{k_{2^{m-1}-1}}e
 \\ & \phantom{XXXXXXXX} \alpha_{t_{k_{2^{m-1}}}}\big(b_me \alpha_{t_{k_1}}\big( b_{k_1} \dots \alpha_{t_{k_{2^{m-1}-1}}}\big(b_{k_{2^{m-1}-1}}e\big)\dots\big),
 \end{align*}
which proves the claim.

It is instructive to specify the choice of indices $k_1, k_2, \dots k_{2^{m}-1}$ appearing in Claim 3. Indeed
\[ \begin{array}{c}
  k_{2^{m-1}} = m \\
  k_{2^{m-2}}= k_{3\cdot 2^{m-2}} = m-1\\
   k_{2^{m-3}}= k_{3\cdot2^{m-3}}= k_{5\cdot2^{m-3}}= k_{7\cdot2^{m-3}}=m-2\\
   \cdots \cdots \cdots \\
  k_1= k_3=k_5 = \dots = k_{2^m-1}=1.
  \end{array}
\]
We wish to now prove that the $t_m$ are actually in $P$. To this end, note that by the recursive formula $t_i = s_i v_i^{-1}$ where $v_i\in P$.
This implies that
\begin{align*}
t_m &= s_m t_{k_1}^{-1}\cdots t_{k_{2^{m-1}-1}}^{-1} 
\\ &= s_m(s_{k_1}v_{k_1}^{-1})^{-1}\cdots (s_{k_{2^{m-1}-1}}v_{k_{2^{m-1}-1}}^{-1})^{-1}
\\ &= s_m\prod_{i=1}^{m-1} s_{m-i}^{-i} v_{m-i}^i
\\ &= s_m's_0 \prod_{i=1}^{m-1} v_{m-i}^i \in P
\end{align*}
by (\ref{semigrouptweak2}).

\vspace{.2in}

 Claim 2 shows now that $\| (eU_{s_0}B)^{2^m}\|$ is at least as large as the norm of each non-zero term in (\ref{summand2}). By Claim 3 and setting $k_{2^m} = m+1$, one of these terms is $ 2^{-p_k}e_{m+1}U_{s_{m+1}}b_{m+1}$, which is non-zero. Furthermore for this term we have
\[
p_{k}=(m+1) + m+2(m-1) +2^2(m-2)+\dots +2^{m-1} = 2^{m+1}-1
\]
by an easy telescoping argument. Hence,
\[
\|(eU_{s_0}B)^{2^m}\| \geq \|2^{-p_k}e_{m+1}U_{s_{m+1}}b_{m+1}\| = \frac{1}{2^{2^{m+1}-1}}.
\]
Using this estimate in (\ref{summand2}), we obtain (\ref{normestimate2}), which is the desired contradiction. Hence $\A\times_{\alpha} P$ is semisimple.
\end{proof}

Corollary \ref{quasi-inner} transfers with no changes in the proof to this semigroup context and Theorem \ref{TUHFsemisimple} with some changes.

 \begin{corollary}
 Let $\A$ be a strongly maximal TAF algebra and $P$ a positive spanning cone of a discrete abelian group acting on $\A$ by quasi-inner isometric automorphisms. $\A$ is semisimple if and only if $\A \times_\alpha P$ is semisimple.
 \end{corollary}

 \begin{theorem}
 Let $(\A, P, \alpha)$ be a dynamical system with $\A$ a strongly maximal TUHF algebra and $P$ a positive spanning cone of a discrete abelian group. $\A$ is semisimple if and only if $\A \times_\alpha P$ is semisimple.
 \end{theorem}
 \begin{proof}
 If $\A \times_\alpha P$ is semisimple then $(\A, P, \alpha)$ is linking by Theorem \ref{semigroupsemisimple}. Using the exact same proof as Theorem \ref{TUHFsemisimple} we get that $\A$ is semisimple.
 
 Conversely, due to the failure of Theorem \ref{firstsemisimple} in the semicrossed product case we need a different proof.
 To this end, assume that $\A$ is semisimple. Because $\A$ is a TUHF algebra Donsig's criterion can be strengthened into the fact that for any two matrix units $e,f\in \A$ we have $e\A f \neq \{0\}$. This is due to the fact that $e_{1,n}^{(n)}\A e_{1,n}^{(n)} \neq \{0\}$ which implies that $e_n^{(n)}\A e_1^{(n)} \neq \{0\}$ for all $n\in \bbN$ such that $\T_n \subset \A$. Therefore, for any matrix unit $e\in \A$ and $t\in P$ this gives that $e\A \alpha_{t}(e) \neq \{0\}$ and so $(\A,P, \alpha)$ is linking.
 \end{proof}

\chapter[Crossed Products and Tensor Algebras]{The Crossed Product as the Tensor Algebra of a C$^*$-correspondence.} \label{Hao}

There are three sources of inspiration for the results in this chapter. First we saw in Definition~\ref{relativedefn} that given a system $(\A, \G, \alpha)$ there is a whole family of crossed products, parametrized by the possible $\ca$-covers of $\A$, which we coined as relative crossed products. In Corollary~\ref{allrcoincide} we verified that all relative \textit{reduced} crossed products coincide. This raises the question if a similar result is valid for the relative (full) crossed products. Theorem~\ref{HNtensor1} indicates that this is a very delicate problem that among other things it also rubs shoulders with the validity of WEP for certain $\ca$-algebras. 

For a second inspiration recall that we have already verified that the identities
\begin{equation} \label{envidentity}
\begin{aligned}
\cenv(\A \rtimes_{\alpha} \G )&\simeq \cenv(\A) \rtimes_{\alpha} \G \\
\cenv(\A \rtimes_{\alpha} ^r\G ) &\simeq \cenv(\A) \rtimes_{\alpha}^r \G 
\end{aligned}
\end{equation}
 are indeed true whenever $\A$ ia a Dirichlet algebra and $\G$ is an arbitrary discrete group (Theorems~\ref{Dirichletenvr} and \ref{Dirichletenv}) or $\A$ is arbitrary but $\G$ is abelian (Theorem~\ref{abelianenv}). In this chapter we continue to investigate the validity of such identities. We will show that for a very special class of operator algebras and group actions, the validity of (\ref{envidentity}) is intimately connected with an open problem in $\ca$-algebra theory, \textit{the Hao-Ng isomorphism problem}, which we will describe shortly.
 
 There is a third source of inspiration for the results of this chapter. In Theorem~\ref{crossedtensor} we proved that the crossed product of a tensor algebra of a $\ca$-correspondence with a discrete group may fail to be a tensor algebra. And yet we noticed that for an elliptic M\"obius transformation $\alpha$ of  the unit disc, the crossed product $\bbA(\bbD) \rtimes_{\alpha} \bbZ$ of the disc algebra is isomorphic to the semicrossed product $C(\bbT) \times_{\alpha} \bbZ^+$ and thus a tensor algebra. It turns out that this fact is not just a curiosity but generalizes considerably. As we shall see, the crossed product of \textit{any} tensor algebra by a generalized gauge automorphism is once again a tensor algebra of some other $\ca$-correspondence.
 
 \section{Discrete groups} \label{discrete groups}
 
 Let us set up the framework of study for this chapter and describe the Hao-Ng isomorphism problem. Let $(X, \C)$ be a non-degenerate $\ca$- correspondence over a $\ca$-algebra $\C$ and let $\G$ be a discrete group. Assume that there is a group representation $\alpha: \G \rightarrow \Aut \T_X$ so that $\alpha_{s}(\C)= \C$ and $\alpha_{s}(X)= X$, for all $s \in \G$. We call such an action $\alpha$ a \textit{generalized gauge action} of $\G$ on $(X, \C)$. Clearly the action $\alpha$ restricts to a generalized gauge action $\alpha : \G \rightarrow \Aut \T^+_X$, which in turn extends to a generalized gauge action on $\O_X$. Generalized gauge actions, and in particular the so-called quasi-free actions, have received a great deal of attention in the study of $\ca$-algebras and their states~\cite{EF, Exel, KK, LN}. 

If  $(X, \C)$, $\G$ and $\alpha$ are as above, we define a $\ca$-correspondence $(X\rtimes_{\alpha}^r \G, \C \rtimes_{\alpha}^{r} \G)$ as follows. Identify formal (finite) sums of the form $\sum_s x_sU_s$, $x_s \in X$, $s \in \G$, with their image in $\O_X \cpr$ under $\overline{\pi} \rtimes \lambda_{\H}$, where $\pi$ is a faithful representation of $\O_X$ on $\H$. We call the collection of all such sums $\big(X\rtimes_{\alpha}^r \G\big)_0$. This allows a left and right action on $\big(X\rtimes_{\alpha}^r \G\big)_0$ by $\big(\C \cpr \big)_0$, i.e., finite sums of the form $\sum_{s} c_s U_s \in \C\cpr $, simply by multiplication. The fact that $\alpha$ is a generalized gauge action guarantees that 
\[
\big(\C \cpr \big)_0  \big(X\rtimes_{\alpha}^r \G\big)_0 \big(\C \cpr \big)_0  \subseteq \big(X\rtimes_{\alpha}^r \G\big)_0.
\]
Equip $\big(X\rtimes_{\alpha}^r \G\big)_0$ with the $\big(\C \rtimes_{\alpha}^r \G\big)_0$-valued inner product $\sca{.,.}$ defined by $\sca{S, T} \equiv S^*T$, with $S, T \in \big(X\rtimes_{\alpha}^r \G\big)_0$. The completion of $\big(X\rtimes_{\alpha}^r \G\big)_0$ with respect to the norm coming from $\sca{.,.}$ becomes a $(\C\rtimes_{\alpha}^r \G)$-correspondence denoted as $X\rtimes_{\alpha}^r \G$.  

\begin{theorem}[Hao-Ng Theorem {\cite[Theorem 2.10]{HN}}]   \label{HNThm}
 Let $(X, \C)$ be a non-degenerate $\ca$-correspondence and let $\alpha : \G \rightarrow (X, \C)$ be a generalized gauge action of a discrete\footnote{Note that the Hao-Ng theorem holds for arbitrary locally compact groups. } amenable group $\G$. Then $\O_X\cpr \simeq \O_{X \cpr}$ via a $*$-isomorphism that maps generators to generators. 
\end{theorem}

The reader familiar with the work of Hao-Ng may have noticed that our notation in Theorem~\ref{HNThm} differs from that of Hao and Ng in \cite[Theorem 2.10]{HN}. Indeed Hao and Ng state their result for a different $\ca$-correspondence, which is denoted as $X\cpf$ and it is defined below. As it turns out, in the case where $\G$ is amenable the $\ca$-correspondences $X\cpf$ and $X\cpr$ are unitarily equivalent and so the corresponding Cuntz-Pimsner algebras are isomorphic. See Remark~\ref{HaoNgremarkreferee} (ii). 

The following is a consequence of the Hao-Ng Theorem that demonstrates its significance for our work.

\begin{corollary} \label{HNconseq}
Let $(X, \C)$ be a non-degenerate $\ca$-correspondence and let $\alpha : \G \rightarrow (X, \C)$ be a generalized gauge action of a discrete amenable group $\G$. Then
\[
\T^+_{X} \cpr \simeq \T^+_{X\cpr} \quad \mbox{and} \quad
\cenv\big (\T^+_{X}\cpr \big) \simeq  \O_{X} \cpr .
\]
\end{corollary}

\begin{proof} The conclusion follows directly from Theorem~\ref{HNThm} and Theorem~\ref{r=f}.\end{proof}
 
Beyond amenable groups the two notions of a crossed product differ and we distinguish two cases. For the reduced crossed product, the definition of $(X\rtimes_{\alpha}^r \G, \C \rtimes_{\alpha}^{r} \G)$ has already been given. The situation is not so tame with the full crossed product. In this case we have (at least) three crossed product $\ca$-correspondences 
\begin{itemize}
\item[(i)] \textit{\textbf{The $\ca$-correspondence $X\rtimes_\alpha \G$}}  (\cite{BKQR, HN}). Let  $\big(X\rtimes_{\alpha} \G\big)_0$ denote all formal (finite) sums of the form $\sum_s x_sU_s$, $x_s \in X$, $s \in \G$. This allows a left and right action on $\big(X\rtimes_{\alpha} \G\big)_0$ by $\big(\C \cpf \big)_0$, i.e., finite sums of the form $\sum_{s} c_s U_s $, $c_s \in \C$, simply by allowing the obvious multiplication rules or the ones coming from $\G$-covariance. Equip $\big(X\rtimes_{\alpha} \G\big)_0$ with the $\C \rtimes_{\alpha} \G$-valued inner product $\sca{.,.}$ defined by 
\[
\sca{S, T} \equiv \sum_{s,t} \, \alpha_{s^{-1}}(\sca{x_s,y_{st}})U_{t} \in \C \rtimes_{\alpha} \G,
\]
where $S= \sum_s x_sU_s$ and $T= \sum_{t} y_tU_t$ are both in $\big(X\rtimes_{\alpha} \G\big)_0$. The completion of $\big(X\rtimes_{\alpha} \G\big)_0$ with respect to the inner product $\sca{.,.}$ becomes a $\C\cpf$-correspondence denoted as $X\cpf$.
\item[(ii)] \textit{\textbf{The $\ca$-correspondence $X\check{\rtimes}_\alpha \G$.}} Identify both $\big(X\rtimes_{\alpha} \G\big)_0$ and $\big( \C \cpf \big)_0$ with their natural images inside $\T_X \cpf$. This allows a left and right action on $\big(X\rtimes_{\alpha} \G\big)_0$ by $\big(\C \cpf \big)_0$ simply by multiplication. Equip $\big(X\rtimes_{\alpha} \G\big)_0$ with the $\C \check{\rtimes}_{\alpha} \G$-valued inner product $\sca{.,.}$ defined by $\sca{S, T} \equiv S^*T$, $S, T \in \big(X\rtimes_{\alpha} \G\big)_0$, where $\C \check{\rtimes}_{\alpha} \G$ denotes the $\ca$-subalgebra of $\T_X \cpf$ generated by $\big( \C \cpf \big)_0$. The completion of $\big(X\rtimes_{\alpha} \G\big)_0$ with respect to the inner product $\sca{.,.}$ becomes a $ \C \cpu $-correspondence denoted as $X\check{\rtimes}_{\alpha} \G$.
\item[(iii)] \textit{\textbf{The $\ca$-correspondence $X\hat{\rtimes}_\alpha\G$.}} Identify both $\big(X\rtimes_{\alpha} \G\big)_0$ and $\big( \C \cpf \big)_0$ with their natural images inside $\O_X \cpf$ this time. This allows again a left and right action on $\big(X\rtimes_{\alpha} \G\big)_0$ by $\big(\C \cpf \big)_0$ simply by multiplication. Equip $\big(X\rtimes_{\alpha} \G\big)_0$ with the $\C \hat{\rtimes}_{\alpha} \G$-valued inner product $\sca{.,.}$ defined by $\sca{S, T} \equiv S^*T$, $S, T \in \big(X\rtimes_{\alpha} \G\big)_0$, where $\C \hat{\rtimes}_{\alpha} \G$ denotes the $\ca$-subalgebra of $\O_X \cpf$ generated by $\big( \C \cpf \big)_0$. The completion of $\big(X\rtimes_{\alpha} \G\big)_0$ with respect to the inner product $\sca{.,.}$ becomes a $\C \hat{\rtimes}\G$-correspondence denoted as $X\hat{\rtimes}_{\alpha} \G$.
\end{itemize}

\begin{remark} \label{HaoNgremarkreferee}
(i) The issue with the above definitions is that the algebras $\C \cpf$, $\C \check{\rtimes}_{\alpha} \G$ and $\C \hat{\rtimes}_{\alpha} \G$ might not be isomorphic. It is not even clear that there is an inclusion $\C \cpf  \subseteq \O_X \cpf$, something that would be implied if for instance $\C \cpd \simeq \C \cpf$ canonically. Indeed even in the case of the trivial action, such an inclusion would translate to 
\[
\C \otimes_{max} \ca(\G) \subseteq \O_X \otimes_{max} \ca(\G),
\]
an inclusion that hinges on the validity of WEP for $\C$. (See \cite[Corollary 3.6.8]{BO}.) Nevertheless, as we shall see in Remark~\ref{quigg}, the correspondences $X\rtimes_\alpha \G$ and $X\check{\rtimes}_\alpha \G$ are unitarily equivalent via an association that sends generators to generators. We are thankful to the authors of \cite{BKQR} for pointing this out to us. 

(ii) In the case where $\G$ is amenable, it is not difficult to see that all crossed product $\ca$-correspondences associated with the generalized gauge action $\alpha: \G \rightarrow (X, \C)$ are unitarily equivalent. We verify this for the $\ca$-correspondences $(X \cpr, \C\cpr)$ and $(X\cpf, \C \cpf)$.  Let $(\bar{\rho}_{\infty}, \overline{t}_{\infty})$ be the universal covariant representation of $(X, \C)$. Given formal sums $S= \sum_s x_sU_s$ and $T= \sum_{t} y_tU_t$, we have 
\[\begin{aligned}
(\C\cpr )_0 \ni \sca{S, T} &=  (\overline{\pi} \rtimes \lambda_{\H} )\big(\sum_s\, \overline{t}_{\infty}(x_s)U_s\big)^*(\overline{\pi} \rtimes \lambda_{\H} )\big(\sum_t\,\overline{t}_{\infty}(y_t)U_t\big)\\		
          		&= \sum_{s,t} \overline{\pi}\left(\alpha_{s^{-1}}\left(\overline{t}_{\infty}(x_s)^*\overline{t}_{\infty}(y_t)\right)\right)\lambda_{\H}(s^{-1}t) \\
		&= \sum_{s,t} \overline{\pi}\big(\alpha_{s^{-1}}\big( \bar{\rho}(\sca{x_s,y_t}\big)\lambda_{\H}(s^{-1}t)\\ 
		 &=(\overline{\pi} \rtimes \lambda_{\H} )\big(\sum_{s,t} \, \alpha_{s^{-1}}(\sca{x_s,y_t})U_{s^{-1}t} \big)\\
		 &=(\overline{\pi} \rtimes \lambda_{\H} )\big(\sum_{s,t} \, \alpha_{s^{-1}}(\sca{x_s,y_{st}})U_{t} \big)\\
		&=(\overline{\pi} \rtimes \lambda_{\H} )(\sca{S,T}),
		\end{aligned}\]
where the last inner product comes from $X \cpf$. Since $\G$ is amenable, $(\overline{\pi} \rtimes \lambda_{\H} )$ is a faithful representation of $\C \cpf$. Hence the above calculation shows that the mappings
\[
\begin{aligned}
(X \cpr)_0&\ni \sum_{s}\bar{t}_{\infty}(x_s)U_s \xrightarrow{\phantom{x}W\phantom{x}} \sum_{s} x_sU_s \in ( X \cpf)_0\\
(\C \cpr)_0&\ni \sum_{t}\overline{\rho}_{\infty}(c_t)U_t \xrightarrow{\phantom{xx}\sigma\phantom{x}}\sum_{t} c_tU_t \in ( \C\cpf)_0
\end{aligned}
\]
can be extended to $X\cpr$ and $\C\cpr$ respectively, so that the pair $(W, \sigma)$ implements the desired unitary equivalence between $(X \cpr, \C\cpr)$ and $(X\cpf, \C \cpf)$.
\end{remark}

The \textit{Hao-Ng isomorphism problem}, as popularized in \cite{BKQR, KQR, Katsuracomm, Kim}, asks whether given a non-degenerate $\ca$-correspondence $(X, \C)$ and a generalized gauge action of a discrete group $\G$, one has isomorphisms of the form $\O_X\cpf \simeq \O_{X \cpf}$ or $\O_X\cpr\simeq \O_{X \cpr}$. The analysis in this chapter indicates that in addition to the correspondence $X\rtimes_\alpha \G$, we should also pay attention to the correspondence $X\check{\rtimes}_\alpha \G$. As it turns out, a recasting of the Hao-Ng isomorphism problem using the correspondence $X\check{\rtimes}_\alpha \G$ is equivalent to resolving the identity (\ref{introident}) in that special case.

For the moment we demonstrate a result of independent interest, a tool for detecting whether a given operator algebra is completely isometrically isomorphic to the tensor algebra of some naturally occurring $\ca$- correspondence. We call this result the \textit{Extension Theorem}.  But first we need a lemma. 

\begin{lemma} \label{normincrease}
 Let $S_0, S_1, S_2, \dots S_n$ be bounded operators on a Hilbert space $\H$ and let $V$ be the forward shift on $l^2(\bbN )$. Then,
\[
\| \sum _{k = 0}^{n} \,  S_k \| \leq \| \sum _{k = 0}^{n} \,  S_k \otimes V^k\|
\]
\end{lemma}

\begin{proof}
Consider the character $\delta_1$ on $\ca(V)$ which is obtained by taking quotient on $C(\bbT)$ and then evaluating at $1$. This induces a $*$-homomorphism 
\[
\id \otimes \delta_1: \ca(S)\otimes \ca(V) \longrightarrow \ca(S).
\]
The conclusion follows by applying $\id \otimes \delta_1$ on $ \sum _{k = 0}^{n} \,  S_k  \otimes V^k$.
\end{proof}

In what follows, if $\S \subseteq B(\H)$, then $\alg(\S)$ will denote the (not necessarily unital) algebra generated by $\S$, while $\overline{\alg}(\S)$ will denote its norm closure.

\begin{theorem}[Extension Theorem] \label{extend}
Let $\C \subseteq B(\H)$ be a $\ca$-algebra and let $X \subseteq  B(\H)$ be a closed $\C$-bimodule with $X^*X \subseteq \C$. If $\A \equiv \overline{\alg}(X\cup \C)$ and $U$ denotes the forward shift acting on $l^2 (\bbZ)$, then the following are equivalent
\begin{itemize}
\item[(i)] $\A$ is completely isometrically isomorphic to the tensor algebra $\T_{(X, \C) }^+$ via a map that sends generators to generators.
\item[(ii)] The association
\begin{equation} \label{umapplus}
\begin{aligned}
\C \ni C &\longrightarrow C \otimes I, \\
X \ni S &\longrightarrow S \otimes U
\end{aligned}
\end{equation}
 extends to a well-defined, completely contractive multiplicative map on $\A$.
\end{itemize}
\end{theorem}

\begin{proof}
We will be showing that condition (ii) above is equivalent to
\begin{itemize}
\item[(iii)] The association
\begin{equation} \label{vmapplus}
\begin{aligned}
\C \ni C &\longrightarrow C \otimes I, \\
X \ni S &\longrightarrow S \otimes V
\end{aligned}
\end{equation}
 extends to a well-defined, completely contractive multiplicative map on $\A$, where $V$ denotes the forward shift acting on $\ell^2(\bbN)$.
 \end{itemize}
 In order to establish the equivalence of (ii) and (iii) we need to verify
\begin{equation} \label{standardeq}
\| \sum_{k=0}^{n} S_k \otimes U^k\| = \| \sum_{k=0}^{n} S_k \otimes V^k\|,
\end{equation}
where $S_0, S_1, \dots S_n$ ranges over arbitrary elements of $B(\H)$.

Assume that $U$ acts on $l^2(\bbZ)$, with orthonormal basis $\{e_n\}_{n \in \bbZ}$, let $P_m$ be the orthogonal projection on the subspace generated by $\{e_n\}_{n=m}^{\infty}$ and let $V= P_0 U P_0$. Clearly,
\[ \big\| \sum_{k=0}^{n} S_k \otimes U^k\big\| = \sup_{m \in \bbN} \left\{  \Big\| \big(  \sum_{k=0}^{n}  S_k \otimes U^k\big)\mid_{ I \otimes P_m}\Big\| \right\}.
\]
However
\begin{align*}
 \Big\| \big(  \sum_{k=0}^{n}  S_k \otimes U^k\big)\mid_{ I \otimes P_m}\Big\| &=
 \Big\| \big(  \sum_{k=0}^{n}  S_k \otimes U^kU^mU^{-m} \big)\mid_{ I \otimes P_m}\Big\| \\
&=
 \Big\| (I\otimes U^{m}) \big(  \sum_{k=0}^{n}  S_k \otimes U^k U^{-m} \big)\mid_{ I \otimes P_m}\Big\| \\
&=
 \Big\|  \big(  \sum_{k=0}^{n}  S_k \otimes U^k  \big)(I \otimes U^{-m})\mid_{ I \otimes P_m}\Big\| \\
&=
 \Big\|  \big(  \sum_{k=0}^{n}  S_k \otimes U^k  \big)\mid_{ I \otimes P_0}\Big\| \\
&= \| \sum_{k=0}^{n} S_k \otimes V^k\|
\end{align*}
as desired. An analogous argument establishes the matricial version of (\ref{standardeq}), thus establishing the equivalence of (ii) and (iii).

In order to complete the proof, we need to establish the equivalence of (i) and (iii). 

Let $(\pi, t)$ be the representation of the $\ca$-correspondence $ (X, \C)$, with $\pi(C) = C \otimes I$, $C \in \C $ and $t(X) = S \otimes V$, $S \in X$. It is easy to see that the presence of the factor $V$ guarantees that the representation $(\pi, t)$ admits a gauge action. Furthermore, $(\pi, t)$ satisfies (\ref{Icriterion}) and so by the Gauge-Invariant Uniqueness Theorem (Theorem~\ref{GIUThm}) it extends to a faithful representation $\Phi$ of the Toeplitz-Cuntz-Pimsner algebra $\T_{X}$. We therefore obtain a completely isometric representation $\Phi$ of the tensor algebra $\T_{X}^+$ satisfying $\Phi(t_{\infty}(S))=S\otimes V$, $S \in X$ and $\Phi(\rho_{\infty}(C)) =C \otimes I$, $ C \in \C$. 

If (i) is valid and $\Psi: \A \rightarrow \T^+_{X}$ is the completely isometric isomorphism that maps generators to generators, then $\Phi\circ \Psi$ is the map that implements (\ref{vmapplus}). On the other hand, if (iii) is valid and (\ref{vmapplus}) is implemented by a completely contractive homomorphism $\Psi$, then (the matricial version of) Lemma~\ref{normincrease} implies that $\Psi$ is a complete isometry and $\Psi^{-1}\circ \Phi$ implements the isomorphism desired in (i)
\end{proof} 

\begin{remark} \label{extensionremark}
If in Theorem~\ref{extend} one makes the assumption that $\C$ contains a contractive approximate unit for $X$, then the correspondence $(X, \C)$ turns out to be non-degenerate.
\end{remark} 

We now examine the full crossed product $\ca$-algebras $\O_X \cpf$ and $\T_X \cpf$ and we consider the non-selfadjoint operator algebras $\T^+_X \rtimes_{ \O_X, \alpha} \G$ and $\T^+_X \rtimes_{ \T_X, \alpha} \G$ sitting inside them. Are any of these two algebras the tensor algebra of some $\ca$-correspondence? What are their $\ca$-envelopes? Are these relative crossed products isomorphic? The following provides answers to these questions.

\begin{theorem} \label{HNtensor1}
Let $(X, \C)$ be a non-degenerate $\ca$-correspondence and let $\alpha: \G \rightarrow (X, \C)$ be the generalized gauge action of a discrete group $\G$. Then
\begin{itemize}
\vspace{.07in}
\item[(i)] $\T^+_{X} \rtimes_{\O_X ,\alpha} \G \simeq \T^+_{X\cpd}\quad \mbox{  and  } \quad \cenv\big (\T^+_{X } \rtimes_{\O_X ,\alpha} \G \big) \simeq  \O_{X\cpd}$
\vspace{.1in}
\item[(ii)]  $\T^+_{X} \rtimes_{\T_X ,\alpha} \G \simeq \T^+_{X \cpu} \quad \mbox{ and } \quad \cenv\big (\T^+_{X} \rtimes_{\T_X ,\alpha} \G \big) \simeq  \O_{X\cpu}$
\vspace{.1in}
\end{itemize}
\vspace{.03in}
\end{theorem}

\begin{proof} 
(i) Let $(\pi_{\infty}, \phist_{\infty}, \H)$ be the universal covariant representation of $(\O_X, \G, \alpha)$ and let $U$ be the forward shift acting on $l^2(\bbZ)$. Any representation of $\O_X$ is the integrated representation of some covariant representation of $(X, \C)$; this applies in particular to $\pi_{\infty}$ and so  
\begin{align*}
\C\ni &c \longmapsto \pi_{\infty}(c) \in B(\H_{\infty}) \\
X\ni &x \longmapsto \pi_{\infty}(x)   \in B(\H_{\infty})
\end{align*}
is a covariant representation of $(X, \C)$. Hence 
\begin{align*}
\C\ni &c \longmapsto \pi_{\infty}(c) \otimes I \in B(\H_{\infty}\otimes l^2(\bbZ)) \\
X\ni &x \longmapsto \pi_{\infty}(x) \otimes U  \in B(\H_{\infty}\otimes l^2(\bbZ))
\end{align*}
is also a covariant representation of $(X, \C)$ and therefore integrates to a representation of $\O_X$ denoted as $\pi$. Set $\phist(s)=\phist_{\infty}(s)\otimes I$, $s \in \G$, and notice that the triple $(\pi, \phist, \H_{\infty}\otimes l^2(\bbZ))$ is a covariant representation for the system $(\O_X, \G, \alpha)$. Therefore it integrates to a completely contractive $*$-representation 
\[
\pi \rtimes \phist \, : \,O_X \cpf \longrightarrow  B(\H_{\infty}\otimes l^2(\bbZ)).
\]
Consider now the $\ca$-correspondence $X \cpd$ as defined in the beginning of the chapter, with the understanding that formal (finite) sums of the form $\sum_s x_s U_s \in \big(X \cpf\big)_0$ are identified with their images inside $\O_X \cpf$ under the map $\pi_{\infty} \rtimes \phist_{\infty}$. First notice that 
\[
(X \cpd )^*(X \cpd) \subseteq \C\hat{ \rtimes}_{\alpha}^{} \G.
\]
Furthermore the identities
\[
\Big( \sum_{s} \pi_{\infty}(c_s)\phist_{\infty}(s)\Big) \otimes I = (\pi\rtimes \phist) \big(\sum_s c_s U_s \big), \quad c_s \in \C, s \in \G
\]
and 
\[
\Big( \sum_{s} \pi_{\infty}(x_s)\phist_{\infty}(s)\Big) \otimes U = (\pi\rtimes \phist) \big(\sum_s x_s U_s \big), \quad x_s \in X, s \in \G
\]
show that the map 
\begin{align*}
 \C\hat{ \rtimes}_{\alpha}^{} \G \ni  \sum_s c_s U_s  & \longmapsto \Big(\sum_s c_s U_s \Big) \otimes I \\
X \cpd \ni  \sum_s x_s U_s  & \longmapsto \Big(\sum_s x_s U_s \Big) \otimes U
\end{align*}
extends to a completely contractive map on $\alg \big(X \cpd \cup \C\hat{ \rtimes}_{\alpha}^{} \G \big)$. Hence by Theorem~\ref{extend} (Extension Theorem) we have that 
\[
\T^+_{X} \rtimes_{\O_X ,\alpha} \G = \overline{\alg} \big(X \cpd \cup \C\hat{ \rtimes}_{\alpha}^{} \G \big)\simeq \T^+_{X\cpd}
\]
as desired.

The identification $\cenv\big (\T^+_{X } \rtimes_{\O_X ,\alpha} \G \big) \simeq  \O_{X\cpd}$ follows from \cite[Theorem 3.7]{KatsoulisKribsJFA}.

(ii) We just repeat the proof of part (i) modified accordingly. Actually here the proof can be made more elementary as we do not really need to use the forward shift $U \in l^2(\bbZ)$. Instead we can use the forward shift $V \in l^2(\bbN)$. In that case, the Extension Theorem is just a straightforward application of Theorem~\ref{GIUThm}.  

Indeed let $(\pi_{\infty}, \phist_{\infty}, \H)$ be the universal covariant representation of $(\T_X, \G, \alpha)$ and let $V$ be the forward shift acting on $l^2(\bbN)$. The representation
\begin{align*}
\C\ni &c \longmapsto \pi_{\infty}(c) \otimes I \in B(\H_{\infty}\otimes l^2(\bbN)) \\
X\ni &x \longmapsto \pi_{\infty}(x) \otimes V  \in B(\H_{\infty}\otimes l^2(\bbN))
\end{align*}
is also a Toeplitz representation of $(X, \C)$ and therefore integrates to a representation of $\T_X$ denoted as $\pi$. Set $\phist(s)=\phist_{\infty}(s)\otimes I$, $s \in \G$, and notice that the triple $(\pi, \phist, \H_{\infty}\otimes l^2(\bbN))$ is a covariant representation for the system $(\T_X, \G, \alpha)$. Therefore it integrates to a completely contractive $*$-representation $\pi \rtimes \phist \, : \T_X \cpf \rightarrow  B(\H_{\infty}\otimes l^2(\bbN))$. Using $\pi \rtimes \phist$ we can show as before that the assignment 
\begin{align*}
 \C\check{ \rtimes}_{\alpha}^{} \G \ni  \sum_s c_s U_s  & \longmapsto \Big(\sum_s c_s U_s \Big) \otimes I \\
X \cpu \ni  \sum_s x_s U_s  & \longmapsto \Big(\sum_s x_s U_s \Big) \otimes V
\end{align*}
extends to a completely contractive map on $\alg \big(X \cpu \cup \C\hat{ \rtimes}_{\alpha}^{} \G \big)$. Hence by Theorem~\ref{extend} (Extension Theorem), we have that 
\[
\T^+_{X} \rtimes_{\T_X ,\alpha} \G = \overline{\alg} \big(X \cpu \cup \C\check{ \rtimes}_{\alpha}^{} \G \big)\simeq \T^+_{X\cpu}
\]
as desired.

The identification $\cenv\big (\T^+_{X } \rtimes_{\T_X ,\alpha} \G \big) \simeq  \O_{X\cpu}$ follows once again from \cite[Theorem 3.7]{KatsoulisKribsJFA}.
\end{proof}

\begin{remark} \label{quigg}
It turns out that Theorem~\ref{HNtensor1} (ii) can be refined even further. Indeed, in \cite[Theorem 3.1]{BKQR} it is shown that $\T_X\cpf \simeq \T_{X\cpf}$ via a $*$-isomorphism that maps generators to generators. This implies that $\C \cpf \simeq \C \cpu$ canonically and so the correspondences $X\rtimes_\alpha \G$ and $X\check{\rtimes}_\alpha \G$ are unitarily equivalent via an association that sends generators to generators. Hence one can recast Theorem~\ref{HNtensor1} (ii) as 
\[
\T^+_{X} \rtimes_ {\T_X, \alpha} \G \simeq \T^+_{X \rtimes_{\alpha}\G} \quad \mbox{ and } \quad \cenv(\T^+_{X} \rtimes_{\T_X, \alpha} \G ) \simeq \O_{X \rtimes_{\alpha}\G}.
\]
\end{remark}

The previous result shows that the problem of deciding whether all relative full crossed products are isomorphic seems to be a delicate issue. In this particular case, the presence of an isomorphism between $\T^+_{X} \rtimes_{\T_X ,\alpha} \G$ and $\T^+_{X} \rtimes_{\O_X ,\alpha} \G$ is equivalent to the isomorphism between the tensor algebras $ \T^+_{X \rtimes_{\alpha}\G}$ and $\T^+_{X \cpd}$. Currently there are no criteria for verifying an isomorphism between tensor algebras. The standing conjecture is that the obvious sufficient condition, i.e., unitary equivalence of the corresponding correspondences, is also necessary for the existence of an isomorphism.

In light of Theorem~\ref{HNtensor1}, we offer the following modified version of the Hao-Ng isomorphism problem

\vspace{.1in}

\noindent \textbf{Hao-Ng Isomorphism Conjecture for full crossed products}. \textit{Let $(X, \C)$ be a non-degenerate $\ca$-correspondence and let $\alpha: \G \rightarrow (X, \C)$ be the generalized gauge action of a discrete group $\G$. Then
\[
\O_X \cpf   \simeq \O_{X\cpd} \simeq\O_{X\rtimes_{\alpha}\G}
\]
}
Note that if (\ref{introident}) was valid for the relative crossed product $\T^+_{X} \rtimes_{\O_X ,\alpha} \G$, i.e., $$\cenv\big(\T^+_{X} \rtimes_{\O_X ,\alpha} \G \big) \simeq \cenv( \T^+_{X} )\rtimes_{\alpha} \G  \simeq \O_X \cpf,$$ then Theorem~\ref{HNtensor1}(i) would imply the first half of the Hao-Ng isomorphism conjecture. The other half of the conjecture would follow from a similar argument involving (\ref{introident})  for the relative crossed product $\T^+_X \rtimes_{\T_X, \alpha}\G$, Theorem~\ref{HNtensor1}(ii) and \cite[Theorem 3.1]{BKQR}. However the validity of (\ref{introident}) and its variants is one of the main problems left open in this monograph. Nevertheless, in the case of a Hilbert bimodule $X$ or an abelian group $\G$, it turns out that this is the case; see the end of this chapter for more on this.

One can also formulate an analogue of the Hao-Ng isomorphism conjecture for Toeplitz algebras. As we explained earlier, the validity of the analogous conjecture $$\T_X\cpf \simeq \T_{X \check{\times}_{\alpha}\G}\simeq \T_{X\cpf}$$ has already been established in \cite[Theorem 3.1]{BKQR} .

Now we deal with the reduced crossed product and wonder whether $\T^+_X \cpr$ is a tensor algebra, provided that $\alpha$ is a generalized gauge action of $\G$. Unfortunately the strategy of the proof of Theorem~\ref{HNtensor1} does not work here as it is not clear whether the representation $\pi \rtimes \phist$ appearing in the proof can be modified to give a representation of the \textit{reduced} crossed product $\O_X \cpr$. Instead we adopt a different approach.

\begin{theorem} \label{HNtensor2}
Let $(X, \C)$ be a non-degenerate $\ca$-correspondence and let $\alpha: \G \rightarrow (X, \C)$ be a generalized gauge action of a discrete group $\G$. Then
\[
\T^+_{X} \cpr \simeq \T^+_{X\cpr}.
\]
Therefore, $$\cenv\big (\T^+_{X} \cpr \big) \simeq  \O_{X\cpr}.$$
Furthermore, $\T_{X} \cpr \simeq \T_{X\cpr}$.
\end{theorem}

\begin{proof} Because of Corollary~\ref{allrcoincide} all relative reduced crossed products coincide and so we have flexibility in choosing which manifestation of $\T^+_X\cpr$ to work with. We choose the ``natural" one $\T^+_X \rtimes^r_{\T_X, \alpha}\G \subseteq \T_X \cpr$.

Notice that the $\ca$-algebra $\T_X$ contains a unitarily equivalent copy of the $\ca$-correspondence $(X, \C)$ and for the rest of the proof we envision  $(X, \C)$ as a subset of $\T_X$. Similarly the $\ca$-algebra $\T_X \cpr$ contains a (unitarily equivalent) copy of $(X\rtimes_{\alpha}^r \G, \C \rtimes_{\alpha}^{r} \G)$. Indeed $\T_X \cpr$ contains naturally a faithful copy of $\C \cpr$ and so the map
\[
\O_X \cpr \supseteq (X\rtimes_{\alpha}^r \G)_0 \ni \sum_g x_gu_g \longmapsto  \sum_g x_g u_g \in \T_X \cpr
\]
extends to a unitary equivalence of $\ca$-correspondences that embeds $(X\rtimes_{\alpha}^r \G, \C \rtimes_{\alpha}^{r} \G)$ inside $\T_X \cpr$.

Let $\rho \colon \T_X \rightarrow B(\H)$ be some faithful $*$-representation and let $V$ be the forward shift acting on $l^2(\bbN)$. The map
\begin{align*}
\C\ni &c \longmapsto \rho(c) \otimes I \in B(\H \otimes l^2(\bbN)) \\
X\ni &x \longmapsto \rho(x) \otimes V  \in B(\H \otimes l^2(\bbN))
\end{align*}
is a Toeplitz representation of $(X, \C)$ that admits a gauge action and and satisfies the requirements of Theorem~\ref{GIUThm}. Therefore it establishes a faithful representation $\pi : \T_X \rightarrow B(\H \otimes l^2(\bbN))$.

Now view the regular representation $\overline{\pi} \rtimes \lambda_{\H \otimes l^2(\bbN) }$ as a representation of the $\ca$-correspondence $(X\rtimes_{\alpha}^r \G, \C \rtimes_{\alpha}^{r} \G)$. Since
\begin{align*}
\big( \C \rtimes_{\alpha}^{r} \G \big)_0 \ni \sum_{g} c_{g}u_{g} \ \ \longmapsto \ \ &(\overline{\pi} \rtimes \lambda )\big( \sum_{g} c_{g}u_{g} \big) \\
&=\sum_g\sum_h \rho\big(\alpha_h^{-1}(c_g) \big)\otimes I \otimes e_{h, g^{-1}h} \\
\big( X \rtimes_{\alpha}^{r} \G \big)_0 \ni \sum_{g} x_{g}u_{g} \ \ \longmapsto \ \ &(\overline{\pi} \rtimes \lambda )\big( \sum_{g} x_{g}u_{g} \big) \\
&=\sum_g\sum_h \rho\big(\alpha_h^{-1}(x_g) \big)\otimes V \otimes e_{h, g^{-1}h} ,
\end{align*}
($e_{p,q}$ denotes the rank-one isometry on $l^2(\G)$ that maps $\xi_q$ on $\xi_p$, $p,q\in \G$) the above extends to an isometric representation of $(X\rtimes_{\alpha}^r \G, \C \rtimes_{\alpha}^{r} \G)$ that admits a gauge action (because of the middle factor $V$) and satisfies the requirements of Theorem~\ref{GIUThm}. Hence its integrated form is a canonical faithful representation of the Toeplitz-Cuntz-Pimsner algebra $\T_{X\cpr}$. In other words, if $(\pi_{\infty}, t_{\infty})$ is the universal Toeplitz representation of $(X\rtimes_{\alpha}^r \G, \C \rtimes_{\alpha}^{r} \G)$, then there exists $*$-isomorphism  
\[
\phi \colon \ca(\pi_{\infty}, t_{\infty}) \longrightarrow \overline{\pi}\rtimes \lambda\big( \T_X \cpr\big)
\]
satisfying 
\[
\phi \Big( \pi_{\infty}\big( \sum_{g} c_{g}u_{g} \big)\Big) = (\overline{\pi} \rtimes \lambda )\big( \sum_{g} c_{g}u_{g} \big) , \mbox{ for all }\sum_{g} c_{g}u_{g} 
 \in\big( \C \rtimes_{\alpha}^{r} \G \big)_0 
 \]
 and 
 \[
\phi \Big( t_{\infty}\big( \sum_{g} x_{g}u_{g} \big)\Big) =  (\overline{\pi} \rtimes \lambda )\big( \sum_{g} x_{g}u_{g} \big), \mbox{ for all }\sum_{g} x_{g}u_{g} 
 \in\big( X \rtimes_{\alpha}^{r} \G \big)_0 
 \]
Since $\pi$ is faithful, $\overline{\pi}\rtimes \lambda$ is a faithful representation of $\T_{ X}\cpr$ and so $(\overline{\pi} \rtimes \lambda)^{-1}\circ \phi$ establishes a $*$-isomorphism from $\ca(\pi_{\infty}, t_{\infty}) \simeq \T_{ X\cpr}$ onto $\T_X\cpr$ that maps $\T^+_{X\cpr}$ onto $\T^+_{X} \cpr$ in a canonical way. Hence $\T^+_{X} \cpr \simeq \T^+_{X\cpr}$.

Finally the isomorphism $\cenv\big (\T^+_{X} \cpr \big) \simeq  \O_{X\cpr}$ follows from
\cite[Theorem 3.7]{KatsoulisKribsJFA}, which implies the identification $\cenv(\T^+_{X \cpr}) \simeq \O_{X \cpr}$.
\end{proof}

Let us import yet another result from the $\ca$-algebra theory and use it to our advantage.

\begin{corollary} \label{exact}
Let $(X, \C)$ be a non-degenerate $\ca$-correspondence and let $\alpha: \G \rightarrow \Aut \O_X$ be a generalized gauge action of a discrete and exact group $\G$. Then
\[
\cenv \big(\T^+_{X} \rtimes_{\alpha}^r \G \big) \simeq  \O_{X} \cpr.
\]
\end{corollary}

\begin{proof} This follows directly from \cite[Theorem 5.5 (i)]{BKQR}.
\end{proof}

\section{The general case of a locally compact group.} \label{general case} In Section~\ref{discrete groups} we focused our attention on discrete groups for two reasons. First, the prerequisites for understanding our theory are not as many as in the general case of a locally compact group. If someone is just interested in using the crossed product in order to obtain new examples of tensor algebras, then this chapter gives an easy access. One can actually read all previous results in Chapter~\ref{Hao} with only minimal understanding of the previous chapters. On the other hand, one of the major open problems in this area, the Hao-Ng isomorphism problem, is open even for discrete groups with all its difficulties present even in that special case. 

Nevertheless, with the exception of Corollary~\ref{exact}, all previous results in Chapter~\ref{discrete groups} hold for arbitrary locally compact groups. In what follows we demonstrate how to obtain one such result, Theorem~\ref{HNtensor1}, in the generality of a locally compact group. 

We start by defining the correspondence $(X\cpd, \C \cpd)$. Let $(X, \C)$ be a non-degenerate $\ca$-correspondence and let $(\bar{\rho}_{\infty}, \overline{t}_{\infty})$ be the universal covariant representation of $(X, \C)$, acting on some Hilbert space $\H_{\infty}$. Let $C \cpd$ be the completion of $C_c\big(\G,  \bar{\rho}_{\infty}(\C)\big) \subseteq \O_X\cpf$ and similarly let $X \cpd$ be the completion of $C_c\big(\G,  \bar{t}_{\infty}(X)\big) \subseteq \O_X\cpf$.

\begin{lemma} \label{tensorprop}
If $\C \cpd$ and $X \cpd$ are as above, then
\vspace{.025in}
\begin{itemize}
\item[(i)] $(X \cpd)^*(X \cpd) \subseteq \C \cpd$
\vspace{.05in}
\item[(ii)] $(\C \cpd) (X \cpd )(\C \cpd )\subseteq  X \cpd$.
\end{itemize}
\end{lemma}

\begin{proof}
If $x, y \in \bar{t}_{\infty}(X)$ and $z , w \in C_c(\G)$, then 
 \begin{equation*}
 \begin{aligned}
&\big((z\otimes x)^*  (w\otimes y )\big)(s)&&=\int \Delta(r^{-1})\overline{z(r^{-1})}\alpha_{r}(x^*)w(r^{-1}s)\alpha_{r}(y)d\mu(r) &\\
 & &&=\int \Delta(r^{-1})\overline{z(r^{-1})}\alpha_{r}(x^*y)w(r^{-1}s)d\mu(r).&
     \end{aligned}
 \end{equation*}
 However, 
 \[
 x^*y \in \big(\bar{t}_{\infty}(X)\big)^*\big(\bar{t}_{\infty}(X) \big)\subseteq \bar{\rho}_{\infty}(\C)
 \]
 and so 
 \[
 (z\otimes x)^*  (w\otimes y ) \in C_c\big(\G , \bar{\rho}_{\infty}(\C)\big) \subseteq \C \cpd.
 \]
 Since elementary tensors are dense in $X \cpd$, this proves (i).
 
 For (ii), let $c \in \bar{\pi}_{\infty}(\C)$, $x \in \bar{t}_{\infty}(X)$ and $z , w \in C_c(\G)$. Then,
  \begin{equation*}
 \begin{aligned}
&\big((z\otimes c)  (w\otimes x )\big)(s)&&=\int z(r)c\alpha_r\big( w(r^{-1}s)x\big)d\mu(r) &\\
 & &&=\int z(r) w(r^{-1}s)\alpha_r\big(\alpha_r^{-1}(c)x\big)d\mu(r) &
     \end{aligned}
 \end{equation*}
  However $\G$ acts by gauge automorphisms and so 
 \[
\alpha_r^{-1}(c)x \in  \bar{\pi}_{\infty}(\C) \bar{t}_{\infty}(X)\subseteq  \bar{t}_{\infty}\big( \phi_X(\C)X \big)  \subseteq \bar{t}_{\infty}(X).
 \]
Hence $(\C \cpd) (X \cpd ) \subseteq X \cpd$ and similarly $(X \cpd )(\C \cpd )\subseteq X \cpd$. This establishes (ii).
 \end{proof}
 
Allow $\C \cpd$ to act on the left and right of $X \cpd$  simply by multiplication. Then Lemma~\ref{tensorprop} shows that $X\cpd$ equipped with that action and the $\C \cpd$-valued inner product $\sca{\cdot,\cdot}$ defined by $\sca{S, T} \equiv S^*T$, $S, T \in X \cpd$, becomes a $\ca$-correspondence over $\C \cpd$.

\begin{lemma} \label{algrelation}
Let $(X, \C)$ be a non-degenerate $\ca$-correspondence and let $(X\cpd, \C \cpd)$ be as above. Then 
\[
\overline{\alg}\big( X\cpd, \C \cpd \big) = \T^+_{X} \rtimes_{\O_X ,\alpha} \G.
\]
\end{lemma}

\begin{proof}
Let $z \in C_c (\G)$ and $a \in \T^+_X$. If $a = c +\sum_{n=1}^{\infty}   x_n $
with $c \in \bar{\rho}_{\infty}(\C)$ and $x_n \in \bar{t}(X^{\otimes n})$, $n \in \bbN$, then we have
\begin{equation} \label{tensors}
z\otimes a = z \otimes c +\sum_{n=1}^{\infty} \, z\otimes x_n.
\end{equation}
Since elementary tensors are dense in $C_c\big(\G, \T_X^+\big)$, it suffices by (\ref{tensors}) to prove that 
\[
z \otimes x \in \overline{\alg}\big( X\cpd, \C \cpd \big) 
\]
for any $z \in C_c(\G)$ and $x \in \bar{t}(X^{\otimes n})$, $ n \in \bbN$. 

We will show this by induction. The case $n=1$ is obvious. 
Assume that the result is true for all $k \leq n-1$. Let $x= x'y \in \bar{t}(X^{\otimes n})$ with $x' \in \bar{t}(X)$ and $y \in \bar{t}(X^{\otimes n-1})$.

\vspace{.12in}
\noindent \textit{Claim:} If $\{w_i \}_{i \in \bbI}$ are as in Lemma~\ref{cai}, then
\[
z \otimes x = \lim_{i \in \bbI}\, (w_i \otimes x')(z \otimes y).
\]
\vspace{.1in}
\noindent
Indeed, let $i : \O_X \rightarrow M(\O_X \cpf)$ be as in \cite[Proposition 2.34]{Will}. Then,
\begin{equation} \label{prelimit}
(w_i \otimes x')(z \otimes y)= i(x')\big((w_i \otimes I)(z \otimes y) \big).
\end{equation}
However, by Lemma~\ref{cai}, the net $\{w_i \otimes I \}_{i \in \bbI}$ is a contractive approximate identity. Hence by taking limits in (\ref{prelimit}) we obtain 
\[
\lim_{i \in \bbI}\, (w_i \otimes x')(z \otimes y)= i(x')(z \otimes y)=z \otimes x'y= z\otimes x
\]
as desired.

\vspace{.1in}

The claim and the inductive hypothesis show now that $$z\otimes x \in \overline{\alg}\big( X\cpd, \C \cpd \big)$$ and the proof of the lemma is complete.
\end{proof}

\begin{theorem} \label{HNtensor3}
Let $(X, \C)$ be a non-degenerate $\ca$-correspondence and let $\alpha: \G \rightarrow (X, \C)$ be the generalized gauge action of a locally compact group $\G$. Then
\[
\T^+_{X} \rtimes_{\O_X ,\alpha} \G \simeq \T^+_{X\cpd}\quad \mbox{  and  } \quad \cenv\big (\T^+_{X } \rtimes_{\O_X ,\alpha} \G \big) \simeq  \O_{X\cpd}\]
\end{theorem}

\begin{proof}
If $(\bar{\rho}_{\infty}, \overline{t}_{\infty})$ is the universal covariant representation of $(X, \C)$, then the representation
\begin{align*}
\bar{\rho}_{\infty}(\C) \ni & c \longmapsto c\otimes I  \in B\big(\H_{\infty} \otimes \ell^2(\bbZ)\big) \\
\bar{t}_{\infty} (X) \ni &\,x \longmapsto   x \otimes U  \in B\big(\H_{\infty} \otimes \ell^2(\bbZ)\big) ,
\end{align*}
is also covariant, where $U$ denotes the forward shift on $\ell^2(\bbZ)$. Therefore it integrates to a $*$-representation $\pi: \O_X \rightarrow \O_X \otimes C(\bbT)$. Clearly $\pi$ is equivariant with respect to the dynamical systems $(\O_X, \G, \alpha)$ and $\big(\O_X\otimes C(\bbT), \G, \alpha\otimes \id\big)$. Therefore, \cite[Corollary 2.48]{Will} implies the existence of a $*$-homomorphism 
\[
\pi\rtimes \id : \O_X \cpf \longrightarrow \big(\O_X \otimes C(\bbT)\big) \rtimes_{\alpha\otimes\id}\G
\]
satisfying $\pi\rtimes \id (f)(s) = \pi \big(f(s)\big)$, $s \in \G$, for all $f \in C_c(\G, \O_X)$. By \cite[Corollary 2.75]{Will}  there exists a $*$-isomorphism 
\[
\phi : \big(\O_X \otimes C(\bbT)\big) \rtimes_{\alpha\otimes\id}\G \longrightarrow \big(\O_X  \cpf \big) \otimes C(\bbT)
\]
which carries $z\otimes (a\otimes d)\mapsto (z\otimes a)\otimes d$, with $a \in O_X$, $d \in C(\bbT)$ and $z \in C_c(\G)$. Hence, the completely contractive mapping $\phi\circ (\pi\rtimes \id)$ implements the assignment
\begin{align*} 
C_c\big(\G, \bar{\rho}_{\infty}(\C)\big) \ni&\, z \otimes c  \longmapsto (z \otimes c) \otimes I \\
C_c\big(\G, \bar{t}_{\infty} (X)\big) \ni &\, z \otimes  x \longmapsto  (z \otimes c) \otimes U.  
\end{align*}
This implies that the requirements of the Extension Theorem are satisfied for the $\C \cpd$-bimodule $X\cpd$. Hence $$\overline{\alg}\big( X\cpd, \C \cpd \big) \simeq \T_{X \cpd}^+ .$$ The conclusion follows now from Lemma~\ref{algrelation}.
\end{proof}

A similar approach works for $(X \cpu, \C, \cpu)$. This $\ca$-correspondence is built with the aid of the universal Toeplitz representation $(\rho_{\infty}, t_{\infty})$.  We define $C \cpu$ to be the completion of $C_c\big(\G,  {\rho}_{\infty}(\C)\big) \subseteq \T_X\cpf$ and similarly we let $X \cpu$ to be the completion of $C_c\big(\G,  {t}_{\infty}(X)\big) \subseteq \T_X\cpf$. By repeating our previous arguments, we obtain the other half of Theorem~\ref{HNtensor1}, i.e., 
 $$\T^+_{X} \rtimes_{\T_X ,\alpha} \G \simeq \T^+_{X \cpu} \quad \mbox{ and } \quad \cenv\big (\T^+_{X} \rtimes_{\T_X ,\alpha} \G \big) \simeq  \O_{X\cpu}$$
As we mentioned in Remark~\ref{quigg}, the $\ca$-correspondences $(X \cpu, \C \cpu)$ and $(X \cpf, \C\cpf)$ are unitarily equivalent via a canonical map. However it is not clear to us whether or not $(X \cpd, \C \cpd)$ and the $\ca$- correspondence $(X \cpf, \C, \cpf)$, as defined in \cite[pg. 1082]{BKQR}, are unitarily equivalent. This issue is resolved affirmatively by the Hao-Ng Theorem in the case where $\G$ is amenable. In what follows we verify this in another important case by offering a resolution to the Hao-Ng isomorphism problem in that case as well.

\section{Hilbert $\ca$-bimodules}
A $\ca$-correspondence $(X, \C, \phi_X)$ is said to be a Hilbert $\C$-bimodule, if there exists a right $\C$-valued inner product $[\cdot,\cdot]$ which satisfies 
\[
\phi_X\big( [\xi, \zeta]\big) \eta = \xi \sca{\zeta, \eta} , \quad \mbox{for all } \xi, \zeta, \eta \in X.
\]
There are many useful characterizations of Hilbert bimodules. For instance, by \cite[Proposition 5.18 (iii)]{KatsuraJFA} $(X, \C, \phi_X)$ is a Hilbert $\C$-bimodule iff the restriction  of $\phi_X$ on $J_X$ maps onto $\K(X)$.

The following settles the Hao-Ng conjecture for Hilbert bimodules.

\begin{theorem} \label{HNBim}
Let $(X, \C)$ be a non-degenerate Hilbert bimodule and let $\alpha: \G \rightarrow (X, \C)$ be the generalized gauge action of a locally compact group $\G$. Then 
\[
\O_X \cpf   \simeq \O_{X\cpd} \simeq \O_{X\cpf}.
\]
\end{theorem}

\begin{proof}
Kakariadis has proven~\cite[Theorem 2.2]{Kak} that a $\ca$-correspondence $(X, C)$ is a Hilbert bimodule iff the tensor algebra $\T^+_X$ is Dirichlet. Therefore we can apply Theorem~\ref{Dirichletenv} and Theorem~\ref{HNtensor3} to show that 
\[
\O_X\cpf \simeq \cenv(\T_X^+)\cpf \simeq \cenv(\T_X^+ \rtimes_{\O_X , \alpha} \G) \simeq O_{X \cpd}
\]
as desired.
It remains to verify that $O_{X\cpd} \simeq \O_{X\cpf}$. Let $\Phi$ be the conditional expectation appearing in (\ref{condexpO}). Since $(X, \C)$ is a Hilbert bimodule, $\Phi$ projects onto $\C$ \cite[Proposition 5.18 (i)]{KatsuraJFA}. Furthermore $\Phi$ commutes with $\alpha$. Hence the requirements of \cite[Section~\ref{intro}0, Proposition]{Combes} or \cite{Itoh} are satisfied and so $\C \cpf \simeq \C \cpd$ via a map that sends generators to generators. This completes the proof.
\end{proof}

It is instructive to recast Theorem~\ref{HNBim} in the language of Abadie~\cite{Ab}.

\begin{corollary} \label{nice}
Let $(\beta , \gamma)$ be a covariant action of a locally compact group $\G$ on a Hilbert $\C$-bimodule $X$. If $\alpha$ is the strongly continuous action of $\G$ on $\C \rtimes X$ induced by $(\beta , \gamma)$, then $(\C \rtimes X)\rtimes_{\alpha} \G \simeq (\C \rtimes_{\beta} \G)\rtimes ( X \rtimes_{\gamma} \G)$.
\end{corollary}

Abadie's~\cite{Ab} ``covariant pair" and its ``induced strongly continuous action" constitute the same framework of study as the "generalized gauge action of a locally compact group" of this monograph.  What Abadie defines as $\C \rtimes X$ is isomorphic to the Cuntz-Pimsner algebra $\O_X$ and so the above corollary is indeed a recasting of Theorem~\ref{HNBim}. 

Corollary~\ref{nice} was obtained by Abadie as Proposition 4.5 but only in the case where $\G$ is amenable. It is a technical result with a rather long proof. Hao and Ng \cite{HN} considered Abadie's result as a motivating force for their theory. They gave a very short proof of it \cite[Corollary 2.12]{HN} as an application of their theory, but again, only in the case where $\G$ is amenable. It is quite pleasing to see that our ``non-selfadjoint" approach removes the requirement of $\G$ being amenable from all previous considerations.  

In \cite{HN}, Hao and Ng give a second application of their theorem, this time involving the generalized gauge action of an abelian group $\G$. Actually using the results of this monograph, we can give an alternative proof of the Hao-Ng Theorem for the case where $\G$ is abelian. Indeed combining Theorem~\ref{abelianenv} and \cite[Theorem 3.7]{KatsoulisKribsJFA} we obtain 
\[
\O_X\cpf \simeq \cenv(\T_X^+)\cpf \simeq \cenv(\T_X^+ \rtimes_{\alpha} \G).
\]
However the amenability of $\G$ and Theorem~\ref{HNtensor3} imply
\[
\cenv(\T_X^+ \rtimes_{\alpha} \G) \simeq \cenv(\T_X^+ \rtimes_{\O_X, \alpha} \G)  \simeq \O_{X \cpd} \simeq \O_{X \cpf}
\]
as desired. It is worth mentioning that even the case $\G = \bbT$ of the Hao-Ng Theorem is being used in current research.


\chapter{Concluding Remarks and Open Problems} \label{problems}

We close the monograph with a brief discussion of various open problems that have appeared throughout and we consider them important for the further development of the theory.

\begin{problem} \label{identity}
If $(\A, \G , \alpha)$ is a dynamical system, then verify the identity
\[
\cenv\big(\A \cpf\big)\simeq \cenv(\A)\cpf.
\]
\end{problem}
Without any doubt this is the most important problem left open in the monograph. At the end of the previous chapter we indicated that a positive resolution of Problem \ref{identity} and its relative crossed product variants will also imply a positive resolution of the Hao-Ng isomorphism problem.  We have verified Problem~\ref{identity} in the case where $\G$ is a locally compact abelian group (Theorem~\ref{abelianenv})  and in the case where $\A$ is Dirichlet (Theorem \ref{Dirichletenvr}).

\begin{problem} \label{relative}
Give an example of a dynamical system $(\A, \G , \alpha)$ and two $\alpha$-admissible $\ca$-covers $(C_i, j_i)$ for $\A$, $j=1,2$, so that 
\[
\A \rtimes_{\C_1, j_1, \alpha} \G \not\simeq \A \rtimes_{\C_2, j_2, \alpha} \G 
\]
\end{problem}

Theorem~\ref{r=f} shows that for such a (counter)example, $\G$ will have to be non-amenable. This problem also relates to the various crossed product $\ca$-correspondences appearing in Chapter~\ref{Hao} and our recasting of the Hao-Ng isomorphism problem.

\begin{problem} \label{TXrelative}
Let $(X, \C)$ be a non-degenerate $\ca$-correspondence and let $\alpha: \G \rightarrow (X, \C)$ be the generalized gauge action of a locally compact group. Is $\T_X^+ \cpf $ the tensor algebra of some $\ca$-correspondence?
\end{problem}

In Chapter~\ref{Hao} we did not deal with the full crossed product $\T^+_X \cpf$ as it is not relevant to the Hao-Ng isomorphism problem. Nevertheless it is important to know the answer. Note that this problem too is open only for non-amenable groups. If Problem~\ref{relative} has a negative answer, i.e., all relative full crossed products are isomorphic, then Theorem~\ref{HNtensor3} will imply a positive answer for this problem.

\begin{problem} \label{probsemis}
If $\A$ is semisimple does it follow that $\A \rtimes_{\alpha} \bbR$ is also semisimple? What about the converse?
\end{problem}

This problem is motivated by Theorems~\ref{firstsemisimple} and \ref{Tsemis} which treat the cases where $\G$ is either discrete and abelian or compact and abelian respectively. What about other groups? It would also be interesting to have a characterization of semisimplicity for algebras of the form $\A \cpf$ where $\A$ is a strongly maximal TAF algebra and $\G=\bbT$ or $\bbR$.

\begin{problem} \label{diagprob}
Characterize the diagonal for  either $\A \cpf$ or $\A \cpr$.
\end{problem}

Of course the ``right" answer is that the diagonal of $\A \cpf$ is $(\A \cap\A^*)\cpf$, while the diagonal of $\A \cpr$ is $(\A \cap\A^*)\cpr$. Theorem~\ref{diagonal descr} verifies that in the case where $\G$ is a discrete amenable group. Also algebras of the form $\A \cpf$ or $\A \cpr$ that happen to be tensor algebras for some correspondence $(X, \C)$ have diagonal equal to $\C$. So we can characterize the diagonal of the crossed products appearing in Chapter~\ref{Hao}. We know nothing beyond these two cases.

\begin{problem}
When are two algebras of the form $A(\bbD) \rtimes_{\alpha} \bbZ$ isomorphic as algebras?
\end{problem}

Of course there is nothing special about the disc algebra $A(\bbD)$ but this seems to be the simplest case of the isomorphism problem for non-selfadjoint crossed products and yet we know very little even in that special case. Note that if $\alpha$ is an elliptic M\"obius automorphism of the disc, then $A(\bbD) \rtimes_{\alpha} \bbZ \simeq C(\bbT) \times_{\alpha}\bbZ^+$ and so the theory of Davidson and Katsoulis \cite{DavKatCr} applies. 

\begin{problem}
Give complete isomorphism invariants for algebras of the form $\A \rtimes_{\alpha} \bbZ$, where $\A$ is a strongly maximal TAF algebra and $\alpha$ an isometric automorphism.
\end{problem}

The TAF algebras have been classified up to isometric isomorphism through the use of the the fundamental groupoid. (See \cite{Pow} and the references therein.) We wonder whether one can develop an analogous theory for crossed products of such algebras. There is nothing special for $\G = \bbZ$; a broader theory would be welcome as well.

\vspace{0.1in}

{\noindent}{\it Acknowledgement.} EK has benefited from many discussions through the years that inspired him and guided him in fleshing out some of the ideas appearing in this monograph. In that respect, he is particularly grateful to Mihalis Anoussis, Aristedes Katavolos and Justin Peters for their insight and  patient listening. EK is also grateful to the second named author and the University of Virginia for inviting him to visit during the preparation of this work and the subsequent hospitality. 

Certain parts of this monograph were presented by EK to an audience during his visit at the Chinese Academy of Science in the summer of 2016. EK would like to thank both his hosts Professors Liming Ge and Wei Yuan for the stimulating conversations and their hospitality during his stay.

Both authors are grateful to the anonymous referee, whose comments and corrections helped improve the quality of the monograph. 

\vspace{0.1in}

{\noindent}{\it Note added in proof.} In a recent paper \cite{KatsIMRN} the first named author has resolved Problem \ref{identity} for all discrete amenable groups by verifying the identity $\cenv(\A \cpr) \simeq \cenv(\A)\cpr$, for any discrete group $\G$. This has also consequences for the Hao-Ng isomorphism problem. On the other hand, Harris and Kim in \cite{HK} present counterexamples which show that the identity of Problem \ref{identity} may fail for discrete groups which are not amenable. In the same paper, Harris and Kim also answer Problem~\ref{relative}. 

In \cite{KatsRamJFA}, the authors of the present paper continue their investigation on Takai duality and its applications. In particular, they show that any semicrossed product of an operator algebra by an abelian ordered group is stably isomorphic to a non-selfadjoint crossed product. This leads to more examples of non-semisimple operator algebras which give semisimple crossed products. In \cite{KatsRamJFA} it is also shown that Problem~\ref{probsemis} has a negative answer. In the same paper, the authors partly resolve Problem~\ref{diagprob} by characterizing the diagonal of $\A\cpf$ in the case where $\G$ is a compact abelian group.

In \cite{KatsRamHaoNg} the authors of the present paper continue the investigation of the Hao-Ng isomorphism problem and its connections with non-selfadjoint crossed products. Among others, they provide a positive answer to Problem~\ref{TXrelative}. In earlier versions of this monograph it was claimed that the validity of (\ref{introident}) for crossed products of tensor algebras by gauge automorphisms is equivalent to the positive resolution of the Hao-Ng isomorphism problem. Even though in this monograph it is established that (\ref{introident}) and its variants imply the Hao-Ng isomorphism, the converse requires more work. The case of hyperrigid $\ca$-correspondences is treated in \cite{KatsRamHaoNg}.

\backmatter
\bibliographystyle{amsplain}


\end{document}